%% file: main.tex
\documentclass[12pt]{amsart}

\allowdisplaybreaks
\usepackage{graphicx} 
\usepackage{scrextend}
\usepackage{amsfonts, amsmath, amssymb,amsthm,comment}  
\usepackage{times,enumerate}
\usepackage{mathdots}%
\usepackage{nccmath}
\usepackage{mathtools}
\usepackage{enumitem}
\usepackage{ulem}
\usepackage{multicol}
    {\end{pmatrix}\end{medsize}}%
\usepackage{dsfont,bbm}
\usepackage{rotating}
\usepackage{lscape}
\usepackage{url}
\usepackage{multicol}
\usepackage{thmtools, thm-restate}
\usepackage{blkarray}
\usepackage{multicol}
\usepackage[a4paper,margin=2.5cm]{geometry}
 \usepackage{hyperref}
\usepackage{cancel}
\usepackage{multirow}
\usepackage{pifont}
\usepackage{tikz}
\usetikzlibrary{topaths}
\tikzstyle{every picture}+=[remember picture,inner xsep=0,inner ysep=0.25ex]
\usetikzlibrary{calc}
\usepackage{kbordermatrix}
\renewcommand{\kbldelim}{(}
\renewcommand{\kbrdelim}{)}
\renewcommand{\kbrowstyle}{\displaystyle}
\renewcommand{\kbcolstyle}{\displaystyle}
\def\VR{\kern-\arraycolsep\strut\vrule &\kern-\arraycolsep}
\def\vr{\kern-\arraycolsep & \kern-\arraycolsep}
\makeatletter
\newcommand*{\sublabel}[1]{%
    \let\old@currentlabel\@currentlabel%
    \renewcommand{\@currentlabel}{\theenumii}%
    \label{#1}%
    \let\@currentlabel\old@currentlabel%
}
\makeatother
\allowdisplaybreaks

\graphicspath{ {./figs/} }

\newcommand*\interior[1]{\mathring{#1}}

\hypersetup{
    colorlinks,
    linkcolor={blue},
    citecolor={blue},
    urlcolor={blue}
}

\DeclareMathOperator{\Span}{span}

\DeclareMathOperator{\Rank}{rank}

\DeclareMathOperator{\Card}{card}
\DeclareMathOperator{\sign}{sign}

\makeatletter
\def\widebreve{\mathpalette\wide@breve}
\def\wide@breve#1#2{\sbox\z@{$#1#2$}%
     \mathop{\vbox{\m@th\ialign{##\crcr
\kern0.08em\brevefill#1{0.8\wd\z@}\crcr\noalign{\nointerlineskip}%
                    $\hss#1#2\hss$\crcr}}}\limits}
\def\brevefill#1#2{$\m@th\sbox\tw@{$#1($}%
  \hss\resizebox{#2}{\wd\tw@}{\rotatebox[origin=c]{90}{\upshape(}}\hss$}
\makeatletter

\newcommand{\RR}{\mathbb R}

\newcommand{\NN}{\mathbb N}
\newcommand{\ZZ}{\mathbb Z}

\newcommand{\cT}{\mathcal T}

\newcommand{\cF}{\mathcal F}
\newcommand{\cB}{\mathcal B}

\newcommand{\cI}{\mathcal I}

\newcommand{\cM}{\mathcal M}

\newcommand{\cN}{\mathcal N}

\newcommand{\cG}{\mathcal G}

\newcommand{\cA}{\mathcal A}

\newcommand{\cH}{\mathcal H}

\newcommand{\cZ}{\mathcal Z}

\newcommand{\cC}{\mathcal C}

\newcommand{\benu}{\begin{enumerate}}
\newcommand{\eenu}{\end{enumerate}}
\newcommand{\bop}{\begin{opomba}}
\newcommand{\eop}{\end{opomba}}

\newcommand{\supp}{\mathrm{supp}}

\newtheorem{theorem}{Theorem}[section]
\newtheorem{corollary}[theorem]{Corollary}
\newtheorem{lemma}[theorem]{Lemma}
\newtheorem{proposition}[theorem]{Proposition}

\theoremstyle{definition}

\newcommand{\mc}{\mathcal}
\newcommand{\mbb}{\mathbb}

\definecolor{green-new}{rgb}{0.0, 0.5, 0.0}

\definecolor{cyan}{rgb}{0.0, 0.8, 1.0}

\usepackage{setspace}

\newtheorem{remark}[theorem]{Remark}

\numberwithin{equation}{section}

\begin{document}

\title[TMP on reducible cubic curves II]{Constructive approach to the truncated moment problem on reducible cubic curves: 
Hyperbolic type relations}

\author[S. Yoo]{Seonguk Yoo}
\address{Department of Mathematics Education and RINS, Gyeongsang National University, Jinju, 52828, Korea. }
\email{seyoo@gnu.ac.kr}

\author[A. Zalar]{Alja\v z Zalar${}^{2}$}
\address{Alja\v z Zalar, 
Faculty of Computer and Information Science, University of Ljubljana  \& 
Faculty of Mathematics and Physics, University of Ljubljana  \&
Institute of Mathematics, Physics and Mechanics, Ljubljana, Slovenia.}
\email{aljaz.zalar@fri.uni-lj.si}
\thanks{${}^2$Supported by the ARIS (Slovenian Research and Innovation Agency)
research core funding No. P1-0288 and grants No.\ J1-50002, J1-60011.}


\date{\today}
\maketitle

\begin{abstract}
In this paper, we solve \textit{constructively} the bivariate truncated moment problem (TMP) of even degree on reducible cubic curves, where the conic part is a hyperbola.
 According to the classification from our previous work \cite{YZ24},
 these represent three out of nine possible canonical forms of reducible cubic curves 
 after applying an affine linear transformation.
 The TMP on the union of three parallel lines, the circular and the parabolic type TMP were solved constructively in \cite{Zal22a,YZ24}, while in this paper
 we consider three cases 
 of hyperbolic type, 
 i.e.,  a type without real self-intersection points,
 a type with a simple real self-intersection point
 and
 a type with a double real self-intersection point.
 In all cases, we also establish bounds on the number of atoms in a minimal representing measure.
\end{abstract}


\input{Introduction}

\input{Preliminaries}

\input{Common-approach}

\input{Hyperbolic-type-1-new}

\input{Hyperbolic-type-2-new}

\input{Hyperbolic-type-3-new}

\end{document}

%% file: Introduction.tex
\section{
    Introduction
    }

Let $\ZZ_+$ stand for nonnegative integers. 
    Given a real $2$--dimensional sequence
	$$\beta\equiv\beta^{(2k)}=\{\beta_{0,0},\beta_{1,0},\beta_{0,1},\ldots,\beta_{2k,0},\beta_{2k-1,1},\ldots,
		\beta_{1,2k-1},\beta_{0,2k}\}$$
of degree $2k$ and a closed subset $K$ of $\RR^2$, the \textbf{truncated moment problem ($K$--TMP)} supported on $K$ for $\beta^{(2k)}$
asks to characterize the existence of a positive Borel measure $\mu$ on $\RR^2$ with support in $K$, such that
	\begin{equation}
		\label{moment-measure-cond}
			\beta_{i,j}=\int_{K}x^iy^j d\mu\quad \text{for}\quad i,j\in \ZZ_+,\; 
            i+j\leq 2k.
	\end{equation}
If such a measure exists, we say that $\beta^{(2k)}$ has a representing measure 
	supported on $K$ and $\mu$ is its $K$--\textbf{representing measure ($K$--rm).}

In the degree-lexicographic order $\textit 1,X,Y,X^2,XY,Y^2,\ldots,X^k,X^{k-1}Y,\ldots,Y^k$ of rows and columns, the corresponding moment matrix to $\beta$
is equal to 
	\begin{equation}	
	\label{281021-1448}
            \mc M(k)=
		\mc M(k;\beta):=
		\left(\begin{array}{cccc}
		\mc M[0,0](\beta) & \mc M[0,1](\beta) & \cdots & \mc M[0,k](\beta)\\
		\mc M[1,0](\beta) & \mc M[1,1](\beta) & \cdots & \mc M[1,k](\beta)\\
		\vdots & \vdots & \ddots & \vdots\\
		\mc M[k,0](\beta) & \mc M[k,1](\beta) & \cdots & \mc M[k,k](\beta)
		\end{array}\right),
	\end{equation}
where
	$$\mc M[i,j](\beta):=
		\left(\begin{array}{ccccc}
		\beta_{i+j,0} & \beta_{i+j-1,1} & \beta_{i+j-2,2} & \cdots & \beta_{i,j}\\
		\beta_{i+j-1,1} & \beta_{i+j-2,2} & \beta_{i+j-3,3} & \cdots & \beta_{i-1,j+1}\\
		\beta_{i+j-2,2} & \beta_{i+j-3,3} & \beta_{i+j-4,4} & \cdots & \beta_{i-2,j+2}\\
		\vdots & \vdots & \vdots & \ddots &\vdots\\
		\beta_{j,i} & \beta_{j-1,i+1} & \beta_{j-2,i+2} & \cdots & \beta_{0,i+j}\\
		\end{array}\right).$$
Let 
$\RR[x,y]_{\leq k}:=\{p\in \RR[x,y]\colon \deg p\leq k\}$ 
stand for the set of real polynomials in variables $x,y$ of total degree at most $k$.
For every $p(x,y)=\sum_{i,j} a_{ij}x^iy^j\in \RR[x,y]_{\leq k}$ we define
its \textbf{evaluation} $p(X,Y)$ on the columns of the matrix $\mc M(k)$ by replacing each capitalized monomial $X^iY^j$
in $p(X,Y)=\sum_{i,j} a_{ij}X^iY^j$ by the column of $\mc M(k)$, indexed by this monomial.
Then $p(X,Y)$ is a vector from the linear span of the columns of $\mc M(k)$. If this vector is the zero one, i.e., all coordinates are equal to 0, 
then $p(X,Y)$ is a \textbf{column relation} of $\mc M(k)$.
A column relation $p(X,Y)$ is \textbf{nontrivial}, if 
$p\not\equiv 0$.
The matrix $\mc M(k)$ is \textbf{recursively generated (rg)} if for $p,q,pq\in \RR[x,y]_{\leq k}$ such that $p(X,Y)$ is a column relation of $\mc M(k)$, 
it follows that $(pq)(X,Y)$ is also a column relation of $\mc M(k)$.
The matrix $\mc M(k)$ is \textbf{$p$--pure} if the only column relation of $\mc M(k)$ are those determined recursively by $p$. In this case the TMP for $\beta$ is called \textbf{$p$--pure}.

For $p\in \RR[x,y]$ we denote by
$\cZ(p):=\{(x,y)\in \RR^2\colon p(x,y)=0\}$
the zero set of $p$ and by $\deg p$ its total degree.

A \textbf{concrete solution} to the TMP is a set of necessary and sufficient conditions for the existence of a $K$--representing measure $\mu$, 
that can be tested in numerical examples. 
Among necessary conditions, $\mc M(k)$ must be positive semidefinite (psd) and rg \cite{CF04,Fia95}, and by \cite{CF96} if the support $\supp(\mu)$ of $\mu$ is a subset of 
$\cZ(p)$ for a polynomial $p\in \RR[x,y]_{\leq k}$, 
then $p$ is a column relation of $\mc M(k)$.
The bivariate $\cZ(p)$--TMP (\textit{not necessarily $p$--pure}) is concretely solved in the following cases: (i) $\deg p= 1$ \cite{CF08},
(ii) $\deg p=2$ \cite{CF02,CF04,CF05,Fia15}, 
(iii) $p$ is irreducible with $\deg p=3$ \cite{KZ25+}
and (iv) $p$ is reducible, $\deg p=3$ and $p$ has a special form \cite{Zal22a,YZ24}. 
The bivariate \textit{$p$--pure} TMP is concretely solved also for: 
(v) $p(x,y)=xy+q(x)-x^4$ with $\deg q=3$ \cite{YZ24+}, 
(vi) $p(x,y)=y-x^4$ \cite{FZ25+}
and (vii) $p$ is reducible with $\deg p=3$ and $\beta$ is 
\textit{purely pure} (i.e., the corresponding linear functional
is strictly positive on nonzero polynomials, positive on $\cZ(p)$) \cite{KZ25+}. 
For a more detailed description and some other less concrete solutions to the TMP
on plane algebraic curves we refer the reader to \cite[p.~ 3]{YZ24} or \cite[p.\ 2--3]{KZ25+},
while for a recent development in the area of moment problems to a monograph \cite{Sch17}.

A \textbf{constructive solution} to the $K$--TMP is a solution, where not only the existence of a $K$--rm is characterized, but a concrete $K$--rm is explicitly constructed.

The motivation for this paper was to solve the TMP \textit{constructively} on reducible cubic curves of hyperbolic type, according to the classification of \cite[Proposition 3.1]{YZ24}.
By applying an affine linear transformation, each TMP on reducible cubic curve is equivalent to the TMP on one of nine canonical cases of the form $yc(x,y)=0$, where $c\in \RR[x,y]$, $\deg c=2$.
In \cite{Zal22a}, the case of three parallel lines is solved constructively, while in \cite{YZ24}, the solutions to 
 the \textit{circular type} (the curve is a line and a circle touching at a double real point)
and the \textit{parabolic type} relations (the curve is a line and a parabola that intersect tangentially at a real point) are presented.
In this paper, we solve the TMP constructively for the cases
$c(x,y)=1-xy$, $c(x,y)=x+y-xy$ and $c(x,y)=ay+x^2-y^2$, $a\in \RR\setminus \{0\}$,
which are called in \cite{YZ24} the \textit{hyperbolic type $1$, $2$} and $3$ $\textit{relations}$, respectively.
We also characterize the number of atoms in a minimal representing measure, i.e., a measure with the minimal number of atoms in the support. The question of bounds on the cardinality of minimal representing measures in the TMP, supported on algebraic curves, which is always finite by \cite{Ric57} (or \cite[Theorem 1.24]{Sch17}), has attracted a recent attention of several authors (see \cite{RS18,dDS18,dDK21,Zal24,BBS24+,RTT25+}).

In terms of the self-intersection points of the cubic $yc(x,y)=0$, we can classify the hyperbolic types from the previous paragraph into a type without real self-intersection points (type 1),
a type with a single real self-intersection point (type 2)
and a type with a double real self-intersection point (type 3).
To prove our main results, we follow the idea presented in \cite{Zal22a,YZ24},
which characterizes the existence of a decomposition of $\beta$ into the sum $\beta^{(\ell)}+\beta^{(c)}$,
where $\beta^{(\ell)}=\{\beta_{i,j}^{(\ell)}\}_{i,j\in \ZZ_+,\; i+j\leq 2k}$ and $\beta^{(c)}=\{\beta_{i,j}^{(c)}\}_{i,j\in \ZZ_+,\; 
i+j\leq 2k}$
admit a $\RR$--rm and a $\cZ(c)$--rm, respectively.
The crucial property of the forms of the cubic, which makes this idea realizable, is that the line is equal to $y=0$. This ensures that all but two moments of $\beta^{(\ell)}$ and $\beta^{(c)}$ are not already determined by the original sequence, i.e.,
 $\beta_{2k-1,0}^{(\ell)}$,
 $\beta_{2k,0}^{(\ell)}$,
 $\beta_{2k-1,0}^{(c)}$,
 $\beta_{2k,0}^{(c)}$
in the hyperbolic type 1 case
(as in the case of three parallel lines \cite{Zal22a})
,
 $\beta_{0,0}^{(\ell)}$,
 $\beta_{2k,0}^{(\ell)}$,
 $\beta_{0,0}^{(c)}$,
 $\beta_{2k,0}^{(c)}$
in the hyperbolic type 2 case
(as in the parabolic type case \cite[Section 6]{YZ24})
and
 $\beta_{0,0}^{(\ell)}$,
 $\beta_{1,0}^{(\ell)}$,
 $\beta_{0,0}^{(c)}$,
 $\beta_{1,0}^{(c)}$
in the hyperbolic type 3 case
(as in the circular type case \cite[Section 5]{YZ24}).
Then, by an involved analysis, the characterization of the existence of a decomposition $\beta=\beta^{(\ell)}+\beta^{(c)}$ can be done in all three cases.
We mention that the analysis in the hyperbolic type cases is more demanding than in the corresponding cases with the same positions of the free moments from \cite{Zal22a,YZ24} stated in parentheses, since the solution to the TMP on a hyperbola (see Subsection \ref{subsection:hyperbolic-TMP}) contains more linear algebraic requirements than in the case of other conics.


\subsection{ Reader's Guide}
The paper is organized as follows. 
In Section \ref{preliminiaries} we fix notation and present some preliminary results 
needed to establish our main results.
In Section \ref{Section-common-approach} we recall the approach for solving the TMP constructively on reducible cubic curves in the canonical form $yc(x,y)=0$ developed in \cite[Section 4]{YZ24}.
In Sections \ref{Hyperbolic-type-1}--\ref{sec:hyperbolic-type-3} we solve constructively  the TMP for reducible cubic
curves of hyperbolic types 1--3, respectively, and characterize the cardinality of minimal representing measures
(see Theorems \ref{100622-2204}, \ref{sol:x-y-axy}, 
\ref{thm:hyperbolic-2-minimal-measures},
\ref{221023-1854} and
\ref{thm:hyperbolic-3-minimal-measures}).
Numerical examples demonstrating the main results are also given 
(see Subsections
\ref{ex:hyperbolic-type-1}, \ref{ex:hyperbolic-type-2} and \ref{ex:hyperbolic-type-3}).

%% file: Preliminaries.tex
\section{Preliminaries}
\label{preliminiaries}

We write $\RR^{n\times m}$ for the set of $n\times m$ real matrices. For a matrix $M$ 
we call the linear span of its columns a \textbf{column space} and denote it by $\cC(M)$.
The set of real symmetric matrices of size $n$ will be denoted by $S_n$. 
For a matrix $A\in S_n$ the notation $A\succ 0$ (resp.\ $A\succeq 0$) means $A$ is positive definite (pd) (resp.\ positive semidefinite (psd)).
We write $\mathbf{0}_{t_1,t_2}$ for a $t_1\times t_2$ matrix with only zero entries and 
$\mathbf{0}_{t}=\mathbf{0}_{t,t}$ for short, where $t_1,t_2,t\in \NN$.
The notation
    $E^{(\ell)}_{i,j}$, 
    $\ell\in \NN$,
    stands for the usual $\ell\times \ell$ coordinate matrix with the only nonzero entry at position $(i,j)$, equal to 1.

In the rest of this section let $k\in \NN$ and $\beta=\beta^{ (2k)}=\{\beta_{i,j}\}_{i,j\in \ZZ_+,\; i+j\leq 2k}$ be a bivariate sequence of degree $2k$.

\subsection{Moment matrix}
Let $\mc M(k)$ be the moment matrix of $\beta$ (see \eqref{281021-1448}).
Let $Q_1, Q_2$ be subsets of the set $\{X^iY^j\colon  i,j \in \ZZ_+,\; i+j\leq k\}$.
We denote by 
$\mc M(k)_{Q_1,Q_2}$ 
the submatrix of $\mc M(k)$ consisting of the rows indexed by the elements of $Q_1$
and the columns indexed by the elements of $Q_2$. In case $Q:=Q_1=Q_2$, we write 
$\mc M(k)_{Q}:=\mc M(k)_{Q,Q}$
for short. 

\subsection{Affine linear transformations} 
\label{sub:affine-linear-trans}

The existence of representing measures is invariant under invertible affine linear transformations of the form 
\begin{equation}
\label{alt}
    \phi(x,y)=(\phi_1(x,y),\phi_2(x,y)):=(a+bx+cy,d+ex+fy),\; (x,y)\in \RR^{2},
\end{equation}
$a,b,c,d,e,f\in \RR$ with $bf-ce \neq 0$.
Namely, let  $L_{\beta}:\mbb{R}[x,y]_{\leq 2k}\to \RR$ be a \textbf{Riesz functional} of the sequence $\beta$ defined by 
$$
	L_{\beta}(p):=\sum_{\substack{i,j\in \ZZ_+,\\ i+j\leq 2k}} a_{i,j}\beta_{i,j},\qquad \text{where}\quad p=
	\sum_{\substack{i,j\in \ZZ_+,\\ i+j\leq 2k}} a_{i,j}x^iy^j.
$$
We define $\widetilde \beta=\{\widetilde \beta_{i,j}\}_{i,j\in \ZZ_+,\; i+j\leq 2k}$ by
	$$\widetilde \beta_{i,j}=L_{\beta}(\phi_1(x,y)^i \cdot \phi_2(x,y)^j).$$
By \cite[Proposition 1.9]{CF04}, $\beta$ admits a $r$--atomic rm supported on $K$ if and only if $\widetilde \beta$ admits a $r$--atomic rm supported on $\phi(K)$.
We write $\widetilde \beta=\phi(\beta)$ and $\mc M(k;\widetilde \beta)=\phi(\mc M(k;\beta))$.

\subsection{Generalized Schur complements}\label{SubS2.1}
Let 
	\begin{equation*}
		M=\left( \begin{array}{cc} A & B \\ C & D \end{array}\right)\in \RR^{(n+m)\times (n+m)}
	\end{equation*}
be a real matrix where $A\in \RR^{n\times n}$, $B\in \RR^{n\times m}$, $C\in \RR^{m\times n}$  and $D\in \RR^{m\times m}$.
The \textbf{generalized Schur complement} \cite{Zha05} of $A$ (resp.\ $D$) in $M$ is defined by
	$$M/A=D-CA^\dagger B\quad(\text{resp.}\; M/D=A-BD^\dagger C),$$
where $A^\dagger $ (resp.\ $D^\dagger $) stands for the Moore-Penrose inverse of $A$ (resp.\ $D$). 

The following lemma will be used in the proofs of our main results.

\begin{lemma}
	\label{140722-1055} 
	Let $n,m\in \NN$ and
		\begin{equation*}
			M=\left( \begin{array}{cc} A & B \\ B^{T} & C\end{array}\right)\in S_{n+m},
		\end{equation*} 
	where $A\in S_n$, $B\in \RR^{n\times m}$ and $C\in S_m$.
	If $\Rank M=\Rank A$, then the matrix equation
	\begin{equation}\label{140722-1055-eq}
		\begin{pmatrix}
			A\\
			B^T
		\end{pmatrix}
		W
		=
		\begin{pmatrix}
			B\\
			C
		\end{pmatrix},
	\end{equation}
	where $W\in \RR^{n\times m}$, is solvable and 
	the solutions are precisely the solutions of the matrix equation $AW=B$.
	In particular, $W=A^{\dagger}B$ satisfies \eqref{140722-1055-eq}.
\end{lemma}

The following theorem is a characterization of psd $2\times 2$ block matrices. 

\begin{theorem}[{\cite{Alb69}}]
    \label{block-psd} 
	Let 
		\begin{equation*}
			M=\left( \begin{array}{cc} A & B \\ B^{T} & C\end{array}\right)\in S_{n+m}
		\end{equation*} 
	be a real symmetric matrix where $A\in S_n$, $B\in \RR^{n\times m}$ and $C\in S_m$.
	Then: 
	\begin{enumerate}
		\item 
                \label{021123-1702}
                The following conditions are equivalent:
			\begin{enumerate}
			     \item 
                    \label{pt1-281021-2128} 
                        $M\succeq 0$.
                    \smallskip
			     \item 
                    \label{pt2-281021-2128} 
                        $C\succeq 0$,         
                    $\cC(B^T)\subseteq\cC(C)$ and $M/C\succeq 0$.
                    \smallskip
				 \item 
                    \label{pt3-281021-2128}
                        $A\succeq 0$,    
                        $\cC(B)\subseteq\cC(A)$ and $M/A\succeq 0$.
			\end{enumerate}
                \smallskip
		\item 
                \label{prop-2604-1140-eq2}
                If $M\succeq 0$, then 
		      $$\Rank M= \Rank A+\Rank M/A=\Rank C+\Rank M/C.$$
	\end{enumerate}
\end{theorem}


\subsection{Partially positive semidefinite matrices and their completions}\label{SubS2.4}

A \textbf{partial matrix} $A=(a_{i,j})_{i,j=1}^n$ is a matrix of real numbers $a_{i,j}\in \RR$, where some of the entries are not specified. 

A partial symmetric matrix $A=(a_{i,j})_{i,j=1}^n$ is 
\textbf{partially positive semidefinite (ppsd)} 
(resp.\ \textbf{partially positive definite (ppd)}) 
if the following two conditions hold:
\begin{enumerate} 
  \item $a_{i,j}$ is specified if and only if $a_{j,i}$ is specified and $a_{i,j}=a_{j,i}$.
  \item All fully specified principal minors of $A$ are psd (resp.\ pd). 
\end{enumerate}

For $n\in \NN$ write $[n]:=\{1,2,\ldots,n\}$.
We denote by 
$A_{Q_1,Q_2}$ 
the submatrix of 
$A\in \RR^{n\times n}$ 
consisting of the rows indexed by the elements of $Q_1\subseteq [n]$
and the columns indexed by the elements of $Q_2\subseteq [n]$. 
In case $Q:=Q_1=Q_2$, we write 
$A_{Q}:=A_{Q,Q}$
for short. 

\begin{lemma}[{\cite[Lemma 2.4]{YZ24}}]
	\label{psd-completion}
	Let 
		$A(\mathbf{x})$	
	be a partially positive semidefinite symmetric matrix of size $n\times n$ with the missing entries 
        in the positions $(i,j)$ and $(j,i)$, $1\leq i<j\leq n$.
	Let 
	\begin{align*}
		A_1 &= (A(\mathbf{x}))_{[n]\setminus \{i,j\}},\; 
            a=(A(\mathbf{x}))_{[n]\setminus \{i,j\},\{i\}},\; 
            b=(A(\mathbf{x}))_{[n]\setminus \{i,j\},\{j\}},\;
		\alpha=(A(\mathbf{x}))_{i,i},\;
		\gamma=(A(\mathbf{x}))_{j,j}.
	\end{align*}
	Let
		$$A_2=(A(\mathbf{x}))_{[n]\setminus \{j\}}	
			=\begin{pmatrix}
				A_1 & a \\
				a^T & \alpha
			\end{pmatrix}\in S_{n-1},\qquad
		A_3=(A(\mathbf{x}))_{[n]\setminus \{i\}}
			=\begin{pmatrix}
				A_1 & b \\
				b^T & \gamma
			\end{pmatrix}\in S_{n-1},$$
	and
		$$x_{\pm}:=b^TA_1^{\dagger}a\pm \sqrt{(A_2/A_1)(A_3/A_1)}\in \RR.$$
	Then: 
	\begin{enumerate}
		\item\label{psd-comp-pt1} $A(x_{0})$ is positive semidefinite if and only if $x_0\in [x_-,x_+]$. 
	 	\item\label{psd-comp-pt2} 
			$$\Rank A(x_0)=
			\left\{\begin{array}{rl}
			\max\big\{\Rank A_2, \Rank A_3\big\},& \text{for}\;x_0\in \{x_-,x_+\},\\[0.5em]
			\max\big\{\Rank A_2, \Rank A_3\big\}+1,& \text{for}\;x_0\in (x_-,x_+).
			\end{array}\right.$$
            \item 
            The following statements are equivalent:
	\begin{enumerate}
		\item $x_-=x_+$.\smallskip
		\item $A_2/A_1=0$ or $A_3/A_1=0$.\smallskip
		\item $\Rank A_2=\Rank A_1$ or $\Rank A_3=\Rank A_1$.
	\end{enumerate}
	\end{enumerate}
\end{lemma}

\subsection{Extension principle}

\begin{proposition}[{\cite[Proposition 2.4]{Fia95} or \cite[Lemma 2.4]{Zal22a}}]
	\label{extension-principle}
	Let $\cA\in S_n$ be positive semidefinite, 
	$Q$ a subset of the set $\{1,\ldots,n\}$
	and	
	$\cA_Q$ the restriction of $\cA$ to the rows and columns from the set $Q$. 
	If $\cA_Qv=0$ for a nonzero vector $v$,
	then $\cA\widehat v=0$, where $\widehat{v}$ is a vector with the only nonzero entries in the rows from $Q$ and such that the restriction 
	$\widehat{v}_Q$ to the rows from $Q$ equals to $v$. 
\end{proposition}

\subsection{(Strong) Hamburger TMP}\label{SubS2.2}
In this subsection we recall the solutions to the univariate TMP and its strong version, since it will be essentially used in the proofs of our main results.\\

Let $k\in \NN$
and
		$\gamma:=(\gamma_0,\ldots,\gamma_{2k} )\in \RR^{2k+1}$.
We say that $\gamma$ is \textbf{$\RR$--representable} if there is a positive Borel measure $\mu$ on $\RR$ such that $\gamma_i=\int_\RR x^i\;d\mu$ for $0\leq i\leq 2k$. 
Characterizing the existence of the $\RR$--rm for $\gamma$ is called the \textbf{truncated Hamburger moment problem (THMP)} or also the \textbf{$\RR$--TMP}.\\

We define the Hankel matrix corresponding to $\gamma$ by
	\begin{equation}
        \label{vector-v}
		A_{\gamma}:=\left(\gamma_{i+j} \right)_{i,j=0}^k
					=\left(\begin{array}{ccccc} 
							\gamma_0 & \gamma_1 & \gamma_2 & \cdots & \gamma_k\\
							\gamma_1 & \gamma_2 & \iddots & \iddots & \gamma_{k+1}\\
							\gamma_2 & \iddots & \iddots & \iddots & \vdots\\
							\vdots 	& \iddots & \iddots & \iddots & \gamma_{2k-1}\\
							\gamma_k & \gamma_{k+1} & \cdots & \gamma_{2k-1} & \gamma_{2k}
						\end{array}\right)
					\in S_{k+1}.
	\end{equation}
For $m\leq k$ we denote
the upper left--hand corner $\left(\gamma_{i+j} \right)_{i,j=0}^m\in S_{m+1}$ of $A_{\gamma}$ of size $m+1$ by $A_{\gamma}(m)$, while the lower right--hand corner of $A_\gamma$ of size $m+1$ by $A_\gamma[m]$.\\

The solution to the THMP is the following.

\begin{theorem} [{\cite[Theorems 3.9--3.10]{CF91}}]
    \label{Hamburger} 
	For $k\in \NN$ and $\gamma=(\gamma_0,\ldots,\gamma_{2k})\in \RR^{2k+1}$ with $\gamma_0>0$, the following statements are equivalent:
\begin{enumerate}	
	\item 
            There exists a $\RR$--representing measure for $\gamma$.\smallskip
        \item 
            There exists a $(\Rank A_\gamma)$--atomic $\RR$--representing measure for $\gamma$.\smallskip
	\item\label{pt5-v2206} 
            $A_\gamma$ is positive semidefinite
            and one of the following holds:
            \smallskip
            \begin{enumerate}
                \item 
                $A_\gamma(k-1)$ is positive definite.
                \smallskip
                \item 
                $\Rank A_\gamma(k-1)=\Rank A_\gamma$.
            \end{enumerate}
\end{enumerate}
\end{theorem}
\medskip

Let $k_1,k_2\in \NN$
and
		\begin{equation} 
            \label{def:strong}
            \widetilde\gamma:=
            (
                \widetilde\gamma_{-2k_1},
                \widetilde\gamma_{-2k_1+1},
                \widetilde\gamma_{-2k_2+2},
                \ldots,
                \widetilde\gamma_{2k_2-1},
                \widetilde\gamma_{2k_2} 
            )
                \in \RR^{2k_1+2k_2+1}.
            \end{equation}
We say that $\widetilde\gamma$ is \textbf{strongly $\RR$--representable} if there is a positive Borel measure $\mu$ on $\RR\setminus \{0\}$ such that
$\widetilde\gamma_i=\int_\RR x^i\;d\mu$ for $-2k_1\leq i\leq 2k_2$.
Characterizing the existence of the $(\RR\setminus \{0\})$--rm for $\widetilde\gamma$ is called \textbf{the 
strong truncated Hamburger moment problem (STHMP)}.\\

The solution to the STHMP is the following.

\begin{theorem} 
    \label{thm:strong-Hamburger} 
	Let $k_1,k_2\in \NN$ and $\widetilde\gamma$ as in \eqref{def:strong} with $\gamma_{-2k_1}>0$.
    Define $\gamma:=(\gamma_{0},\gamma_1,\ldots,\gamma_{2k_1+2k_2})\in \RR^{2k_1+2k_2+1}$ by
    $\gamma_i:=\widetilde\gamma_{i-2k_1}$.
    The following statements are equivalent:
\begin{enumerate}	
	\item 
            There exists a $(\RR\setminus\{0\})$--representing measure for $\widetilde\gamma$.\smallskip
        \item 
            There exists a $(\Rank A_\gamma)$--atomic $(\RR\setminus \{0\})$--representing measure for $\gamma$.\smallskip
	\item 
            $A_\gamma$ is positive semidefinite
            and one of the following holds:
            \smallskip
            \begin{enumerate}
                \item 
                $A_\gamma$ is positive definite.
                \smallskip
                \item\label{thm:strong-Hamburger-pt2-b} 
                    $\Rank A_\gamma=\Rank A_\gamma(k_1+k_2-1)=\Rank A_\gamma[k_1+k_2-1]$.
            \end{enumerate}
\end{enumerate}
\end{theorem}

Let $k\in \NN$. We say a sequence $\gamma=(\gamma_0,\gamma_1,\ldots,\gamma_{2k})\in \RR^{2k+1}$ is \textbf{$(\RR\setminus \{0\})$--representable} if there is a positive Borel measure $\mu$ on $\RR\setminus \{0\}$ such that
$\gamma_i=\int_{\RR\setminus\{0\}} x^i\;d\mu$ for $0\leq i\leq 2k$.

Note that Theorem \ref{thm:strong-Hamburger} above characterizes when a given sequence $\gamma$ is $(\RR\setminus \{0\})$--representable.

\begin{remark}
The matrix version of Theorem \ref{thm:strong-Hamburger} appears in \cite{Sim06} using involved operator theory as the main tool. A proof of the scalar version using linear algebra techniques is \cite[Theorems 3.1]{Zal22b}.
\end{remark}
\bigskip


\subsection{Hyperbolic TMP}
\label{subsection:hyperbolic-TMP}

We will need the following solution to the hyperbolic TMP (see \cite[Corollary 3.5]{Zal22b} and Remark
\ref{281023-1930} below).

\begin{theorem}
	\label{090622-1611}
	Let $p(x,y)=xy-1$ nd 
	$\beta:=\beta^{(2k)}=\{\beta_{i,j}\}_{i,j\in \ZZ_+,i+j\leq 2k}$, where $k\geq 2$. 
	Let
	\begin{equation}
		\label{300822-1412}
			\cB=\{Y^k,Y^{k-1},\ldots,Y,\mathit{1},X,X^2,\ldots,X^k\}.
	\end{equation}
	Then the following statements are equivalent:
	\begin{enumerate}
		\item
			\label{040722-1324-pt1} 
				$\beta$ has a $\mc Z(p)$--representing measure.
            \smallskip
		\item
			\label{040722-1324-pt2} 
				$\beta$ has a $(\Rank \mc M(k))$--atomic $\mc Z(p)$--representing measure.
            \smallskip
		\item
			\label{040722-1324-pt4} 
				$\mc M(k)$ is positive semidefinite, the relations $\beta_{i+1,j+1}=\beta_{i,j}$ hold for every $i,j\in \ZZ_+$ with $i+j\leq 2k-2$ 
				and one of the following statements holds:
            \smallskip
		\begin{enumerate}
			\item
				\label{040722-1403-pt1}  	
					$\mc M(k)_{\cB}$ is positive definite.
                \smallskip
			\item
				\label{040722-1403-pt2} 
					$\Rank \mc M(k)=\Rank \mc M(k)_{\cB\setminus\{Y^k\}}=\Rank \mc M(k)_{\cB\setminus\{X^k\}}$.
		\end{enumerate}
	\end{enumerate}
\end{theorem}

\begin{remark}
    \label{281023-1930}
        The first solution to the hyperbolic TMP
 is \cite[Theorem 1.5]{CF05}, which
 contains a condition called \textit{variety condition}.
 To apply the solution to the hyperbolic TMP, when solving the TMP on a reducible cubic with an irreducible component equivalent to the hyperbola $xy=1$ after applying an invertible affine linear transformation,
 it is not easy to check the variety condition symbolically.
 Theorem \ref{090622-1611} does not contain the variety condition, but only linear algebraic conditions.
 Theorem \ref{090622-1611} is
 a slight improvement of \cite[Corollary 3.5]{Zal22b}.
 Namely, instead of \eqref{040722-1324-pt4}
 the statement in \cite[Corollary 3.5]{Zal22b}
 reads:
 \begin{enumerate}
 \item[(3$'$)]
 $\mc M(k)$ is positive semidefinite, recursively generated and if
				$\Rank \mc M(k)_\cB=2k$, then
				$$\Rank \mc M(k)_{\cB\setminus\{X^k\}}=\Rank \mc M(k)_{\cB\setminus\{Y^k\}}=2k.$$
 \end{enumerate}
 Furthermore, in the proof of \cite[Corollary 3.5]{Zal22b} it is shown
 that (3$'$) is equivalent to
 a variant of \eqref{040722-1324-pt4} (labelled (A) in the proof), in which all rg relations are assumed, not only \textit{pure ones}, 
 i.e., \begin{equation}\label{020822-1740}
		\beta_{i+1,j+1}=\beta_{i,j}\text{ holds for every }i,j\in \ZZ_+\text{ with }i+j\leq 2k-2,
	\end{equation}
 or equivalently
 \begin{equation}\label{281023-2036}
 X^{i+1}Y^{j+1}=X^iY^j
 \text{ hold
 for all }i,j\in \ZZ_+\text{ with }i+j\leq k-2.
 \end{equation}
 The improvement of Theorem \ref{090622-1611} compared to \cite[Corollary 3.5]{Zal22b}
 lies in the replacement of the assumption
 \textit{$\cM(k)$ is rg} by the seemingly weaker assumption that \textit{$\cM(k)$ satisties all pure rg relations} (i.e., \eqref{020822-1740}).
 We now explain why this can be done.
 Due to the existence of the relations \eqref{281023-2036},
 it is sufficient to assume that other relations are among columns and rows, indexed by
 $\cB$ (see \eqref{300822-1412}).
Defining
 $v=
 (
 \beta_{0,2k},
 \beta_{0,2k-1},
 \beta_{0,2k-2},
 \ldots,
 \beta_{0,0},
 \beta_{1,0},
 \beta_{2,0},
 \ldots,
 \beta_{2k,0}
 )\in \RR^{4k+1},$
 note that
 $\cM(k)_{\cB}=A_v$ (see \eqref{vector-v})
 is a singular psd Hankel matrix,
 which is rg in the sense of a univariate sequence \cite[Section 2]{Zal22b}
 by the assumptions in \eqref{040722-1403-pt2} and \cite[Proposition 2.1.(4),(5)]{Zal22b}.
        By \cite[Theorem 3.1]{Zal22b}, $v$ has a representing measure supported on 
        $\RR\setminus \{0\}$,
 where $\beta_{0,i}$ corresponds to the moment of $x^{-i}$ and $\beta_{j,0}$ corresponds to the moment of $x^j$.
 By \cite[Claim in the proof of Corollary 3.5]{Zal22b}, $\beta$ has a $\cZ(xy-1)$--rm.
\end{remark}


\begin{corollary}
	\label{cor:x+y+axy}
	Let $p(x,y)=x+y-xy$ and 
	$\beta:=\beta^{(2k)}=\{\beta_{i,j}\}_{i,j\in \ZZ_+,i+j\leq 2k}$, where $k\geq 2$. 
	Let
	\begin{equation}
		\label{basis:x+y+axy}
			\cB=\{Y^k,Y^{k-1},\ldots,Y,\mathit{1},X,X^2,\ldots,X^k\}.
	\end{equation}
	Then the following statements are equivalent:
	\begin{enumerate}
		\item
			\label{cor:x+y+axy-pt1} 
				$\beta$ has a $\mc Z(p)$--representing measure.
            \smallskip
		\item
			\label{cor:x+y+axy-pt2} 
				$\beta$ has a $(\Rank \mc M(k))$--atomic $\mc Z(p)$--representing measure.
            \smallskip
		\item
			\label{cor:x+y+axy-pt3} 
				$\mc M(k)$ is positive semidefinite, the relations $\beta_{i+1,j+1}=\beta_{i+1,j}+\beta_{i,j+1}=0$ hold for every $i,j\in \ZZ_+$ with $i+j\leq 2k-2$ 
				and one of the following statements holds:
            \smallskip
		\begin{enumerate}
			\item
				\label{cor:x+y+axy-pt3-1}  	
					$\mc M(k)_{\cB}$ is positive definite.
                \smallskip
			\item
				\label{cor:x+y+axy-pt3-2} 
					$\Rank \mc M(k)=\Rank \mc M(k)_{\cB\setminus\{Y^k\}}=\Rank \mc M(k)_{\cB\setminus\{X^k\}}$.
		\end{enumerate}
	\end{enumerate}
\end{corollary}

\begin{proof}
Note that applying an affine linear transformation $\phi(x,y)=(x+1,y+1)$
(see Section \ref{sub:affine-linear-trans}) to $\beta$ we obtain a new sequence 
$\widetilde\beta^{(2k)}=\{\widetilde\beta_{i,j}\}_{i,j\in \ZZ_+,i+j\leq 2k}$
satisfying the relations $\widetilde\beta_{i+1,j+1}=\widetilde\beta_{i,j}$ for 
$i,j\in \ZZ_+$ with $i+j\leq 2k-2$.
Since the existence of a $\cZ(p)$--rm for $\beta$ is equivalent to the existence of a 
$\cZ(xy-1)$--rm for $\widetilde\beta$, Corollary \ref{cor:x+y+axy} follows by 
Theorem \ref{090622-1611}.
\end{proof}


\begin{corollary}
	\label{090622-1611-hyp-v3}
	Let $p(x,y)=ay+x^2-y^2$, $a\in \RR\setminus \{0\}$, and 
	$\beta:=\beta^{(2k)}=\{\beta_{i,j}\}_{i,j\in \ZZ_+,i+j\leq 2k}$, where $k\geq 2$. 
	Let
	\begin{equation}
		\label{300822-1412-v3}
			\cB'=\{YX^{k-1},YX^{k-2},\ldots,Y,\mathit{1},X,X^2,\ldots,X^k\}.
	\end{equation}
	Then the following statements are equivalent:
	\begin{enumerate}
		\item
			\label{040722-1324-pt1} 
				$\beta$ has a $\mc Z(p)$--representing measure.
            \smallskip
		\item
			\label{040722-1324-pt2} 
				$\beta$ has a $(\Rank \mc M(k))$--atomic $\mc Z(p)$--representing measure.
            \smallskip
		\item
			\label{040722-1324-pt4} 
				$\mc M(k)$ is positive semidefinite, the relations $\beta_{i,j+2}=\beta_{i+2,j}+a\beta_{i,j+1}$ hold for every $i,j\in \ZZ_+$ with $i+j\leq 2k-2$ 
				and one of the following statements holds:
            \smallskip
		\begin{enumerate}
			\item
				\label{040722-1403-pt1-v3}  	
					$\mc M(k)_{\cB'}$ is positive definite.
                \smallskip
			\item
				\label{040722-1403-pt2-v3} 
					$\Rank \mc M(k)=\Rank \mc M(k)_{\cB'\setminus\{YX^{k-1}\}}=\Rank \mc M(k)_{\cB'\setminus\{X^k\}}$.
		\end{enumerate}
	\end{enumerate}
\end{corollary}

\begin{proof}
Note that applying an affine linear transformation 
    $\phi(x,y)=\big(\frac{2}{a}(x+\frac{a}{2}-y),\frac{2}{a}(x+\frac{a}{2}+y)\big)$
(see Section \ref{sub:affine-linear-trans}) to $\beta$ we obtain a new sequence 
$\widetilde\beta^{(2k)}=\{\widetilde\beta_{i,j}\}_{i,j\in \ZZ_+,i+j\leq 2k}$
satisfying the relations $\widetilde\beta_{i+1,j+1}=\widetilde\beta_{i,j}$ for 
$i,j\in \ZZ_+$ with $i+j\leq 2k-2$. 
Since the existence of a $\cZ(p)$--rm for $\beta$ is equivalent to the existence of a 
$\cZ(xy-1)$--rm for $\widetilde\beta$, Corollary \ref{090622-1611-hyp-v3} follows by 
Theorem \ref{090622-1611}.
\end{proof}

%% file: Common-approach.tex
\section{Common approach to all cases}
\label{Section-common-approach}

In this section we recall the constructive approach to solving the TMP on reducible cubic curves in the canonical form $yc(x,y)=0$ developed in \cite[Section 4]{YZ24}.\\

Let  
\begin{equation}
\label{degree-lex}
    \cC=\{\mathit{1},X,Y,X^2,XY,Y^2,\ldots,X^k,X^{k-1}Y,\ldots,Y^k\}
\end{equation}
be the set of columns and rows of the moment matrix $\cM(k)$
in the degree-lexicographic order.
Let 
\begin{equation}
\label{form-of-p}
    p(x,y)=y\cdot c(x,y)\in \RR[x,y]_{\leq 3}
\end{equation}
be a polynomial of degree 3 in one of the canonical forms from \cite[Proposition 3.1]{YZ24}, were  $c(x,y)$ a polynomial of degree 2.
A given 2--dimensional sequence 
$\beta=\{\beta_{i,j}\}_{i,j\in \ZZ_+,i+j\leq 2k}$ 
of degree $2k$, $k\in \NN$,
will have a $\cZ(p)$--rm if and only if it can be decomposed as
\begin{equation}
\label{decomposition-v2}
	\beta=\beta^{(\ell)}+\beta^{(c)},
\end{equation} 
where
\begin{align*}
    \beta^{(\ell)}
    &:=
    \{\beta_{i,j}^{(\ell)}\}_{i,j\in \ZZ_+,i+j\leq 2k}
    \quad
    \text{has a measure on }y=0,\\
    \beta^{(c)}    
    &:=
    \{\beta_{i,j}^{(c)}\}_{i,j\in \ZZ_+,i+j\leq 2k}
    \quad
    \text{has a measure on the conic }c(x,y)=0,
\end{align*}
and the sum in \eqref{decomposition-v2} is the component-wise sum.
On the level of moment matrices, \eqref{decomposition-v2}
is equivalent to 
\begin{equation}
\label{decomposition-v3}
	\cM(k;\beta)=\cM(k;\beta^{(\ell)})+\cM(k;\beta^{(c)}).
\end{equation}
Note that if $\beta$ has a $\cZ(p)$--rm,
then the matrix $\mc M(k;\beta)$ satisfies the relation 
$p(X,Y)=\mathbf{0}$ and by rg also
\begin{equation}
    \label{071123-0657}
        (x^iy^jp)(X,Y)=\mathbf{0}
        \quad
        \text{for }i,j=0,\ldots,k-3\text{ such that }i+j\leq k-3.
\end{equation}

Let $\cT\subseteq \cC$ be a subset, such that $\{1,X,\ldots,X^k\}\subseteq \cT$ and the columns from $\cT$
span the column space $\cC(\cM(k;\beta))$.
We write 
    $\vec{X}^{(0,k)}:=(1,X,\ldots,X^k)$,
    $\cT_1=\cT\setminus \{1,X,\ldots,X^k\}$
and
    \begin{align}
    \label{071123-1939}
		\widetilde\cM(k;\beta)
		&:=
        Q\cM(k;\beta)Q^T
        =
		\kbordermatrix{
		& \vec{X}^{(0,k)}
            & \vec{\cT}_1
            & \overrightarrow{\cC\setminus\cT} \\
			(\vec{X}^{(0,k)})^T & A_{11} & A_{12} & A_{13}\\[0.3em]
		  (\vec{\cT}_1)^T & (A_{12})^T & A_{22} & A_{23}\\[0.3em]
            (\overrightarrow{\cC\setminus\cT})^T & (A_{13})^T & (A_{23})^T & A_{33}\\
		},
    \end{align}
where $\vec{\cT}_1$ and $\overrightarrow{\cC\setminus\cT}$ are tuples of elements of $\cT_1$ and $\cC\setminus\cT$, arranged in some order, and 
$Q$ is the appropriate permutation matrix.
In this new order of rows and columns, \eqref{decomposition-v3}
becomes equivalent to
\begin{equation}
\label{decomposition-v4}
	\widetilde\cM(k;\beta)=
 \widetilde\cM(k;\beta^{(\ell)})+
 \widetilde\cM(k;\beta^{(c)}),
\end{equation}
where 
$\widetilde\cM(k;\beta^{(\ell)}):=
Q\cM(k;\beta^{(\ell)})Q^T$
and
$\widetilde\cM(k;\beta^{(c)}):=
Q\cM(k;\beta^{(c)})Q^T.$
By the form of the atoms we know that 
$\widetilde\cM(k;\beta^{(c)})$ 
    and 
$\widetilde\cM(k;\beta^{(\ell)})$ 
will have forms
    \begin{align}
    \label{071123-0642}
    \begin{split}
		\widetilde\cM(k;\beta^{(c)})
		&=
		\kbordermatrix{
		& \vec{X}^{(0,k)}
            & \vec{\cT}_1
            & \overrightarrow{\cC\setminus\cT} \\
			(\vec{X}^{(0,k)})^T & A & A_{12} & A_{13}\\[0.3em]
		  (\vec{\cT}_1)^T & (A_{12})^T & A_{22} & A_{23}\\[0.3em]
            (\overrightarrow{\cC\setminus\cT})^T & (A_{13})^T & (A_{23})^T & A_{33}\\
		},\\[0.5em]
		\widetilde\cM(k;\beta^{(\ell)})
		&=
		\kbordermatrix{
		& \vec{X}^{(0,k)}
            & \vec{\cT}_1
            & \overrightarrow{\cC\setminus\cT} \\
			(\vec{X}^{(0,k)})^T & A_{11}-A & \mathbf{0} & \mathbf{0}\\[0.3em]
		  (\vec{\cT}_1)^T & \mathbf{0} & \mathbf{0} & \mathbf{0}\\[0.3em]
            (\overrightarrow{\cC\setminus\cT})^T & \mathbf{0} & \mathbf{0} & \mathbf{0}\\
		}
    \end{split}
    \end{align}
for some Hankel matrix $A\in S_{k+1}.$
    Define matrix functions 
    $$
    \cF:S_{k+1}\to S_{\frac{(k+1)(k+2)}{2}}
    \qquad\text{and}\qquad
    \cH:S_{k+1}\to S_{k+1}
    $$
    by
    \begin{align}
    \label{071123-1940}
    \begin{split}
    \cF(\mathbf{A})
		&=
		\kbordermatrix{
		& \vec{X}^{(0,k)}
            & \vec{\cT}_1
            & \overrightarrow{\cC\setminus\cT} \\
			(\vec{X}^{(0,k)})^T & \mathbf{A} & A_{12} & A_{13}\\[0.3em]
		  (\vec{\cT}_1)^T & (A_{12})^T & A_{22} & A_{23}\\[0.3em]
            (\overrightarrow{\cC\setminus\cT})^T & (A_{13})^T & (A_{23})^T & A_{33}\\
		},\\[0.5em]
    \cH(\mathbf{A})
        &=
        \kbordermatrix{
        & \vec{X}^{(0,k)}\\
        (\vec{X}^{(0,k)})^T & A_{11}-\mathbf{A}
        }.
    \end{split}
    \end{align}
    Using \eqref{071123-0642}, \eqref{decomposition-v4}
    becomes equivalent to
    \begin{equation}
        \label{decomposition-v5}
        \widetilde{\mc M}(k;\beta)
        =
        \cF(A)+\cH(A)\oplus \mathbf{0}_{\frac{k(k+1)}{2}}
    \end{equation}
    for some Hankel matrix $A\in S_{k+1}$.

 	\begin{lemma}[{\cite[Lemma 4.1]{YZ24}}]
		\label{071123-0646}
            Let $k\in \NN, k\geq 3$. Assume the notation above.
            Then the sequence 
		$\beta=\{\beta_{i,j}\}_{i,j\in \ZZ_+,i+j\leq 2k}$
            has a $\cZ(p)$--representing measure if and only if there exist a Hankel matrix 
		$A\in S_{k+1}$,
		such that:
		\begin{enumerate}
			\item
                \label{301023-1756-pt1}
                The sequence 
				with the moment matrix $\cF(A)$ has a $\cZ(c)$--representing measure.
			\item
                \label{301023-1756-pt2}
   The sequence with the moment matrix $\cH(A)$ has a $\RR$--representing measure.
		\end{enumerate}
	\end{lemma}

\begin{lemma}[{\cite[Lemma 4.2]{YZ24}}]
    \label{071123-1942}
    Let $k\in \NN, k\geq 3$. Assume the notation above
    and the sequence 
	$\beta=\{\beta_{i,j}\}_{i,j\in \ZZ_+, i+j\leq 2k}$ admits a $\cZ(p)$--representing measure.
    Let $$A:=A_{\big(\beta_{0,0}^{(c)},\beta_{1,0}^{(c)},\ldots,\beta_{2k,0}^{(c)}\big)}\in S_{k+1}$$ be a Hankel matrix (see \eqref{vector-v})
    such that $\cF(A)$ admits a $\cZ(c)$--representing measure and $\cH(A)$ admits a $\RR$--representing measure.
    Let $c(x,y)$ be of the form
        \begin{align}
        \label{081123-1004}
        \begin{split}
        &c(x,y)
        =a_{00}+a_{10}x+a_{20}x^2+a_{01}y+a_{02}y^2+a_{11}xy\quad
        \text{with }a_{ij}\in \RR\\
        &\text{and exactly one of the coefficients }a_{00},a_{10},a_{20}\text{ is nonzero}.
        \end{split}
        \end{align}
    If:
    \begin{enumerate}
        \item
        \label{071123-1942-pt1}
        $a_{00}\neq 0$, then 
            $$
            \beta_{i,0}^{(c)}
            =
            -\frac{1}{a_{0,0}}
            (a_{01}\beta_{i,1}+a_{0,2}\beta_{i,2}+a_{11}\beta_{i+1,1})
            \quad
            \text{for }i=0,\ldots,2k-2.
            $$
        \item 
        \label{071123-1942-pt2}
        $a_{10}\neq 0$, then 
            $$
            \beta_{i,0}^{(c)}
            =
            -\frac{1}{a_{1,0}}
            (a_{01}\beta_{i,1}+a_{0,2}\beta_{i,2}+a_{11}\beta_{i+1,1})
            \quad
            \text{for }i=1,\ldots,2k-1.
            $$
        \item 
        \label{071123-1942-pt3}
        $a_{20}\neq 0$, then 
            $$
            \beta_{i,0}^{(c)}
            =
            -\frac{1}{a_{2,0}}
            (a_{01}\beta_{i,1}+a_{0,2}\beta_{i,2}+a_{11}\beta_{i+1,1})
            \quad
            \text{for }i=2,\ldots,2k.
            $$
    \end{enumerate}
\end{lemma}

Lemma \ref{071123-1942} states that if $c$ is as in \eqref{081123-1004}, all but two entries of the Hankel matrix $A$ from Lemma \ref{071123-0646} are uniquely determined by $\beta$.
The following lemma gives the smallest candidate for $A$ in Lemma \ref{071123-0646} with respect to the usual Loewner order of matrices.

\begin{lemma}[{\cite[Lemma 4.3]{YZ24}}]
	\label{071123-2008}
    Assume the notation above
    and let
	$\beta=\{\beta_{i,j}\}_{i,j\in \ZZ_+,i+j\leq 2k}$, where $k\geq 3$, be a sequence of degree $2k$.
    Assume that $\widetilde{\mc{M}}(k;\beta)$ is positive semidefinite and 
    satisfies the column relations \eqref{071123-0657}.
		Then:
		\begin{enumerate}	
                \item 
                    \label{071123-2008-pt0}
                    $\cF(A)\succeq 0$ for some $A\in S_{k+1}$
                    if and only if
                    $A\succeq A_{12}(A_{22})^{\dagger} (A_{12})^T$.
                \smallskip
			\item
				\label{071123-2008-pt1}
					$\cF\big(A_{12}(A_{22})^{\dagger} (A_{12})^T\big)\succeq 0$ 
					and 	
					$\cH\big(A_{12}(A_{22})^{\dagger} (A_{12})^T\big) \succeq 0$.
                \smallskip
			\item
				\label{071123-2008-pt2} 
					$\cF\big(A_{12}(A_{22})^{\dagger} (A_{12})^T\big)$ satisfies the column relations
                    $(x^iy^jc)(X,Y)=0$
                    for $i,j\in \ZZ_+$ such that $i+j\leq k-2$.
                \smallskip
			\item 	
				\label{071123-2008-pt3}
                    We have that
                        \begin{align*}
                        \Rank \widetilde{\mc M}(k;\beta)
                        &=
                        \Rank A_{22}+
                        \Rank \big(A_{11}-A_{12}(A_{22})^{\dagger} (A_{12})^T\big)\\
                        &=\Rank \cF\big(A_{12}(A_{22})^{\dagger} (A_{12})^T\big)+
                        \Rank \cH\big(A_{12}(A_{22})^{\dagger} (A_{12})^T\big).
                        \end{align*}
		\end{enumerate}
\end{lemma}

\begin{remark}
\label{general-procedure}
By Lemmas \ref{071123-0646}--\ref{071123-2008}, 
solving the $\cZ(p)$--TMP for 
    $\beta=\{\beta_{i,j}\}_{i,j\in \ZZ_+,i+j\leq 2k}$, 
where $k\geq 3$, 
with $p$ being of the form $yc(x,y)$ and $c$
as in \eqref{081123-1004}, 
the natural procedure is the following:
\begin{enumerate}
    \item\label{general-procedure-pt1} First compute $A_{\min}:=A_{12}(A_{22})^{\dagger}(A_{12})^T$.
        By Lemma \ref{071123-2008}.\eqref{071123-2008-pt2}, there is one entry
        of $A_{\min}$,
        which might need to be changed to obtain a Hankel structure. 
        Namely, in the notation \eqref{081123-1004},
        if:
        \begin{enumerate}
            \item $a_{00}\neq 0$, then the value of $(A_{\min})_{k,k}$ must be changed to $(A_{\min})_{k-1,k+1}$.
            \item $a_{10}\neq 0$, then the value of $(A_{\min})_{1,k+1}$   
                must be changed to $(A_{\min})_{2,k}$.
            \item $a_{20}\neq 0$, then the value of $(A_{\min})_{2,2}$ must be changed to $(A_{\min})_{3,1}$.
        \end{enumerate}
        Let $\widehat A_{\min}$ be  the matrix obtained from $A_{\min}$ after performing the change described above.
    \item Study if $\cF(\widehat A_{\min})$ and $\cH(\widehat A_{\min})$ admit a $\cZ(c)$--rm 
        and a $\RR$--rm, respectively.
        If the answer is yes, $\beta$ admits a $\cZ(p)$--rm. 
        Otherwise by Lemma \ref{071123-1942},
        there are two antidiagonals of the Hankel matrix $\widehat A_{\min}$, 
        which can by varied so that the matrices $\cF(\widehat A_{\min})$ and $\cH(\widehat A_{\min})$
        will admit the corresponding measures. 
        Namely, in the notation \eqref{081123-1004},
        if:
        \begin{enumerate}
            \item $a_{00}\neq 0$, then the last two antidiagonals of $\widehat A_{\min}$ 
            can be changed.
            \item $a_{10}\neq 0$, then the left--upper and the right--lower corner of $\widehat A_{\min}$ 
            can be changed.
            \item $a_{20}\neq 0$, then the first two antidiagonals of $\widehat A_{\min}$ 
            can be changed.
        \end{enumerate}
        To solve the $\cZ(p)$--TMP for $\beta$ one needs to characterize, when it is possible to change these antidiagonals in such a way
        to obtain a matrix $\breve A_{\min}$, such that 
        $\cF(\breve A_{\min})$ and $\cH(\breve A_{\min})$ admit a $\cZ(c)$--rm 
        and a $\RR$--rm, respectively.
\end{enumerate}
\end{remark}

%% file: Hyperbolic-type-1-new.tex
\section{Hyperbolic type 1 relation: 
            $p(x,y)=y(1-xy)$.
        }
\label{Hyperbolic-type-1}

In this section we solve constructively the $\cZ(p)$--TMP for 
the sequence $\beta=\{\beta_{i,j}\}_{i,j\in \ZZ_+,i+j\leq 2k}$ 
of degree $2k$, $k\geq 3$,
where $p(x,y)$ is as in the title of the section.
The main result is Theorem \ref{100622-2204}, which characterizes concrete numerical conditions for the existence of a $\cZ(p)$--rm
for $\beta$ and also the number of atoms needed in a minimal $\cZ(p)$--rm. 
A numerical example demonstrating the main result is presented in Subsection \ref{ex:hyperbolic-type-1}.\\

Assume the notation from Section \ref{Section-common-approach}.
If $\beta$ admits a $\cZ(p)$--TMP, then $\mc M(k;\beta)$ 
must satisfy the relations
\begin{equation}
    \label{110622-0029-v2}
	Y^{2+j}X^{i+1}=Y^{1+j}X^{i}
        \quad
        \text{for }
        i,j\in \ZZ_+
        \text{ such that }
        i+j\leq k-3.
\end{equation}
On the level of moments the relations \eqref{110622-0029-v2} mean that
	\begin{equation}\label{100622-1739}	
		\beta_{i+1,j+2}=\beta_{i,j+1}
  \quad
  	\text{for }
   i,j\in \ZZ_+
   \text{ such that }i+j\leq 2k-3.
	\end{equation}
In the presence of all column relations \eqref{110622-0029-v2}, the column space $\cC(\mc M(k;\beta))$ is spanned by the columns in the tuple
    \begin{equation}
    \label{171023-2028-v2}
    \vec\cT:=
    (
    \underbrace{Y^{k},Y^{k-1},\ldots,Y}_{
    \vec Y^{(k,1)}}, 
    \underbrace{YX,YX^2,\ldots,YX^{k-1}}_{Y\vec{X}^{(1,k-1)}},
    \underbrace{\textit 1,X,\ldots,X^k}_{\vec{X}^{(0,k)}}).
    \end{equation}
	
     Let
    \begin{align}
    \label{171023-1707}
    \begin{split}
    &P
    \text{ be a permutation matrix such that moment matrix }
    \widehat{\mc M}(k):=P\mc M(k; \beta)P^T\text{ has rows}\\
    &\text{and columns indexed in the order }
	\vec Y^{(k,1)}, Y\vec{X}^{(1,k-1)},\vec{X}^{(0,k)},Y^2\vec{X}^{(1,k-2)},\ldots,Y^{k-1}X,
    \end{split}
    \end{align}
    where $Y^j\vec{X}^{(1,k-j)}:=(Y^jX,Y^jX^2,\ldots,Y^jX^{k-j})$
    for $1\leq j\leq k-1$.
	Let $\widehat{\mc{M}}(k)_{\vec \cT}$
    be the restriction of the moment matrix 
    $\widehat{\mc{M}}(k)$
    to the rows and columns in the tuple $\vec\cT$
    and write
	\begin{align*}
		\widehat{\mc{M}}(k)_{\vec\cT}
		&=
		\kbordermatrix{
		& \vec{Y}^{(k,1)} &  Y \vec{X}^{(1,k-1)} &   \vec{X}^{(0,k)}\\
			(\vec{Y}^{(k,1)})^T & B_{11} &  B_{12} &  B_{13}\\[0.2em]	
			(Y \vec{X}^{(1,k-1)})^T & (B_{12})^T & B_{22} &  B_{23}\\[0.2em]
			(\vec{X}^{(0,k)})^T & (B_{13})^T & (B_{23})^T & B_{33}}\\[0.5em]
		&=		
		\kbordermatrix{
		& \vec{Y}^{(k,1)} & Y \vec{X}^{(1,k-1)} & \vec{X}^{(0,k-1)} & X^k &  \\
			(\vec{Y}^{(k,1)})^T & B_{11} & B_{12} & B^{(0,k-1)}_{13} & {b}^{(k)}_{13} \\[0.2em]
			(Y \vec{X}^{(1,k-1)})^T & ({B}_{12})^T & {B}_{22} & {B}^{(0,k-1)}_{23} & {b}^{(k)}_{23} \\[0.2em]
			(\vec{X}^{(0,k-1)})^T & ({B}_{13}^{(0,k-1)})^T & ({B}^{(0,k-1)}_{23})^T & 
                B^{(0,k-1)}_{33}
                 & {b}^{(k)}_{33} \\[0.2em]
			X^{k} & (b^{(k)}_{13})^T & ({b}^{(k)}_{23})^T & (b^{(k)}_{33})^T & \beta_{2k,0}}
		\in  S_{3k}.
	\end{align*}
	We also write
	\begin{align}
		\label{def:M-with-R}
            \begin{split}
		 \widehat{\mc{M}}(k)_{\vec\cT} 
		&=:
            \begin{pmatrix}
                R & m_{12} \\[0.2em]
                (m_{12})^T & \beta_{2k,0}
            \end{pmatrix}.
            \end{split}
	\end{align}

Next we define the matrix $\cN(k)$,
which extends the restriction of 
    $\widehat \cM(k)$ to 
rows and columns in $(\vec{Y}^{(k,1)}, Y\vec{X}^{(1,k-1)},\vec{X}^{(0,k-1)})$, with a row
and a column $YX^k$.
Namely, it contains the only candidates for the corresponding moments, which are generated by the atoms 
in any $\cZ(p)$--rm.
The reason for introducing precisely $\cN(k)$ is the fact, that together with $\cM(k)$, they contain crucial information needed
to characterize the existence of the solution to the $\cZ(p)$--TMP
(see Theorem 
\ref{100622-2204} below).
So,
	\begin{align}
        \label{def-N-k}
        \begin{split}
		\mc N(k)
		&=
		\kbordermatrix{
		& \vec{Y}^{(k,1)}& Y \vec{X}^{(1,k-1)} & \vec{X}^{(0,k-1)}  & YX^{k}\\
			(\vec{Y}^{(k,1)})^T & B_{11} & B_{12} & B^{(0,k-1)}_{13} & a\\[0.2em] 
			 (Y\vec{X}^{(1,k-1)})^T & (B_{12})^T &   B_{22} &   B^{(0,k-1)}_{23} & b \\[0.2em]
			(\vec{X}^{(0,k-1)})^T& (B^{(0,k-1)}_{13})^T &  (B^{(0,k-1)}_{23})^T & B^{(0,k-1)}_{33} & c \\[0.2em]
			YX^{k}& a^T & b^T & c^T & \beta_{2k-1,1}		
		}\in S_{3k},
        \end{split}
	\end{align}
	where
	\begin{align*}
	a
		&=
		\begin{pmatrix} \beta_{0,1} & \beta_{1,1} & \cdots & \beta_{k-1,1}\end{pmatrix}^T,\qquad\qquad
	b
		=
		\begin{pmatrix} \beta_{k,1} & \beta_{k+1,1} & \cdots & \beta_{2k-2,1} \end{pmatrix}^T,\\
	c	
		&=
		\begin{pmatrix} \beta_{k,1} & \beta_{k+1,1} & \cdots & \beta_{2k-1,1}  \end{pmatrix}^T.
\end{align*}
In the definition of $a,b,c$ we used the fact that if the $\cZ(p)$--rm for $\beta$ exists, then the relations \eqref{100622-1739}	
hold also for $i+j\leq 2k-1$. In particular, we used the relations
    $\beta_{2k-1,2}=\beta_{2k-2,1}$,
    $\beta_{k,k+1}= \beta_{k-1,k}$
and 
    $\beta_{2k,2}= \beta_{2k-1,1}$.
These relations also imply that 
\begin{equation}
    \label{relations-260625-0628}
    \begin{pmatrix} B_{12} & a\end{pmatrix}=B^{(0,k-1)}_{13},
    \quad\quad
        \begin{pmatrix} B_{22} & b\end{pmatrix}=B^{(0,k-1)}_{23}
    \quad\quad
    \text{and}
    \quad\quad
        \begin{pmatrix} b^T & \beta_{2k-1,1}\end{pmatrix}=c^T.    
\end{equation}
We also write
\begin{equation}
\label{def:N-with-R}
    \cN(k)=:
        \begin{pmatrix}
            R & n_{12} \\ 
            (n_{12})^T & \beta_{2k-1,1}
        \end{pmatrix},
\end{equation}
where note that $R$ is the same as in \eqref{def:M-with-R} above.
    
    Next we define two additional matrices $F_1$ and $F_2$, needed in the statement of the solution to the $\cZ(p)$--TMP:
    \begin{itemize}
    \item 
    $F_1$ the restriction of $\widetilde\cM(k;\beta^{(\ell)})$ (see \eqref{071123-0642}) to rows and columns in $\vec{X}^{(0,k-1)}$.
    Namely, it contains the only candidates for the corresponding moments, which are generated by the atoms in any $\cZ(p)$--rm, that are supported on the line $y=0$.
    \item  
    $F_2$ is the restriction of $\mc N(k)$ to the rows and columns in 
    $(\vec{Y}^{(k,1)},Y\vec{X}^{(1,k)})$.
    \end{itemize}
    So
    \begin{equation}
    \label{def-F1-F2}
    F_1	
					:=
						B^{(0,k-1)}_{33}-\begin{pmatrix} (B^{(0,k-1)}_{23})^T & c \end{pmatrix},\qquad
					F_2
					:=
						\begin{pmatrix}
							B_{11} & B^{(0,k-1)}_{13}\\[0.3em]         
                                (B^{(0,k-1)}_{13})^T & \begin{pmatrix} (B^{(0,k-1)}_{23})^T & c \end{pmatrix}   
						\end{pmatrix},
      \end{equation}
      where we used \eqref{relations-260625-0628} in the equalities.
Define real numbers
				\begin{align}
					\label{200922-1158}
					\begin{split}
					t'
					&=(n_{12})^TR^{\dagger}\; m_{12},\\
					u'
					&=\beta_{2k,0}-(w_1)^T(F_1)^{\dagger}\; w_1,\\
					u''
					&=(w_2)^T(F_2)^\dagger\; w_2,
					\end{split}
				\end{align}
				where 
				\begin{align*}
                        w_1
					&=
						\begin{pmatrix} \beta_{k,0}-\beta_{k+1,1} & \beta_{k+1,0}-\beta_{k+2,1}  & \cdots & \beta_{2k-2,0}-\beta_{2k-1,1} & \beta_{2k-1,0}-t'\end{pmatrix}^T,\\
					w_2
					&=	
						\begin{pmatrix} \beta_{1,1} & \beta_{2,1} & \cdots        &\beta_{2k-1,1} & t'\end{pmatrix}^T.
				\end{align*}
Note that: 
\begin{itemize} 
    \item 
$w_1$ is the difference of the restriction of the column $X^k$ of $\widehat \cM(k)$ to the rows $\vec{X}^{(0,k-1)}$
 and a vector that is the restriction of the only (up to the choice of $t'$) potential column $YX^{k+1}$ of the extension of $\widehat\cM(k)$ to the rows $\vec{X^{(0,k-1)}}$, if the $\cZ(p)$--rm for $\beta$ exists.
    \item 
$w_2$ is the restriction to the rows $(\vec{Y}^{(k,1)},\vec{X}^{(0,k-1)})$ of the only (up to the choice of $t'$) potential column $X^{k}$ of the matrix $\widetilde \cM(k;\beta^{(c)})$ (see \eqref{071123-0642}),
 which is generated by the atoms lying on $xy=1$ if a $\cZ(p)$--rm for $\beta$ exists.
    \item 
$t'$ is the only candidate for the moment of $x^{2k-1}$
 coming from the atoms in any $\cZ(p)$--rm, that are supported on the conic part $\cZ(xy-1)$ of $\cZ(p)$, if one of
 $\widehat{\mc{M}}(k)_{\vec\cT}$ or $\mc N(k)$ is not positive definite.
 Furthermore,
 $u'$ and $u''$ are the only two candidates for the moment of $x^{2k}$ from the conic part that need to be checked when deciding on the existence of a $\cZ(p)$--rm for $\beta$.
\end{itemize}

Define the sequences
					\begin{align}
                        \label{def:gammas-v2}
                        \begin{split}
						\gamma_1(\mathbf{t},\mathbf u)
							&:=(\beta_{0,0}-\beta_{1,1},\beta_{1,0}-\beta_{2,1},\ldots,\beta_{2k-2,0}-\beta_{2k-1,1},\beta_{2k-1,0}-\mathbf{t},\beta_{2k,0}-\mathbf u),\\
						\gamma_2(\mathbf{t},\mathbf u)
							&:= (\beta_{0,2k},\beta_{0,2k-1},\ldots,\beta_{0,1},\beta_{0,0},\beta_{1,0},\ldots,\beta_{2k-2,0},\mathbf t,\mathbf u).
                        \end{split}
					\end{align}
Let
    $\cF(\mathbf{A})$ and $\cH(\mathbf{A})$ 
    be as in 
    \eqref{071123-1940} with     
    $\vec\cT_1
        :=
            (\vec{Y}^{(k,1)},
            Y\vec{X}^{(1,k)})
    $.
    Define the matrix function 
    \begin{equation}
        \label{301023-1930}
        \mc G:\RR^2\to S_{k+1},\qquad 
	\mc G(\mathbf{t},\mathbf{u})=
        \widehat A_{\min}
        +\mathbf{t}\big(E_{k,k+1}^{(k+1)}+E_{k+1,k}^{(k+1)}\big)
        +\mathbf{u}E_{k+1,k+1}^{(k+1)},
    \end{equation}
where $\widehat A_{\min}$ is as in Remark \ref{general-procedure}.\eqref{general-procedure-pt1}.
We will prove that 
    $A_{\gamma_1(t,u)}$,
    $A_{\gamma_2(t,u)}$
(see \eqref{vector-v}) are equal to 
    $\cH(\cG(t,u))$,
    $\cF(\cG(t,u))P^T\big)_{\vec{Y}^{(k,1)}\cup \vec{X}^{(0,k)}}$,
which represent 
possible restrictions of 
$\widetilde\cM(k;\beta^{(\ell)})$
and $\widetilde \cM(k;\beta^{(c)})$
(see \eqref{071123-0642})
to the rows and columns in $\vec{X}^{(0,k)}$ and $(\vec{Y}^{(k,1)},\vec{X}^{(0,k)})$, respectively.\\

The solution to the $\cZ(p)$--TMP is the following. 
      
	\begin{theorem}
		\label{100622-2204} 
	Let $p(x,y)=y(xy-1)$
	and $\beta:=\beta^{(2k)}=(\beta_{i,j})_{i,j\in \ZZ_+,i+j\leq 2k}$, where $k\geq 3$.
	Assume the notation above.
  
	Then the following statements are equivalent:
	\begin{enumerate}	
		\item\label{100622-2206-pt1} $\beta$ has a $\cZ(p)$--representing measure.
        \smallskip
		\item\label{100622-2206-pt2}  
  $\mc{M}(k;\beta)$ and $\cN(k)$ are positive semidefinite, the relations \eqref{100622-1739}
  hold
    and 
	one of the following statements holds:
        \smallskip
        
		\begin{enumerate}
			\item
				\label{100622-2207-pt1}
				$\widehat{\mc{M}}(k)_{\vec\cT}$ 
                and 
                $\mc{N}(k)$ are positive definite.
                \smallskip
                \item
				\label{200922-1200-pt2-alt}
                    $\gamma_1(t',u)$ and $\gamma_2(t',u)$ are $\RR$--representable and $(\RR\setminus \{0\})$--representable, respectively, for some $u\in \{u',u''\}$.
			  \end{enumerate}
	\end{enumerate}

	Moreover, assume a $\mc Z(p)$--representing measure for $\beta$ exists. If the rank inequality $$\Rank \mc N(k)\leq \Rank \mc{M}(k;\beta)$$ holds, then there is a $(\Rank \mc{M}(k;\beta))$--atomic $\mc Z(p)$--representing measure;
    otherwise there is a $(\Rank \mc{M}(k;\beta)+1)$--atomic one.
\end{theorem}

\begin{remark}
The implication 
    $\eqref{100622-2207-pt1}\Rightarrow\eqref{100622-2206-pt1}$
of Theorem \ref{100622-2204} 
already follows from 
    \cite[Theorem 8.9]{KZ25+}, which characterizes all positive
polynomials on $\cZ(p)$, 
    together with 
    \cite[Proposition 2 and Corollary 6]{dDS18},
    which states that strictly positive  Riesz functional implies the existence of a $\cZ(p)$--rm.
    However, here we present a constructive proof for this implication, which also shows that a minimal measure is $3k$--atomic. 
\end{remark}

        As explained in Remark \ref{general-procedure}, the existence of a $\cZ(p)$--rm for $\beta$ is equivalent to the existence of a pair $(t_0,u_0)\in \RR^2$, such that
        $\cH(\cG(t_0,u_0))$ and $\cF(\mathcal G(t_0,u_0))$ admit a $\cZ(xy-1)$--rm and 
        $\RR$--rm, respectively.
        Note that
        the solutions to the
        $\cZ(xy-1)$--TMP
        and 
        $\RR$--TMP 
        are
        Theorems \ref{Hamburger} and \ref{090622-1611},
        respectively.
        For the forms of  $\cH(\cG(t_0,u_0))$, $\cF(\mathcal G(t_0,u_0))$
        and also in other  parts of  the proof of Theorem \ref{100622-2204},
        we need the following lemma.

\begin{lemma}
	\label{200922-2010}
    Assume the notation above.
    Let 
    $\vec\cT_1
        :=
            (\vec{Y}^{(k,1)},
            Y\vec{X}^{(1,k)})
    $
    and
    $\cM(\mathbf{Z},\mathbf t)$ be a function on $S_{k+1}\times \RR$ defined by
    \begin{equation}
        \label{def-M-Z-t}
        \cM(\mathbf{Z},\mathbf t)=
        \kbordermatrix{
            & \vec{\cT}_1 & \vec{X}^{(0,k)}\\
            (\vec{\cT}_1)^T & F_2  & 
                        \left(\begin{array}{cc}
                        B^{(0,k-1)}_{13} & b^{(k)}_{13} \\ 
                        B^{(0,k-1)}_{23} & b^{(k)}_{23} \\ 
                        c^T & \mathbf{t}
                \end{array}\right)
                \\[0.2em]
            (\vec{X}^{(0,k)})^T  &  \left(\begin{array}{ccc} 
                    (B^{(0,k-1)}_{13})^T & (B^{(0,k-1)}_{23})^T & c\\
                   (b^{(k)}_{13})^T & (b^{(k)}_{23})^T & \mathbf{t}
                    \end{array}\right)
                    & \mathbf{Z}}.
    \end{equation}
    Assume that there  exists $t_0\in \RR$ such that
        $\cM(B_{33},t_0)\succeq 0$.
    Let 
        \begin{equation}
            \label{def:Z0}
                    Z_0
                    :=
                    \begin{pmatrix}
                    (B^{(0,k-1)}_{13})^T & (B^{(0,k-1)}_{23})^T & c\\
                    (b^{(k)}_{13})^T & (b^{(k)}_{23})^T & t_0
                    \end{pmatrix}
                    (F_2)^{\dagger}
                    \left(\begin{array}{cc}
                        B^{(0,k-1)}_{13} & b^{(k)}_{13} \\ 
                        B^{(0,k-1)}_{23} & b^{(k)}_{23} \\ 
                        c^T & t_0
                \end{array}\right).
        \end{equation}
    Then the following statements hold:
		\begin{enumerate}	
			\item
				\label{200922-2010-pt1}
					$\cM\big(Z_0,t_0\big)\succeq 0$ 
					and 	
					$B_{33}-Z_0\succeq 0$.
			\item
				\label{200922-2010-pt2} 
					$\cM(Z_0,t_0)$
                    satisfies the column relations $YX^{i+1}=X^{i}$
                    for $i=0,\ldots,k-1$.
					Hence, 
				\begin{align*}
					Z_0
						&=\cG(t_0,u_0)\qquad \text{for some }u_0\in \RR.
                    \end{align*}    
			\item 	
				\label{200922-2010-pt3} 
					$\Rank \mc N(k)=\Rank F_{1}+\Rank F_2.$ 
		\end{enumerate}
\end{lemma}

\begin{proof}
    By the equivalence between \eqref{pt1-281021-2128} and \eqref{pt3-281021-2128} of Theorem \ref{block-psd}
	used for the pair $(M,A)=(\cM(\mathbf{Z},t_0),F_2)$, Lemma \ref{200922-2010}.\eqref{200922-2010-pt1} follows.

	Relations \eqref{110622-0029-v2} and definitions of $\beta_{2k-1,2}$,
	$\beta_{k,k+1}$ and $\beta_{2k,2}$
	imply that the restriction 
   $$
        \big(\cM(Z_0,t_0)\big)_{
            \vec{\cT}_1,(\vec{\cT}_1, \vec{X}^{(0,k)})}=
        \kbordermatrix{
            & \vec{\cT}_1 & \vec{X}^{(0,k)}\\
            (\vec{\cT}_1)^T &  F_2 & 
                        \left(\begin{array}{cc}
                        B^{(0,k-1)}_{13} & b^{(k)}_{13} \\ 
                        B^{(0,k-1)}_{23} & b^{(k)}_{23} \\ 
                        c^T & t_0
                \end{array}\right)}
    $$
	satisfies the relations $YX^{i+1}=X^{i}$
                    for $i=0,\ldots,k-1$. 
	By Lemma \ref{140722-1055}, the restriction 
   $$
        \big(\cM(Z_0,t_0)\big)_{
            \vec X^{(0,k)},(\vec{\cT}_1,\vec{X}^{(0,k)})}=
        \kbordermatrix{
            & \vec{\cT}_1 & \vec{X}^{(0,k)}\\
            (\vec{X}^{(0,k)})^T & \left(\begin{array}{ccc} 
                    (B^{(0,k-1)}_{13})^T & (B^{(0,k-1)}_{23})^T & c\\
                   (b^{(k)}_{13})^T & (b^{(k)}_{23})^T & t_0
                    \end{array}\right) & Z_0}.
    $$
	also satisfies the relations $YX^{i+1}=X^{i}$
                    for $i=0,\ldots,k-1$, 
	which proves Lemma \ref{200922-2010}.\eqref{200922-2010-pt2}.

     Permuting the rows and columns of $\cN(k)$ to the order
        $(\vec{\cT}_1,\vec{X}^{(0,k-1)})$,
    with a permutation matrix $P_1$,
    we get
        \begin{align}
        \label{def-N-k-perm}
		P_1\mc N(k)(P_1)^T
		&=
            \begin{pmatrix}
                F_2 & \left(\begin{array}{c} B^{(0,k-1)}_{13} \\ B^{(0,k-1)}_{23} \\ c^T\end{array}\right)\\[0.2em]
                \left(\begin{array}{ccc} (B^{(0,k-1)}_{13})^T & (B^{(0,k-1)}_{23})^T & c\end{array}\right)
                    & B^{(0,k-1)}_{33}
            \end{pmatrix}.
	\end{align}
    By Theorem \ref{block-psd}.\eqref{prop-2604-1140-eq2},  used for the pair $(M,A)=(P_1\mc N(k)(P_1)^T,F_2)$,
    noticing that 
	\begin{equation}
        \label{exp-F1}
			\big(P_1\mc N(k)(P_1)^T\big)/F_2=
                B^{(0,k-1)}_{33}-
                \begin{pmatrix}(B^{(0,k-1)}_{23})^T & c\end{pmatrix}
                =
                F_1,
        \end{equation}
    Lemma \ref{200922-2010}.\eqref{200922-2010-pt3} follows.
\end{proof}

\begin{remark}
\label{meaning-of-M-Z-t}
    Note that 
    the restriction of $\cM(B_{33},\mathbf{t})$
    to 
        $(\vec{Y}^{(k,1)}, Y\vec{X}^{(1,k-1)}, \vec{X}^{(0,k)})$
    is $\widehat\cM(k)_{\vec{\cT}}$, while to
        $(\vec{Y}^{(k,1)}, Y\vec{X}^{(1,k)}, \vec{X}^{(0,k-1)})$
    it is $P_1\cN(k)P_1^T$ with $P_1$ as in \eqref{def-N-k-perm}.
\end{remark}

Using Lemma \ref{200922-2010}, the existence of $t_0\in \RR$
such that $\cM(B_{33},t_0)\succeq 0$, implies that
	\begin{align}
        \label{def-H-G}
        \begin{split}
		\cH(\cG(t,u))
		&=
		B_{33}-
        \left(\begin{array}{cc}
            \left(\begin{array}{cc}(B^{(0,k-1)}_{23})^T & c\end{array}\right) & 
            \left(\begin{array}{c} b' \\ t\end{array}\right)\\ 
            \left(\begin{array}{cc}(b')^T & t\end{array}\right)  & u\end{array}\right)
            \\[0.5em]
            &=
		\kbordermatrix{
		& \vec{X}^{(0,k-1)} & X^{k}\\
		(\vec{X}^{(0,k-1)})^T & B^{(0,k-1)}_{33}-\left(\begin{array}{cc}(B^{(0,k-1)}_{23})^T & c\end{array}\right) &
			b^{(k)}_{33}-\left(\begin{array}{c} b' \\ t \end{array}\right)\\[0.5em]
			X^{k} & b^{(k)}_{33}-\left(\begin{array}{cc} (b')^T & t \end{array}\right) & \beta_{2k,0}-u}\\[0.5em]
		&=
		\kbordermatrix{
		& \vec{X}^{(0,k-1)} & X^{k}\\
		(\vec{X}^{(0,k-1)})^T & F_1 &
			b^{(k)}_{33}-\left(\begin{array}{c} b' \\ t \end{array}\right)\\[0.5em]
			X^{k} & b^{(k)}_{33}-\left(\begin{array}{cc} (b')^T & t \end{array}\right) & \beta_{2k,0}-u},
        \end{split}
	\end{align}	
	where
	\begin{equation}
        \label{181023-0931}
	b'=\begin{pmatrix} \beta_{k+1,1} & \beta_{k+2,1} & \cdots & \beta_{2k-1,1} \end{pmatrix}^T,
        \end{equation}
	and
	\begin{align*}
            \cF(\cG(t,u))_{(\vec Y^{(k,1)}, \vec{X}^{(0,k)})}	
            &=
        \kbordermatrix{
		& \vec{Y}^{(k,1)} & \vec X^{(0,k-1)} & X^k\\
		(\vec{Y}^{(k,1)})^T & B_{11} & B^{(0,k-1)}_{13} & b^{(k)}_{13}\\[0.5em]
		(\vec X^{(0,k-1)})^T & (B^{(0,k-1)}_{13})^T & \left(\begin{array}{cc}(B^{(0,k-1)}_{23})^T & c\end{array}\right) & \left(\begin{array}{c} b' \\ t \end{array}\right)\\
		X^k & (b^{(k)}_{13})^T & \left(\begin{array}{cc} (b')^T & t\end{array}\right) & u}\\[0.5em]
	&=
            \kbordermatrix{
            & (\vec{Y}^{(k,1)},\vec X^{(0,k-1)}) &  X^k \\
		(\vec{Y}^{(k,1)},\vec X^{(0,k-1)})^T &
			F_2 &  \left(\begin{array}{c} b^{(k)}_{13} \\b' \\ t \end{array}\right)\\
            X^k & 
			\left(\begin{array}{ccc}(b^{(k)}_{13})^T & (b')^T & t\end{array}\right) &  u}.
	\end{align*}
\medskip


\subsection{Proof of the implication $\eqref{100622-2206-pt1}\Rightarrow\eqref{100622-2206-pt2}$ of Theorem \ref{100622-2204}}
We denote by $\mc M^{(\mu)}(k+1)$ the moment matrix 
	associated to the sequence generated by some finitely atomic $\mc Z(p)$--rm $\mu$ for $\beta$, which exists by \cite{Ric57}.
	The following statements hold:
	\begin{itemize}
		\item The moment matrix $\mc M^{(\mu)}(k+1)$ is psd.
		\item The extension of ${\widehat{\mc M}(k)}_{\vec\cT\setminus \{X^k\}}$ with a row and column $YX^{k}$ is equal to the matrix $\mc N(k)$ due to the relation
			$Y^2X^k=YX^{k-1}$, which is satisfied by the moment matrix $\mc M^{(\mu)}(k+1)$. 
		\item The matrix $\mc N(k)$ is psd as the restriction of $\mc M^{(\mu)}(k+1)$.
	\end{itemize}
    \smallskip
    
    We separate two subcases.\\

    \noindent\textbf{Case 1: $\widehat{\mc M}(k)_{\vec\cT}$ and $\mc N(k)$ are positive definite.}
    This is Theorem \ref{100622-2204}.\eqref{100622-2207-pt1}.\\

    \noindent\textbf{Case 2: At least one of $\widehat{\mc M}(k)_{\vec\cT}$ and $\mc N(k)$ is not positive definite.}
    The restriction $\mc M^{(\mu)}(k+1)_{(\vec\cT,YX^k)}$ is of the form (see \eqref{def:M-with-R}, \eqref{def:N-with-R}) 
    $$
    \mc M^{(\mu)}(k+1)_{(\vec\cT, YX^k)}
    =
    \begin{pmatrix}
        R & m_{12} & n_{12}\\[0.2em]
        (m_{12})^T & \beta_{2k,0} & \beta_{2k,1}(\mu)\\[0.2em]
        (n_{12})^T & \beta_{2k,1}(\mu) & \beta_{2k-1,1}
    \end{pmatrix}.
    $$ 

	\noindent \textbf{Claim 1.} $\beta_{2k,1}(\mu)=t'$,
        where $t'$ is as in \eqref{200922-1158}.\\

	\noindent \textit{Proof of Claim 1.}
	By definition of $t'$ and 
        by Lemma \ref{psd-completion},
    used for 
    $$
    A(\mathbf{x})
    :=
    \begin{pmatrix}
        R & m_{12} & n_{12}\\[0.2em]
        (m_{12})^T & \beta_{2k,0} & \mathbf{x}\\[0.2em]
        (n_{12})^T & \mathbf{x} & \beta_{2k-1,1}
    \end{pmatrix},
    $$
    we have $A(t')\succeq 0$.
    We separate two cases according to invertibility of $\cN(k)$.\\

    \noindent \textbf{Case (i): $\mc N(k)$ is invertible.} 
    It follows that 
    $\Rank \widehat{\mc M}(k)_{\vec\cT}<\Rank \mc N(k)$ and hence by Lemma \ref{psd-completion}, there is no other $t\in \RR$ except $t'$ such that $A(t)\succeq 0$.
	Hence,  $\beta_{2k,1}(\mu)$ must be equal to $t'$.\\

    \noindent \textbf{Case (ii): $\mc N(k)$ is singular.}
    The singularity of $\cN(k)$ implies,
    by Lemma \ref{200922-2010}.\eqref{200922-2010-pt3},
    that
    \begin{equation}
        \label{singularity-F1-F2}
        \text{at least one of the matrices }F_1 \text{ and } F_2 \text{  is singular.}
    \end{equation}
    Let $\cM(\mathbf{Z},\mathbf t)$ be as in \eqref{def-M-Z-t}.
    Note that $\cM(B_{33},\mathbf{t})$ is obtained by permuting rows and columns of $A(\mathbf{t})$.
    In particular, $\cM(B_{33},t')\succeq 0$.
    Now let $t_0$ be any real number such that $\cM(B_{33},t_0)\succeq 0$.
    By Lemma \ref{200922-2010}, it follows that
    \begin{equation}
        \label{positivity-260625-0748}
            \cM(Z_0,t_0)\succeq 0\quad\text{and}\quad B_{33}-Z_0\succeq 0
    \end{equation}
    for $Z_0$ as in \eqref{def:Z0}, and
    \begin{equation}
    \label{form-of-F1}
        B_{33}-Z_0=
        \cH(\cG(t_0,u_0))=
            \begin{pmatrix}
            F_1 & \left(\begin{array}{c}\ast \\ t_0 \end{array}\right)\\
            \left(\begin{array}{cc}\ast & t_0 \end{array}\right) & \ast
            \end{pmatrix}.
    \end{equation}
    Now we separate possible cases in \eqref{singularity-F1-F2}.\\
    
    \noindent \textbf{Case (ii).(I): $F_1$ is singular.} 
	Since $B_{33}-Z_0$ is psd by \eqref{positivity-260625-0748} and has the form \eqref{form-of-F1},
    it follows, by \cite[Theorem 2.4(ii)]{CF91}, 
    that $t_0$ is uniquely determined by $F_1$. Hence, Claim 1 holds in this case.\\
	
    \noindent \textbf{Case (ii).(II): $F_2$ is singular.} 
	The restriction of $\cM(Z_0,t_0)$ to the rows and columns in 
    $(\vec{\cT}_1,X^k)$ is a psd Hankel matrix of the form
    \begin{equation}
    \label{form-of-F2}
        \cM(Z_0,t_0)_{(\vec{\mathcal T}_1,X^k)}=
            \begin{pmatrix}
            F_2 & \left(\begin{array}{c}\ast \\ t_0 \end{array}\right)\\
            \left(\begin{array}{cc}\ast & t_0 \end{array}\right) & \ast
            \end{pmatrix}.
    \end{equation}
    As in the Case (ii).(I), \cite[Theorem 2.4(ii)]{CF91} implies 
    that $t_0$ is uniquely determined by $F_2$. Hence, Claim 1 holds also in this case.\hfill$\blacksquare$\\

As explained in the paragraph following the statement of Theorem \ref{100622-2204},
    there exist $t_0,u_0\in \RR$ such that $\cF(\cG(t_0,u_0))$
    and $\cH(\cG(t_0,u_0))$ admit a $\cZ(xy-1)$--rm and a $\RR$--rm, respectively.
    By Claim 1, we have that $t_0=t'$. 
    Note that the right-lower corner of $Z_0$ is precisely $u''$ (see \eqref{200922-1158}).
    By definitions \eqref{200922-1158} of $u'$ and $u''$, $u'$ is the largest number such that $\cH(\cG(t',u'))\succeq 0$ and
    $u''$ is the smallest number such that $\cF(\cG(t',u''))\succeq 0$.
    In particular, $u''\leq u'$ and $u_0\in [u'',u']$.
    Note that $\cH(\cG(t',u'))=A_{\gamma_1(t',u')}$ 
    admits a $\RR$--rm by 
    Theorem \ref{Hamburger}, 
    since the last column is in the span of the previous ones.
    We have 
    \begin{align}
        \label{181023-1307}
    \begin{split}
    \cF(\cG(t',u'))_{\vec \cT}
    &=\cF\big(\cG(t',u_0)+(u'-u_0)E_{1,1}^{(k+1)}\big)_{\vec{\cT}}\\
    &=\cF(\cG(t',u_0))_{\vec\cT}+(u'-u_0)E_{2k+1,2k+1}^{(3k)}\\
    &\succeq \cF(\cG(t',u_0))_{\vec \cT}.
    \end{split}
    \end{align}
    Since $\cF(\cG(t',u_0))$  
	admits a $\cZ(xy-1)$--rm
    and 
    ${\cF}(\cG(t',u_0))_{(\vec Y^{(k,1)}, \vec{X}^{(0,k)})}
            =
            A_{\gamma_{2}(t',u_0)},$
    it follows that
    $\gamma_{2}(t',u_0)$ is $(\RR\setminus\{0\})$--representable.
    From now on we separate two cases according to the invertibility of $F_2$.\\

\noindent\textbf{Case 2.1: $F_2$ is invertible.}  
    We separate two cases according to the invertibility of 
    $A_{\gamma_{2}(t',u_0)}$.\\

\noindent\textbf{Case 2.1.1: $A_{\gamma_{2}(t',u_0)}\succ 0$.} 
    It follows that 
    $
        \Rank \cF(\cG(t',u_0))_{(\vec{Y}^{(k,1)},\vec{X}^{(0,k)})}=2k+1
    $ 
    and hence
    $
        \Rank \cF(\cG(t',u'))_{(\vec{Y}^{(k,1)}, \vec{X}^{(0,k)})}=2k+1
    $ 
    by \eqref{181023-1307}. 
    By Theorem \ref{thm:strong-Hamburger},
    $\gamma_2(t',u')$ is $(\RR\setminus \{0\})$--representable.\\

\noindent\textbf{Case 2.1.2: $A_{\gamma_{2}(t',u_0)}\succeq 0$ and 
$A_{\gamma_{2}(t',u_0)}\not\succ 0$ .} 
    It follows that $A_{\gamma_{2}(t',u_0)}$  satisfies 
        Theorem \ref{thm:strong-Hamburger}.\eqref{thm:strong-Hamburger-pt2-b}  
    and hence
    \begin{align}
    \label{281023-1738}
    \begin{split}
    2k
    &=\Rank \cF(\cG(t',u_0))_{(\vec{Y}^{(k,1)}, \vec{X}^{(0,k-1)})}
    =\Rank\cF(\cG(t',u_0))_{(\vec{Y}^{(k-1,1)},\vec{X}^{(0,k)})},
    \end{split}
    \end{align}
    where we also used invertibility of $F_2$ in the first equality.
    If $u_0=u'$, \eqref{281023-1738} implies that
    $\gamma_2(t',u')$ is $(\RR\setminus \{0\})$--representable. 
    Otherwise $u'>u_0$ and \eqref{181023-1307}, \eqref{281023-1738}
    imply that 
        $$\Rank \cF(\cG(t',u'))_{(\vec{Y}^{(k,1)}, \vec{X}^{(0,k)})}=2k+1,$$
    which again implies that $\gamma_2(t',u')$ is $(\RR\setminus \{0\})$--representable 
    by Theorem \ref{thm:strong-Hamburger}.\\

\noindent\textbf{Case 2.2: $F_2$ is singular.}  
    If $\gamma_2(t',u_0)$ is $(\RR\setminus \{0\})$--representable,
    then in particular it is $\RR$--representable.
	But due to singularity of $F_2$, $u''$ is the only candidate
    for $u_0$
    by Theorem \ref{Hamburger}.
    Hence, $\gamma_1(t',u'')$ is also $\RR$--representable.\\

    This concludes the proof of the implication $\eqref{100622-2206-pt1}\Rightarrow\eqref{100622-2206-pt2}$ of Theorem \ref{100622-2204}.


\subsection{Proof of the implication $\eqref{100622-2206-pt2}\Rightarrow\eqref{100622-2206-pt1}$ of Theorem \ref{100622-2204}}
	We separate two cases according to the assumptions in 
		\eqref{100622-2206-pt2}.
\bigskip

\noindent 
	\textbf{Case 1: \eqref{100622-2207-pt1} of Theorem \ref{100622-2204} holds.} By Lemma \ref{psd-completion}, used for $A(\mathbf{x})=\cM(B_{33},\mathbf x)$ (as in \eqref{def-M-Z-t}), there exist $t_\ell\in\RR$, $\ell=1,2$, such that (see also Remark \ref{meaning-of-M-Z-t}):
	\begin{align}
        \label{120622-0905}
        \begin{split}
			&\cM(B_{33},t)\succeq 0\quad \text{for every } t\in [t_1,t_2],\\
			&\Rank \cM(B_{33},t_\ell)=\Rank {\mc M}(k)=\Rank {\mc N}(k)\quad \text{for } \ell=1,2,\\			
			&\Rank \cM(B_{33},t)=\Rank 
   {\mc M}(k)+1=\Rank \mc N(k)+1\quad \text{for }t\in (t_1,t_2).
        \end{split}
	\end{align}
	Let $t_0\in [t_1,t_2]$. 
    By Lemma \ref{200922-2010}.\eqref{200922-2010-pt1}, 
	we have $\cM(Z_0,t_0)\succeq 0$,
    where $Z_0$ is as in \eqref{def:Z0}.
    By Theorem \ref{block-psd}.\eqref{prop-2604-1140-eq2},
	used for the pair 
    $(M,A)=
    \big(
    \cM(Z_0,t_0),F_2
    \big),
    $
    we have  
	\begin{equation}\label{161121-1403}
		\Rank \cM(Z_0,t_0)=\Rank F_2.
	\end{equation}
 	By Lemma \ref{200922-2010}.\eqref{200922-2010-pt2}, 
	$$
        \cC\big(\cM(Z_0,t_0)\big)
        =
        \cC\big(\cM(Z_0,t_0)_{(\vec{Y}^{(k,1)}, Y\vec{X}^{(1,k)})}\big)
        =
        \cC\big(\cM(Z_0,t_0)_{(\vec{Y}^{(k,1)},\vec{X}^{(0,k-1)})}\big),
        $$
	which implies that
	\begin{equation}\label{110622-2114}
		\Rank \big(\cM(Z_0,t_0)_{(\vec{Y}^{(k,1)},\vec{X}^{(0,k-1)})}\big)
            =
            \Rank \cM(Z_0,t_0).
	\end{equation}
	By Lemma \ref{200922-2010}.\eqref{200922-2010-pt2}, it follows that 
	$Z_0=\cG(t_0,u(t_0))$ for some $u(t_0)\in \RR$
	and
    by Lemma \ref{200922-2010}.\eqref{200922-2010-pt1},
	$B_{33}\succeq Z_0$. Hence, $\cH(Z_0)\succeq 0$. 
    By \eqref{exp-F1}, 
		$$
        F_1=B^{(0,k-1)}_{33} - \begin{pmatrix}(B^{(0,k-1)}_{23})^T & c\end{pmatrix}
        \succ 0.$$
	By the equivalence between \eqref{pt1-281021-2128} and \eqref{pt3-281021-2128} of Theorem \ref{block-psd},
	used for the pair 
    $
    \big(\cH(\cG(t_0,u(t_0))),F_1\big)
    $
    (see \eqref{def-H-G} for the $2\times 2$ block decomposition of $\cH(\cG(t_0,u(t_0)))$,
    it follows that
	\begin{equation}\label{110622-2125}
		\delta_0:=(\beta_{2k,0}-u(t_0))-
		\Big(b^{(k)}_{33}-
            \begin{pmatrix} b' \\ t_0 \end{pmatrix}\Big)^T
            (F_1)^{-1}\Big(b^{(k)}_{33}-\begin{pmatrix} b' \\ t_0 \end{pmatrix}\Big)\geq 0
	\end{equation}
	and
	\begin{align}\label{110622-2128}
        \begin{split}
		\Rank \cH\big(\cG(t_0,u(t_0))\big)
            &=
			\left\{\begin{array}{rl} 
				\Rank F_1,&	\text{if }\delta_0=0,\\[0.3em]
				\Rank F_1+1,&	\text{if }\delta_0>0.
			\end{array} \right.
            =
            \left\{
            \begin{array}{rl}
                 k,& \text{if }\delta_0=0,\\[0.2em]
                 k+1,& \text{if }
                 \delta_0>0.
            \end{array}
            \right.
        \end{split}
	\end{align}
	By 
        Theorem \ref{block-psd}.\eqref{prop-2604-1140-eq2}, used for the pair
        $(M,A)=\big(\cM(B_{33},t_0),F_2\big)$,
	we have 
	\begin{equation}\label{110622-2129}
            \Rank \cM(B_{33},t_0)=
            \Rank F_2
            +
			\Rank \cH\big(\cG((t_0,u(t_0)))\big).
	\end{equation}	
	Since $\Rank F_2=2k$ by the invertibility of $\mc N(k)$ and $\Rank \cM(B_{33},t_\ell)=3k$, $\ell=1,2$, by \eqref{120622-0905},
	it follows that 
	\begin{equation}
		\label{200922-1347}
			\Rank \cH\big(\cG(t_\ell,u(t_\ell))\big)=k,\quad \ell=1,2.
	\end{equation}
	Note that 
	$A_{\gamma_1(t_\ell,u(t_{\ell}))}=\cH\big(\cG(t_\ell,u(t_\ell))\big)$, $\ell=1,2$, 
	where $\gamma_1(\mathbf t,\mathbf u)$ is as in \eqref{def:gammas-v2}.
	By \eqref{110622-2128} and \eqref{200922-1347},
		$\Rank A_{\gamma_1(t_\ell,u(t_\ell))}(k-1)=\Rank A_{\gamma_1(t_\ell,u(t_\ell))}$, $\ell=1,2$.
	By Theorem \ref{Hamburger}, 
    ${\gamma_1(t_\ell,u(t_\ell))}$, $\ell=1,2$, 
	has a $k$--atomic $\RR$--rm. 
	Note that in $\cF\big(\cG(t_\ell,u(t_\ell))\big)$, $\ell=1,2$, 
	\begin{equation}
		\label{200922-1402} 
			X^k\in \Span\{
                Y^k,Y^{k-1},\ldots,Y,\textit{1},X,X^2,\ldots,X^{k-1}\}.
	\end{equation}
    Further,
	$$
\cF(\cG(t_{\ell},u(t_{\ell})))_{(\vec{Y}^{(k,1)},\vec{X}^{(0,k)})}=
        A_{\gamma_2(t_\ell,u(t_\ell))},
        $$
        where $\gamma_2(\mathbf t,\mathbf u)$ is as in \eqref{def:gammas-v2}. 
	Since \eqref{200922-1402} holds, the sequences $\gamma_2(t_\ell,u(t_\ell))$, $\ell=1,2$, are $\RR$--representable.
	Since 
		$$F_2=A_{\gamma_2(t_1,u(t_1))}(2k-1)=A_{\gamma_2(t_2,u(t_2))}(2k-1)$$ 
	is invertible, 
	by \cite[Proposition 2.5]{Zal22b}, at least one of $\gamma_2(t_1,u(t_1))$ or $\gamma_2(t_2,u(t_2))$ 
	is $(\RR\setminus\{0\})$--representable.
	By Lemma \ref{071123-0646}, $\beta$ has a $(3k)$--atomic $\cZ(p)$--rm, which concludes the proof of the implication $\eqref{100622-2206-pt2}\Rightarrow\eqref{100622-2206-pt1}$ in this case.\\

\noindent
	\textbf{Case 2: \eqref{200922-1200-pt2-alt} of Theorem \ref{100622-2204} holds.} By Lemma \ref{psd-completion}, used for $A(\mathbf{x})=\cM(B_{33},\mathbf{x})$, for $t'$ as in \eqref{200922-1158}, we have $\cM(B_{33},t')\succeq 0$ and 
		$$\Rank \cM(B_{33},t')=\max\big(\Rank \widehat{\mc M}(k),\Rank \mc N(k)\big).$$
	By Lemma \ref{200922-2010}, it follows that $A_{\gamma_2(t',u'')}=\cF(\cG(t',u''))_{(\vec{Y}^{(k,1)},\vec{X}^{(0,k)})}$
	and $A_{\gamma_1(t',u'')}=\cH\big(\cG(t',u'')\big)$ are psd, and
	\begin{equation}
        \label{181023-1316}
		\Rank \cM(B_{33},t')=\Rank A_{\gamma_2(t',u'')}+\Rank A_{\gamma_{1}(t',u'')}.
	\end{equation}
	We separate two subcases according to the invertibility of $F_2$.\\

 \noindent\textbf{Case 2.1: $F_2$ is invertible.}
 Note that $\cG(t',u'')$ is equal to $Z_0$ from \eqref{def:Z0} with $t_0=t'$.
    By definition \eqref{200922-1158}, $u'$ is the largest such that 
    $\cH(\cG(t',u'))\succeq 0$.
    Thus, $u'\geq u''$. 
    We have
    $\cF(\cG(t',u'))\succeq \cF(\cG(t',u''))$ (see the inequality \eqref{181023-1307} above) and
    \begin{align}
    \label{181023-1319}
    \begin{split}
    \Rank\cF(\cG(t',u'))
    &=
    \left\{
        \begin{array}{rl}
            \Rank\cF(\cG(t',u'')),& \text{if }u'=u'',\\[0.3em]
            \Rank\cF(\cG(t',u''))+1,& \text{if }u'>u'',
        \end{array}
    \right.
    \\[0.3em]
    &=
    \left\{
        \begin{array}{rl}
            2k,& \text{if }u'=u'',\\[0.3em]
            2k+1,& \text{if }u'>u'',
        \end{array}
    \right.
    \end{split}
    \end{align}
    where we used the fact that $F_2$ is invertbile in the second equality.
    Note that
    $$A_{\gamma_2(t',u)}=\cF(\cG(t',u))_{(\vec{Y}^{(k,1)},\vec{X}^{(0,k)})}.$$
    If $u'=u''$, then by definition of $u''$, we also have 
    $\Rank A_{\gamma_2(t',u')}(2k-1)=\Rank A_{\gamma_2(t',u')}$.
    Otherwise $u'>u''$ and $\Rank A_{\gamma_2(t',u')}=2k+1$.
	Since $A_{\gamma_1(t',u')}=\cH(\cG(t',u'))$ satisfies the equality $\Rank A_{\gamma_1(t',u')}=\Rank A_{\gamma_1(t',u')}(k-1)$ by definition of $u'$,
	it admits a $(\Rank A_{\gamma_1(t',u')})$--atomic $\RR$--rm
    by Theorem \ref{Hamburger}. 
    Using \eqref{181023-1316} and in the case $u'>u''$ also 
    rank equalities
    $\Rank A_{\gamma_2(t',u'')}=\Rank A_{\gamma_2(t',u')}-1$ (by \eqref{181023-1319})
    and
    $\Rank A_{\gamma_{1}(t',u'')}=\Rank A_{\gamma_{1}(t',u')}+1$ (by definition of $u'$),
    it follows that
	$\beta$ admits a $(\Rank \cM(B_{33},t'))$--atomic $\cZ(p)$--rm.
    This proves the implication $\eqref{100622-2206-pt2}\Rightarrow\eqref{100622-2206-pt1}$ in this case.\\
    
 \noindent\textbf{Case 2.2: $F_2$ is singular.}
    Note that
    $A_{\gamma_2(t',u'')}=\cF(\cG(t',u''))_{(\vec{Y}^{(k,1)},\vec{X}^{(0,k)})}$ satisfies 
    the equality $\Rank A_{\gamma_2(t',u'')}(2k-1)=\Rank A_{\gamma_2(t',u'')}$
    by definition of $u''$.
Moreover, $\gamma_1(t',u'')$ 
admits a 
    $(\Rank A_{\gamma_1(t',u'')})$--atomic $\RR$--rm. 
    Since also
    $\Rank \cM(B_{33},t')=\Rank\cF(\cG(t',u''))+\Rank A_{\gamma_1(t',u'')}$,
    it follows that $\beta$ admits a
	$(\Rank \cM(B_{33},t'))$--atomic $\mc Z(p)$--rm.\\

This concludes the proof of the implication $\eqref{100622-2206-pt2}\Rightarrow\eqref{100622-2206-pt1}$
of Theorem \ref{100622-2204}. Note also that
the moreover part of the theorem follows from the proof of this implication.

\subsection{Example}{\footnote{The \textit{Mathematica} file with numerical computations can be found on the link \url{https://github.com/ZalarA/TMP_cubic_reducible}.}} 
\label{ex:hyperbolic-type-1}
The sequence $\beta$ is said to be \textbf{$p$--purely pure}, if it is $p$--pure and also the matrix $\mc N(k)$ is invertible.
By Theorem \ref{100622-2204}.\eqref{100622-2207-pt1}, a $p$--purely pure sequence $\beta$, such that $\mc M(k;\beta)$ and $\mc N(k)$ are psd,
admits a $\cZ(p)$--rm. The following example shows that, in contrary to the TMP on the union of three parallel lines \cite{Zal22a}, 
in this hyperbolic type case, a $p$--pure sequence $\beta,$
such that $\mc M(k;\beta)$ and $\mc N(k)$ are psd, does not necessarily admit a $\cZ(p)$--rm.

Let $\beta$ be a bivariate degree 6 sequence given by
\begin{spacing}{1.3}

$\beta_{00} = 1$,

$\beta_{10} =\frac{3}{4}$,
$\beta_{01}  = 0$

$\beta_{20}  = 3,$
$\beta_{11} = \frac{1}{2},$
$\beta_{02}  =\frac{5}{16},$

$\beta_{30}  =\frac{9}{2},$
$\beta_{21}  =0,$
$\beta_{12}  =0,$
$\beta_{03}  =0,$

$\beta_{40}  =\frac{17}{64},$
$\beta_{31}  =\frac{5}{4},$
$\beta_{22}  =\frac{1}{2},$
$\beta_{13}  =\frac{5}{16},$
$\beta_{04}  =\frac{17}{64},$

$\beta_{50} =\frac{69}{2},$
$\beta_{41} =0,$
$\beta_{32} =0,$
$\beta_{23} =0,$
$\beta_{14} =0,$
$\beta_{05} =0,$

$\beta_{60}=\frac{231}{2},$
$\beta_{51}=\frac{17}{4},$
$\beta_{42}=\frac{5}{4},$
$\beta_{33}=\frac{1}{2},$
$\beta_{24}=\frac{5}{16},$
$\beta_{15}=\frac{17}{64},$
$\beta_{06}=\frac{81}{256}.$
\end{spacing}
	We will prove below that $\beta$ does not have a $\RR^2$--rm.
	It is easy to check that $\widehat{\mc M}(3)$ is psd
	and satisfies only one column relation $Y^2X=Y$, while the matrix $\mc N(3)$ 
    is 
    psd 
	and has only one column relation $YX^3=5YX-4Y^2.$
The sequences 
    $\gamma_1(\mathbf t,\mathbf u)$ and 
    $\gamma_2(\mathbf t,\mathbf u)$
    (see \eqref{def:gammas-v2}) are equal to
    \begin{align*}
    \gamma_1(\mathbf t,\mathbf u)
    &=
        \Big(
        \frac{1}{2},\frac{3}{4},\frac{7}{4},\frac{9}{2},
        \frac{49}{2},\frac{69}{2}-\mathbf{t},\frac{231}{2},-\mathbf{u}
        \Big),\\
    \gamma_2(\mathbf t,\mathbf u)
    &=
        \Big(
        \frac{81}{256},0,\frac{17}{64},0,\frac{5}{16},0,\frac{1}{2},
        0,\frac{5}{4},0,\frac{17}{4},\mathbf{t},\mathbf{u}
        \Big).
    \end{align*}
Computing $t',u',u''$ (see \eqref{200922-1158}) we get
$t'=0$, $u'=\frac{659}{40}$ and $u''=\frac{65}{4}$.
Since $A_{\gamma_{2}(t',u')}$ satisfies 
$\Rank A_{\gamma_{2}(t',u')}=6$ and
$\Rank A_{\gamma_{2}(t',u')}(6)=\Rank A_{\gamma_{2}(t',u')}[6]=5$,
Theorem \ref{thm:strong-Hamburger} implies that $\gamma_{2}(t',u')$ is not $(\RR\setminus\{0\})$--representable.
Since $A_{\gamma_{2}(t',u'')}$ satisfies 
$\Rank A_{\gamma_{2}(t',u'')}=5$ and
$\Rank A_{\gamma_{2}(t',u'')}[6]=4$,
Theorem \ref{thm:strong-Hamburger} implies that 
$\gamma_{2}(t',u'')$ is not $(\RR\setminus\{0\})$--representable.
So neither of $\gamma_{2}(t',u')$ or $\gamma_{2}(t',u'')$
is $(\RR\setminus\{0\})$--representable,
which implies, 
    by Theorem \ref{100622-2204}, 
that $\beta$ does not admit a $\mc Z(p)$--rm.

%% file: Hyperbolic-type-2-new.tex
\section{Hyperbolic type 2 relation: 
            $p(x,y)=y(x+y-xy)$
        }
\label{sec:hyperbolic-type-2}

In this section we solve constructively the $\cZ(p)$--TMP for 
the sequence $\beta=\{\beta_{i,j}\}_{i,j\in \ZZ_+,i+j\leq 2k}$ 
of degree $2k$, $k\geq 3$,
where $p(x,y)$ is as in the title of the section.
The main results are Theorem \ref{sol:x-y-axy}, which characterizes concrete numerical conditions for the existence of a $\cZ(p)$--rm
for $\beta$
and Theorem \ref{thm:hyperbolic-2-minimal-measures},
which characterizes the number of atoms needed in a minimal $\cZ(p)$--rm. 
A numerical example demonstrating the main results is presented in Subsection \ref{ex:hyperbolic-type-2}.

\begin{remark}
In the classification from
\cite[Proposition 3.1]{YZ24}, in the hyperbolic type 2 relation,  $c(x,y)$ is equal to $x+y+axy$, $a\in \RR\setminus \{0\}$. However, after applying an affine linear transformation (see Subsection \ref{sub:affine-linear-trans}) $\phi(x,y)=(-ax,-ay)$ we can assume that $a=-1$.
\end{remark}

\subsection{Existence of a representing measure}
Assume the notation from Section \ref{Section-common-approach}.
If $\beta$ admits a $\cZ(p)$--TMP, then $\mc M(k;\beta)$ 
must satisfy the relations
\begin{equation}
    \label{130623-0758-v2}
	Y^{2+j}X^{1+i}=Y^{1+j}X^{1+i}+Y^{2+j}\quad
        \text{for }i,j\in \ZZ_+\text{ such that }i+j\leq k-3.
\end{equation}
On the level of moments the relations \eqref{130623-0758-v2} mean that
	\begin{equation}\label{130623-0758-v2}
		\beta_{i+1,j+2}=\beta_{i+1,j+1}+\beta_{i,2+j}
  \quad
  	\text{for }
   i,j\in \ZZ_+
   \text{ such that }i+j\leq 2k-3.
	\end{equation}
In the presence of all column relations \eqref{130623-0758-v2}, the column space $\cC(\mc M(k;\beta))$ is spanned by the columns in the tuple
    \begin{equation}
    \label{171023-2028-v3}
    \vec\cT:=(
        \underbrace{Y^k,Y^{k-1},\ldots,Y}_{\vec Y^{(k,1)}},
        \underbrace{
        YX-Y,YX^2-YX,\ldots,YX^{k-1}-YX^{k-2}}_{\vec{\cT}_2},
        \vec{X}^{(0,k)}).
    \end{equation}
    where 
    $\vec{X}^{(i,j)}:=(X^i,X^{i+1},\ldots,X^j)$, $0\leq i\leq j\leq k$ and $X^0:=\textit{1}$.
      Let
    \begin{align*}
    \label{150925-2258}
    \begin{split}
    &P
    \text{ be a permutation matrix such that moment matrix }
    \widehat{\mc M}(k):=P\mc M(k; \beta)P^T\text{ has rows}\\
    &\text{and columns indexed in the order }
	\vec{\cT},\vec{\cC}\setminus\vec{\cT}.
    \end{split}
    \end{align*}
Let $\widehat{\mc{M}}(k)_{\vec \cT}$
    be the restriction of the moment matrix 
    $\widehat{\mc{M}}(k)$
    to the rows and columns in the tuple $\vec\cT$:
	\begin{align*}
		\widehat{\mc{M}}(k)_{\vec\cT}
		&:=
		\kbordermatrix{
		& \vec{Y}^{(k,1)} &  \vec{\cT}_2 &   \vec{X}^{(0,k)}\\
			(\vec{Y}^{(k,1)})^T & B_{11} &  B_{12} &  B_{13}\\[0.2em]	
			(\vec{\cT}_2)^T & (B_{12})^T & B_{22} &  B_{23}\\[0.2em]
			(\vec{X}^{(0,k)})^T & (B_{13})^T & (B_{23})^T & B_{33}}\\[0.5em]
		&=		
		\kbordermatrix{
		& \vec{Y}^{(k,1)} & \vec{\cT}_2 & \vec{X}^{(0,k-1)} & X^k   \\[0.5em]
			(\vec{Y}^{(k,1)})^T & B_{11} & B_{12} & B^{(0,k-1)}_{13} & b^{(k)}_{13} \\[0.5em]
			(\vec{\cT}_2)^T & ({B}_{12})^T & {B}_{22} & {B}^{(0,k-1)}_{23} & {b}^{(k)}_{23} \\[0.5em]
			(\vec{X}^{(0,k-1)})^T & ({B}^{(0,k-1)}_{13})^T & ({B}^{(0,k-1)}_{23})^T & 
                (B_{33}^{(0,k-1)})
                 & {b}^{(k)}_{33} \\[0.5em]
			X^{k} & (b_{13}^{(k)})^T & ({b}^{(k)}_{23})^T & (b^{(k)}_{33})^T & \beta_{2k,0}}\\[0.5em]
		&=
		\kbordermatrix{
		& \vec{Y}^{(k,1)} & \vec{\cT}_2 & \textit{1} &\vec{X}^{(1,k-1)} & X^k  \\[0.5em]
			(\vec{Y}^{(k,1)})^T & B_{11} & B_{12} & b^{(0)}_{13} & B_{13}^{(1,k-1)} & b_{13}^{(k)} \\[0.5em]
			(\vec{\cT}_2)^T & ({B}_{12})^T & {B}_{22} & {b}^{(0)}_{23} & B_{23}^{(1,k-1)} & {b}^{(k)}_{23} \\[0.5em]
            \textit{1} & (b^{(0)}_{13})^T & (b_{23}^{(0)})^T & \beta_{0,0} & (b^{(1,k-1)}_{33;0})^T & \beta_{k,0} \\[0.5em]
			(\vec{X}^{(1,k-1)})^T & ({B}^{(1,k-1)}_{13})^T & ({B}^{(1,k-1)}_{23})^T & 
                b_{33;0}^{(1,k-1)} 
                 & B^{(1,k-1)}_{33} & b_{33;k}^{(1,k-1)} \\[0.5em]
			X^{k} & (b_{13}^{(k)})^T & ({b}^{(k)}_{23})^T & \beta_{k,0} & (b_{33;k}^{(1,k-1)})^T & \beta_{2k,0}}.
	\end{align*}
Let $\widetilde \cM(k;\beta)$ 
    be as in 
    \eqref{071123-1939} with 
    $\vec{\cT}_1:=(\vec{Y}^{(k,1)},\vec{\cT}_2)$
    and define
    \begin{equation}
        \label{091123-0719-v2}
            A_{\min}:=A_{12}(A_{22})^{\dagger} (A_{12})^T
        \quad\text{and}\quad
            \widehat A_{\min}
            :=A_{\min}+\eta \big(E_{1,k+1}^{(k+1)}+E_{k+1,1}^{(k+1)}\big),
    \end{equation}
    where 
    $
    \eta:=(A_{\min})_{2,k}-(A_{\min})_{1,k+1}.
    $
    See Remark \ref{general-procedure} for the explanation of these definitions.    
    Let
    $\cF(\mathbf{A})$ and $\cH(\mathbf{A})$ 
    be as in 
    \eqref{071123-1940}.
  Define the matrix function 
    \begin{equation}
        \label{301023-1930-v2}
        \mc G:[0,\infty)^2\to S_{k+1},\qquad 
	\mc G(\mathbf{t},\mathbf{u})=
        \widehat A_{\min}
        +\mathbf{t}E_{1,1}^{(k+1)}
        +\mathbf{u}E_{k+1,k+1}^{(k+1)}.
    \end{equation}
 Next we define the sequences $\gamma_1(\mathbf{t},\mathbf{u})$,
 $\gamma_2(\mathbf{t},\mathbf{u})$:
					\begin{align}
                        \label{def:gammas}
                        \begin{split}
						\gamma_1(\mathbf{t},\mathbf u)
						&:=
                            \big(
                                \beta_{0,0}-(A_{\min})_{1,1}-\mathbf{t},
                                \beta_{1,0}-\beta_{1,1}+\beta_{0,1},
                                \beta_{2,0}-\beta_{2,1}+\beta_{1,1},
                                \ldots,\\
                            &\hspace{3cm}
                                \beta_{2k-1,0}-\beta_{2k-1,1}+\beta_{2k-2,1},
                                \beta_{2k,0}-(A_{\min})_{k+1,k+1}-\mathbf{u}
                            \big),\\
						\gamma_2(\mathbf{t},\mathbf u)
						&:= 
                            \big(
                                (A_{\min})_{1,1}+\mathbf{t},
                                \beta_{1,1}-\beta_{0,1},
                                \beta_{2,1}-\beta_{1,1},
                                \ldots,
                                \beta_{2k-1,1}-\beta_{2k-2,1},\\
                        &\hspace{3cm}
                                (A_{\min})_{k+1,k+1}+\mathbf{u}
                            \big).
                        \end{split}
					\end{align}
Observe that
	\begin{align}
        \label{def-H-G-v2}
        \begin{split}
		&\cH(\cG(\mathbf t,\mathbf u))=\\
		=&
        \kbordermatrix{
            & \textit{1} & \vec{X}^{(1,k-1)} & X^k \\
            \textit{1} & \beta_{0,0}-(A_{\min})_{1,1}-\mathbf t & (b_{33;0}^{(1,k-1)})^T-(b_{23}^{(0)})^T & \beta_{k,0}-\beta_{k,1}+\beta_{k-1,1} \\[0.5em]
            (\vec X^{(1,k-1)})^T & b_{33;0}^{(1,k-1)}-b_{23}^{(0)} & B_{33}^{(1,k-1)}-B_{23}^{(1,k-1)} & b_{33;k}^{(1,k-1)}-b_{23}^{(k)} \\[0.5em]
            X^k & \beta_{k,0}-\beta_{k,1}+\beta_{k-1,1} & (b_{33;k}^{(1,k-1)})^T-(b_{23}^{(k)})^T & \beta_{2k,0}-(A_{\min})_{k+1,k+1}-\mathbf u
        }\\[0.3em]
        =&
            A_{\gamma_1{(\mathbf t,\mathbf u)}}
        \end{split}
	\end{align}	
	and
	\begin{align}
        \label{def-F-v2}
        \begin{split}
		&{\cF}(\cG(\mathbf t,\mathbf u))_{(\vec Y^{(k,1)},\vec{X}^{(0,k)})}	\\[0.5em]
        =&
        \kbordermatrix{
		& \vec{Y}^{(k,1)} & \textit{1} & \vec X^{(1,k-1)} & X^k\\
		(\vec{Y}^{(k,1)})^T & B_{11} & b_{13}^{(0)} & B_{13}^{(1,k-1)} & b_{13}^{(k)}\\[0.5em]
        \textit{1} & (b_{13}^{(0)})^T & (A_{\min})_{1,1}+\mathbf t & (b_{23}^{(0)})^T & \beta_{k,1}-\beta_{k-1,1}\\[0.5em]
		(\vec X^{(1,k-1)})^T & (B_{13}^{(1,k-1)})^T & b_{23}^{(0)}  & B_{23}^{(1,k-1)} & b_{23}^{(k)}\\[0.5em]
		X^k & (b_{13}^{(k)})^T & \beta_{k,1}-\beta_{k-1,1} &  (b_{23}^{(k)})^T & (A_{\min})_{k+1,k+1}+\mathbf u}\\[0.5em]
	   =&
        \kbordermatrix{
		& \vec{Y}^{(k,1)} & \vec X^{(0,k)} \\
		(\vec{Y}^{(k,1)})^T & B_{11} & B_{13}\\[0.5em]
		(\vec X^{(0,k)})^T & (B_{13})^T & A_{\gamma_{2}(\mathbf t,\mathbf u)}}.
        \end{split}
	\end{align}

By Lemmas \ref{071123-0646}--\ref{071123-2008} and Remark \ref{general-procedure}, the existence of a $\mc Z(p)$--rm for $\beta$ is equivalent to: 
\begin{align}
\label{measure-cond-hyp-2}
\begin{split}
&\widetilde{\cM}(k;\beta)\succeq 0, 
\text{ the relations }
\eqref{130623-0758-v2}\text{ hold 
and }\\
&\text{there exists } (\tilde t_0,\tilde u_0)\in \RR^2
\text{ such that }
\cF(\mc G(\tilde t_0,\tilde u_0)) 
\text{ and } \cH(\mc G(\tilde t_0,\tilde u_0))\\
&
\text{ admit a }
\cZ(x+y-xy)\text{--rm and a }\RR\text{--rm, respectively.}
\end{split}
\end{align}

We also write
        \begin{align}
        \label{141123-0720}
        \begin{split}
        &\cH(\widehat A_{\min})=\\
        =&
        \kbordermatrix{
            & \textit{1} & \vec{X}^{(1,k-1)}& X^k \\
            \textit{1} & \beta_{0,0}-(A_{\min})_{1,1}& (b_{33;0}^{(1,k-1)})^T-(b_{23}^{(0)})^T & \beta_{k,0}-(A_{\min})_{k+1,1}\\[0.5em]
            \vec X^{(1,k-1)} & b_{33;0}^{(1,k-1)}-b_{23}^{(0)} & B_{33}^{(1,k-1)}-B_{23}^{(1,k-1)} & b_{33;k}^{(1,k-1)}-b_{23}^{(k)} \\[0.5em]
            X^k & \beta_{k,0}-(A_{\min})_{k+1,1} & (b_{33;k}^{(1,k-1)})^T-(b_{23}^{(k)})^T & \beta_{2k,0}-(A_{\min})_{k+1,k+1}
        }\\[0.3em]
        =:&
        \kbordermatrix{
            & \mathit{1} & \vec{X}^{(1,k-1)} & X^k\\
           \mathit{1} & \beta_{0,0}-(A_{\min})_{1,1} & (h_{12})^T & \beta_{k,0}-(A_{\min})_{k+1,1}\\[0.2em]
          (\vec{X}^{(1,k-1)})^T &  h_{12} & H_{22} & h_{23} \\[0.2em]
          X^k & \beta_{k,0}-(A_{\min})_{k+1,1} & (h_{23})^T & \beta_{2k,0}-(A_{\min})_{k+1,k+1}
        },
        \end{split}
        \end{align}
and
\begin{align}
\label{definition-of-K}
\begin{split}
K
&:=
\cH(\widehat A_{\min})\big/H_{22}\\[0.3em]
&=
\begin{pmatrix}
                \beta_{0,0}-(A_{\min})_{1,1} &  \beta_{k,0}-(A_{\min})_{2,k}\\[0.2em]
            \beta_{k,0}-(A_{\min})_{2,k} & \beta_{2k,0}-(A_{\min})_{k+1,k+1}.
\end{pmatrix}
-
\begin{pmatrix}
    (h_{12})^T\\
    (h_{23})^T
\end{pmatrix}
(H_{22})^\dagger
\begin{pmatrix}
    h_{12} & h_{23}
\end{pmatrix}\\[0.3em]
&:=
\begin{pmatrix}
                \beta_{0,0}-(A_{\min})_{1,1}-(h_{12})^T(H_{22})^\dagger h_{12}&  \beta_{k,0}-(A_{\min})_{2,k}-(h_{12})^T(H_{22})^\dagger h_{23}\\[0.3em]
            \beta_{k,0}-(A_{\min})_{2,k}-(h_{23})^T(H_{22})^\dagger h_{12} & \beta_{2k,0}-(A_{\min})_{k+1,k+1}-(h_{12})^T(H_{22})^\dagger h_{12}
\end{pmatrix}.
\end{split}
\end{align}
Let
\begin{align}
\begin{split}
\label{def:t-max-u-max}
t_{\max}&:= \beta_{0,0}-(A_{\min})_{1,1}-(h_{12})^T(H_{22})^\dagger h_{12},\\
u_{\max}&:=\beta_{2k,0}-(A_{\min})_{k+1,k+1}-(h_{12})^T(H_{22})^\dagger h_{12},\\
k_{12}&:=
\beta_{k,0}-(A_{\min})_{2,k}-(h_{23})^T(H_{22})^\dagger h_{12}.
\end{split}
\end{align}
Note that
\begin{equation}
\label{K-eta-no-eta}
    K=
    \begin{pmatrix}
            t_{\max} &  k_{12}\\
            k_{12} & u_{\max}
\end{pmatrix}=
    \cH(A_{\min})\big/H_{22}
    +
    \begin{pmatrix}
    0 & \eta \\
    \eta & 0
    \end{pmatrix}.
\end{equation}


\bigskip

Write
$$
    \cB:=\{Y^k,Y^{k+1},\ldots,Y,\textit{1},X,X^2,\ldots,X^k\}.
$$
We say the matrix $A\in S_{k+1}$ satisfies 
\textbf{the property (Hyp)} if $\cF(A)$ is positive semidefinite and one of the following holds: 
\begin{equation}
\label{Hyp-cond}
\underbrace{\Rank \cF(A)=
    \Rank \cF(A)_{\cB\setminus\{X^{k}\}}=
    \Rank \cF(A)_{\cB\setminus\{Y^{k}\}}}_{\text{(Hyp)}_1}
\quad\text{or}\quad
\underbrace{\Rank \cF(A)=2k+1}_{\text{(Hyp)}_2}.
\end{equation}
\medskip



The solution to the $\cZ(p)$--TMP is the following. 
      
	\begin{theorem}
		\label{sol:x-y-axy} 
	Let $p(x,y)=y(x+y-xy)$
	and $\beta:=\beta^{(2k)}=\{\beta_{i,j}\}_{i,j\in \ZZ_+,i+j\leq 2k}$, where $k\geq 3$.
	Assume the notation above.
	Then the following statements are equivalent:
	\begin{enumerate}	
		\item\label{sol:x-y-axy-pt1} $\beta$ has a $\cZ(p)$--representing measure.
        \smallskip
		\item\label{sol:x-y-axy-pt2}  
  $\mc{M}(k;\beta)$ is positive semidefinite, the relations \eqref{130623-0758-v2}
  hold
  and there exists a pair $(\tilde t,\tilde u)$
  such that 
   $\gamma_{1}(\tilde t,\tilde u)$ is
        $\RR$--representable
        and 
        $A_{\gamma_{2}(\tilde t,\tilde u)}$
        satisfies (Hyp),
        where:
\begin{enumerate}
\item If $u_{\max}=0$,
    $$(\tilde t,\tilde u)\in \{(0,0),
    (t_{\max},0)\}.$$
\item If $u_{\max}>0$ and $k_{12}=0$,
    $$(\tilde t,\tilde u)\in \left\{(0,0),
    \Big(\frac{\eta^2}{u_{\max}},u_{\max}\Big),(t_{\max},u_{\max})\right\}.$$
\item If $u_{\max}>0$ and  $k_{12}\neq 0$,
    $$(
    \tilde t,\tilde u)\in \left\{
    (t_{-,\eta^2},u_{-,\eta^2}),(t_{+,\eta^2},u_{+,\eta^2}),
    \Big(t_{\max}-\frac{|k_{12}|\sqrt{t_{\max}}}{\sqrt{u_{\max}}},
            u_{\max}-\frac{|k_{12}|\sqrt{u_{\max}}}{\sqrt{t_{\max}}}\Big)\right\},$$
    where writing $B:=k_{12}^2-t_{\max}u_{\max}-\eta^2$ we have
    \begin{align*}
        u_{\pm,\eta^2}
             &=
             \frac{-B\pm \sqrt{B^2-4t_{\max}u_{\max}\eta^2}}{2t_{\max}}
        \qquad\text{and}\qquad
        t_{\pm,\eta^2}=\frac{\eta^2}{u_{\pm,\eta^2}}.
    \end{align*}
\end{enumerate}
\end{enumerate}
\end{theorem}
\medskip

Before we prove Theorem \ref{sol:x-y-axy} we need few 
lemmas.
Their statements and the proofs coincide verbatim with \cite[Theorem 6.1, Claims 1--3]{YZ24}, 
but we state them for easier readability.\\

    Let 
\begin{align*}
    \mc R_1
    &=\big\{(t,u)\in \RR^2\colon \cF(\mc G(t,u))\succeq 0\big\}
    \quad\text{and}\quad
    \mc R_2
    =\big\{(t,u)\in \RR^2\colon \cH(\mc G(t,u))\succeq 0\big\}.
\end{align*}
Claims 1 and 2 below describe ranks of 
$\cF(\cG(t,u))$ and $\cH(\cG(t,u))$
for various choices of $(t,u)$ in $\mc R_1$ and $\mc R_2$.

\begin{lemma}[{\cite[Theorem 6.1, Claim 1]{YZ24}}]
\label{lemma:claim-1}
    Assume that $\widetilde{\mc{M}}(k;\beta)\succeq 0$.
    Then
    \begin{equation}
    \label{form-of-R1-parabolic}    
    \mc R_1
    =
    \big\{
    (t,u)\in \RR^2\colon 
    t\geq 0, u\geq 0, tu\geq \eta^2
    \big\}.
    \end{equation}
If $(t,u)\in \mc R_1$, we have
\begin{align}
    \label{rank-R1-parabolic}
    \Rank \cF(\mc G(t,u))=
    \left\{
    \begin{array}{rl}
            \Rank \cF(A_{\min}),&    
            \text{if }
                \eta=t=u=0, 
                \\[0.3em]
            \Rank \cF(A_{\min})+1,&    
            \text{if }
                (\eta=t=0, u>0)
                \text{ or }\\
                &
                (\eta=u=0, t>0)
                \text{ or }(\eta\neq 0, tu=\eta^2),\\[0.3em]
            \Rank \cF(A_{\min})+2,&    
            \text{if }
                tu>\eta^2.
    \end{array}
    \right.
\end{align}
\end{lemma}
\medskip

Define the matrix function
\begin{align}
\label{K(G(t,u))}
\begin{split}
\mc K(\mathbf{t},\mathbf{u})
:=
\cH(\mc G(\mathbf{t},\mathbf{u}))\big/H_{22}
&=
\cH(\widehat A_{\min})\big/H_{22}
-
\begin{pmatrix}
    \mathbf{t} & 0 \\ 0 & \mathbf{u}
\end{pmatrix}\\
&=
K-
\begin{pmatrix}
    \mathbf{t} & 0 \\ 0 & \mathbf{u}
\end{pmatrix}
=
\begin{pmatrix}
    t_{\max}-\mathbf{t} & k_{12} \\ k_{12} & u_{\max}-\mathbf{u}
\end{pmatrix}.
\end{split}
\end{align}

\begin{lemma}[{\cite[Theorem 6.1, Claim 2]{YZ24}}]
\label{lemma:claim-2}
    Assume that $\widetilde{\mc{M}}(k;\beta)\succeq 0$.
    Then
    \begin{align}
    \label{form-of-R2-parabolic}    
    \begin{split}
    \mc R_2
    &=
    \big\{
    (t,u)\in \RR^2\colon 
    \mc K(t,u)\succeq 0
    \big\}\\
    &=\big\{
    (t,u)\in \RR^2\colon 
    t\leq t_{\max}, u\leq u_{\max}, (t_{\max}-t)(u_{\max}-u)\geq k_{12}^2
    \big\}.
    \end{split}
    \end{align}
If $(t,u)\in \mc R_2$, we have
\begin{align}
    \label{rank-R2-parabolic}
    \Rank \cH(\mc G(t,u))=
    \left\{
    \begin{array}{rl}
            \Rank H_{22},&    
            \text{if }
                k_{12}=0, t=t_{\max}, u=u_{\max}, 
                \\[0.3em]
            \Rank H_{22}+1,&    
            \text{if }
                (t_{\max}-t)(u_{\max}-u)=k_{12}^2, (t\neq t_{\max}\text{ or }u\neq u_{\max}),\\[0.3em]
            \Rank H_{22}+2,&    
            \text{if }
                (t_{\max}-t)(u_{\max}-u)>k_{12}^2.
    \end{array}
    \right.
\end{align}
\end{lemma}
\medskip

\begin{lemma}[{\cite[Theorem 6.1, Claim 3]{YZ24}}]
\label{lemma:claim-3}
If $(t,u)\in \mc R_2\cap (\RR_+)^2$, then 
        $$tu\leq (\sqrt{t_{\max}u_{\max}}-\sign(k_{12})k_{12})^2=:p_{\max}.$$
    The equality is achieved if:
    \begin{itemize}
        \item $k_{12}=0$, only in the point $(t,u)=(t_{\max},u_{\max})$.
        \smallskip
        \item $k_{12}\neq 0$, only in point 
            $(t_{p_{\max}},u_{p_{\max}})=
            (
            t_{\max}-\frac{|k_{12}|\sqrt{t_{\max}}}{\sqrt{u_{\max}}},
            u_{\max}-\frac{|k_{12}|\sqrt{u_{\max}}}{\sqrt{t_{\max}}}
            )$.
    \end{itemize}
    Moreover, if $k_{12}\neq 0$, then for every 
    $p\in [0,p_{\max}]$
    there exists a point 
    $(\tilde t,\tilde u)\in \mc R_2\cap (\RR_+)^2$ such that 
    $\tilde t \tilde u=p$ and $(t_{\max}-\tilde t)(u_{\max}-\tilde u)=k_{12}^2$.
\end{lemma}

\medskip
\begin{proof}[Proof of Theorem \ref{sol:x-y-axy}]
The implication  $\eqref{sol:x-y-axy-pt2}\Rightarrow\eqref{sol:x-y-axy-pt1}$ is trivial, since \eqref{sol:x-y-axy-pt2} immediately implies \eqref{measure-cond-hyp-2}.
It remains to prove the implication $\eqref{sol:x-y-axy-pt1}\Rightarrow\eqref{sol:x-y-axy-pt2}$. 
    By \eqref{measure-cond-hyp-2}, there exists $(\tilde t_0,\tilde u_0)$,
    such that 
    $\gamma_{1}(\tilde t_0,\tilde u_0)$ is $\RR$--representable
    and
    $A_{\gamma_{2}(\tilde t_0,\tilde u_0)}$ satisfies (Hyp).
    We separate two cases according to $H_{22}$ (see \eqref{141123-0720}) being positive definite or not.\\

\noindent \textbf{Case 1: $H_{22}$ is not positive definite.}
By Theorem \ref{Hamburger}, it follows that the only option for $\tilde u_0$ is $u_{\max}$.
By Lemma \ref{lemma:claim-2}, we have $k_{12}=0$. 
Applying Theorem \ref{Hamburger} again, for any $t\in [0,t_{\max}]$,
the sequence $\gamma_1(t,u_{\max})$ is $\RR$--representable. We separate two cases according to the value of $u_{\max}$.
\\

\noindent \textbf{Case 1.1: $u_{\max}=0$.} By 
Lemma \ref{lemma:claim-1}, $\eta=0$. 
This and the definition of $A_{\min}$ implies that for any $t\in [0,t_{\max}]$,
$\Rank \cF(\cG(t,0))_{\cB\setminus \{X^k\}}=\Rank \cF(\cG(t,0))$.
If $\tilde t_0>0$, then 
\begin{align}
\label{rank-inequality-chain}
\begin{split}
    \Rank \cF(A_{\min})+1
    &\underbrace{=}_{\eqref{rank-R1-parabolic}}\Rank \cF(\cG(\tilde t_0,0))
    =\Rank \cF(\cG(\tilde t_0,0))_{\cB\setminus \{Y^k\}}\\
    &\leq \Rank \cF(A_{\min})_{\cB\setminus \{Y^k\}}+1
    \leq \Rank \cF(A_{\min})+1,
\end{split}
\end{align}
where 
in the second equality
the assumption that $A_{\gamma_{2}(\tilde t_0,0)}$ satisfies (Hyp) was used.
It follows that all inequalities in \eqref{rank-inequality-chain} must be equalities.
In particular, we have $\Rank \cF(A_{\min})_{\cB\setminus \{Y^k\}}=\Rank \cF(A_{\min})$.
It follows that $A_{\min}=A_{\gamma_2(0,0)}$ satisfies (Hyp) and $(0,0)$ is a good choice for $(\tilde t,\tilde u)$ in Theorem \ref{sol:x-y-axy}.\eqref{sol:x-y-axy-pt2}.\\

\noindent \textbf{Case 1.2: $u_{\max}>0$.} 
By Lemma \ref{lemma:claim-1}, $\tilde t_0\tilde u_0\geq \eta^2$. If $\eta\neq 0$, then $\tilde t_0>0$.
If $\eta=0$, then 
\begin{align*}
2k+1>
\Rank \cF(A_{\min})+1=
\Rank \cF(\cG(0,0))+1
&\underbrace{=}_{\eqref{rank-R1-parabolic}}\Rank \cF(\cG(0,u_{\max}))\\
&\underbrace{>}_{u_{\max}>0}\Rank \cF(\cG(0,u_{\max}))_{\cB\setminus \{X^k\}}.
\end{align*}
Hence, $(0,u_{\max})$ cannot satisfy (Hyp) and thus $\tilde t_0>0$. We separate two cases according to the product $\tilde t_0u_{\max}$, which must be at least $\eta^2$, by Lemma \ref{lemma:claim-1}.\\

\noindent \textbf{Case 1.2.1: $\tilde t_0u_{\max}=\eta^2$.}
In this case $(\frac{\eta^2}{u_{\max}},u_{\max})$
is a good choice for $(\tilde t,\tilde u)$ in Theorem \ref{sol:x-y-axy}.\eqref{sol:x-y-axy-pt2}.\\

\noindent \textbf{Case 1.2.2: $\tilde t_0u_{\max}>\eta^2$.}
We separate two cases according to the rank of $\cF(\cG(\tilde t_0,u_{\max}))$.\\

\noindent \textbf{Case 1.2.2.1: $\Rank \cF(\cG(\tilde t_0,u_{\max}))=2k+1$.} 
The inequality
$\cF(\cG(\tilde t_0,u_{\max}))\preceq \cF(\cG(t_{\max},u_{\max}))$ implies that
    $\Rank \cF(\cG(t_{\max},u_{\max}))=2k+1$
and thus $(t_{\max},u_{\max})$ satisfies (Hyp).
Therefore $(t_{\max},u_{\max})$
is a good choice for $(\tilde t,\tilde u)$ in Theorem \ref{sol:x-y-axy}.\eqref{sol:x-y-axy-pt2}.\\

\noindent \textbf{Case 1.2.2.2: $\Rank \cF(\cG(\tilde t_0,u_{\max}))<2k+1$.}
Then
\begin{align}
\label{rank-inequality-chain-2}
\begin{split}
    \Rank \cF(A_{\min})+2
    &\underbrace{=}_{\eqref{rank-R1-parabolic}}\Rank \cF(\cG(\tilde t_0,u_{\max}))
    =\Rank \cF(\cG(\tilde t_0,u_{\max}))_{\cB\setminus \{Y^k\}}\\
    &\leq \Rank \cF(A_{\min})_{\cB\setminus \{Y^k\}}+2
    \leq \Rank \cF(A_{\min})+2,
\end{split}
\end{align}
where in the second 
equality we used 
the assumption that $A_{\gamma_{2}(\tilde t_0,u_{\max})}$ satisfies (Hyp).
It follows that all inequalities in \eqref{rank-inequality-chain-2} must be equalities.
Since 
$$
    \cF(\cG(\tilde t_0,u_{\max}))_{\cB\setminus \{Y^k\}}
    \preceq 
    \cF(\cG(t_{\max},u_{\max}))_{\cB\setminus \{Y^k\}},
$$
it follows that
$$
    \Rank \cF(\cG(t_{\max},u_{\max}))=
    \Rank \cF(A_{\min})+2=
    \Rank \cF(\cG(t_{\max},u_{\max}))_{\cB\setminus \{Y^k\}}.
$$
Similarly,
$
    \Rank \cF(\cG(t_{\max},u_{\max}))
    =
    \Rank \cF(\cG(t_{\max},u_{\max}))_{\cB\setminus \{X^k\}}.
$
Therefore $(t_{\max},u_{\max})$ satisfies (Hyp) and 
is a good choice for $(\tilde t,\tilde u)$ in Theorem \ref{sol:x-y-axy}.\eqref{sol:x-y-axy-pt2}.\\

\noindent \textbf{Case 2: $H_{22}$ is positive definite.} We separate three cases according to the value of
the pair $(k_{12},\eta)$.\\

\noindent \textbf{Case 2.1: $k_{12}=\eta=0$.} 
By definition of $t_{\max}, u_{\max}$ and Theorem \ref{Hamburger}, 
    $\gamma_1(t,u)$ is $\RR$--representable for every 
\begin{equation}
\label{R-representability}
    (t,u)\in [0,t_{\max})\times [0,u_{\max}]
    \quad\text{and}\quad
    (t,u)=(t_{\max},u_{\max}).
\end{equation}
We separate two cases acccording to the rank of $\cF(A_{\min})$.\\

\noindent \textbf{Case 2.1.1: $\Rank \cF(A_{\min})<2k-1$.} 
    If $t_{\max}=0$, by Lemma \ref{lemma:claim-1}, 
    $\Rank \cF(\cG(\tilde t_0,\tilde u_0))\leq \Rank \cF(A_{\min})+2$.
    Since $(\tilde t_0,\tilde u_0)$ satisfies (Hyp) and $\Rank \cF(A_{\min})<2k-1$,
    it follows that $(\tilde t_0,\tilde u_0)$ satisfies (Hyp)$_1$.
    We separate four options depending on the sign of 
    $\tilde t_0$ and $\tilde u_0$.\\

\noindent \textbf{Case 2.1.1.1: $\tilde t_0=\tilde u_0=0$.}
This means $(0,0)$ is a good choice for $(\tilde t,\tilde u)$ in Theorem \ref{sol:x-y-axy}.\eqref{sol:x-y-axy-pt2}. \\
    
\noindent \textbf{Case 2.1.1.2: $\tilde t_0>0$ and $\tilde u_0>0$.} 
Then
\begin{align}
\label{rank-inequality-chain-3}
\begin{split}
    \Rank \cF(A_{\min})+2
    &\underbrace{=}_{\eqref{rank-R1-parabolic}}\Rank \cF(\cG(\tilde t_0,\tilde u_0))
    =\Rank \cF(\cG(\tilde t_0,\tilde u_0))_{\cB\setminus \{Y^k\}}\\
    &\leq \Rank \cF(A_{\min})_{\cB\setminus \{Y^k\}}+2
    \leq \Rank \cF(A_{\min})+2,
\end{split}
\end{align}
where in the second
equality we used 
the assumption that $A_{\gamma_{2}(\tilde t_0,\tilde u_0)}$ satisfies (Hyp).
It follows that all inequalities in \eqref{rank-inequality-chain-3} must be equalities.
Since 
$$
    \cF(\cG(\tilde t_0,\tilde u_0))_{\cB\setminus \{Y^k\}}
    \preceq 
    \cF(\cG(t_{\max},u_{\max}))_{\cB\setminus \{Y^k\}},
$$
it follows that
$$
    \Rank \cF(\cG(t_{\max},u_{\max}))=
    \Rank \cF(A_{\min})+2=
    \Rank \cF(\cG(t_{\max},u_{\max}))_{\cB\setminus \{Y^k\}}.
$$
Similarly,
$$
    \Rank \cF(\cG(t_{\max},u_{\max}))
    =
    \Rank \cF(\cG(t_{\max},u_{\max}))_{\cB\setminus \{X^k\}}.
$$
Therefore $(t_{\max},u_{\max})$ satisfies (Hyp)$_{1}$ and 
is a good choice for $(\tilde t,\tilde u)$ in Theorem \ref{sol:x-y-axy}.\eqref{sol:x-y-axy-pt2}. \\

\noindent \textbf{Case 2.1.1.3: $\tilde t_0=0$ and $\tilde u_0>0$.} 
Then
    $$
    \Rank \cF(\cG(0,0))_{\cB\setminus \{X^k\}}
    =
    \Rank \cF(\cG(0,\tilde u_0))_{\cB\setminus \{X^k\}}
    \underbrace{<}_{\tilde u_0>0}
    \Rank \cF(\cG(0,\tilde u_0)),
    $$
where in the equality we used the observation $\tilde u_0$ occurs only in the column $X^k$. 
Hence, $\gamma_{2}(0,\tilde u_0)$ cannot satisfy (Hyp). So this case does not occur.\\

\noindent \textbf{Case 2.1.1.4: $\tilde t_0>0$ and $\tilde u_0=0$.} 
Then
\begin{align}
\label{rank-inequality-chain-4}
\begin{split}
    \Rank \cF(A_{\min})+1
    &\underbrace{=}_{\eqref{rank-R1-parabolic}}
    \Rank \cF(\cG(\tilde t_0,0))
    =\Rank \cF(\cG(\tilde t_0,0))_{\cB\setminus \{Y^k\}}\\
    &\leq \Rank \cF(A_{\min})_{\cB\setminus \{Y^k\}}+1
    \leq \Rank \cF(A_{\min})+1,
\end{split}
\end{align}
where in the 
second equality we used 
the assumption that $A_{\gamma_{2}(\tilde t_0,0)}$ satisfies (Hyp).
It follows that all inequalities in \eqref{rank-inequality-chain-4} must be equalities.
In particular, $\Rank \cF(A_{\min})=\Rank \cF(A_{\min})_{\cB\setminus \{Y^k\}}$.
Similarly,
$\Rank \cF(A_{\min})=\Rank \cF(A_{\min})_{\cB\setminus \{X^k\}}$.
Therefore $(0,0)$ satisfies (Hyp)$_{1}$ and 
is a good choice for $(\tilde t,\tilde u)$ in Theorem \ref{sol:x-y-axy}.\eqref{sol:x-y-axy-pt2}.\\

\noindent \textbf{Case 2.1.2: $\Rank \cF(A_{\min})=2k-1$.} 
    The reasoning in the case $\tilde t_0=\tilde u_0=0$
    is the same as in Case 2.1.1.1, 
    in the case $\tilde t_0=0$ and $\tilde u_0>0$
    the same as in Case 2.1.1.3,
    and
    in the case $\tilde t_0>0$ and $\tilde u_0=0$
    the same as in Case 2.1.1.4 above.
    Assume that $\tilde t_0>0$ and $\tilde u_0>0$.
    By Lemma \ref{lemma:claim-1}, 
    $\Rank \cF(\cG(t_{\max},u_{\max}))=2k+1$ and 
    $A_{\gamma_2}(t_{\max},u_{\max})$ satisfies (Hyp).
    Together with \eqref{R-representability}, it follows that 
     $(t_{\max},u_{\max})$
    is a good choice for $(\tilde t,\tilde u)$ in Theorem \ref{sol:x-y-axy}.\eqref{sol:x-y-axy-pt2}.\\

\noindent \textbf{Case 2.2: $k_{12}=0$ and $\eta\neq 0$.}
Since $\eta\neq 0$, it follows that $u_{\max}>0$ (using Lemma \ref{lemma:claim-1}). But then as in Case 1.2
above, one of $(\frac{\eta^2}{u_{\max}},u_{\max})$, $(t_{\max},u_{\max})$
is a good choice for $(\tilde t,\tilde u)$ in Theorem \ref{sol:x-y-axy}.\eqref{sol:x-y-axy-pt2}.\\

\noindent \textbf{Case 2.3: $k_{12}\neq 0$ and $\eta\neq 0$.}
            Let $p_{\max}$ be as in Lemma \ref{lemma:claim-3}.
            We separete two cases according to the value of $p_{\max}$.\\

    \noindent \textbf{Case 2.3.1: $p_{\max}=\eta^2$.} 
    In this case 
    $\mathcal R_1\cap \mathcal R_2=\left\{(t_{p_{\max}},u_{p_{\max}})\right\}$,
    where $(t_{p_{\max}},u_{p_{\max}})$ is as in Lemma \ref{lemma:claim-3}.
    and thus $(t_{p_{\max}},u_{p_{\max}})$ 
    is the only candidate for
            $(\tilde t_0,\tilde u_0)$ in \eqref{measure-cond-hyp-2}.\\

    \noindent \textbf{Case 2.3.2: $p_{\max}>\eta^2$.} 
            We separate two cases depending on 
            $\Rank \cF(A_{\min})$.\\
            
            \noindent \textbf{Case 2.3.2.1: $\Rank \cF(A_{\min})=2k-1$.} 
            Then by Lemma \ref{lemma:claim-1}, 
    $\Rank \cF(\cG(t',u'))=2k+1$ 
    for every $(t',u')$ such that $t'>0,u'>0$, $t'u'>\eta^2$
    and hence $A_{\gamma_2}(t',u')$ satisfies (Hyp)$_2$.
    By Lemma \ref{lemma:claim-3},
    $(t',u')$ equal to 
    $(t_{p_{\max}},u_{p_{\max}})$
    satisfies (Hyp)$_2$, 
    and
    $\gamma_1(t',u')$ is $\RR$--representable, since
            $A_{\gamma_1(t',u'}(k)\succ 0$ (by $t'<t_{\max}$).
    Hence, this $(t',u')$
    is a good choice for $(\tilde t,\tilde u)$ in Theorem \ref{sol:x-y-axy}.\eqref{sol:x-y-axy-pt2}. \\
     
    \noindent \textbf{Case 2.3.2.2: $\Rank \cF(A_{\min})<2k-1$.}
    By Lemma \ref{lemma:claim-1}, 
    $\Rank \cF(\cG(t',u'))<2k+1$ 
    for every $(t',u')\in \mc R_1$
    and hence $A_{\gamma_2}(t',u')$ cannot satisfy (Hyp)$_2$ for any $(t',u')\in \mc R_1$.
    Thus $(\tilde t_0,\tilde u_0)$ satisfies (Hyp)$_1$. 
    We have
    \begin{align}
\label{rank-inequality-chain-7}
\begin{split}
    \Rank \cF(A_{\min})
    +
    \Rank 
    \begin{pmatrix}
    \tilde t_0 & \eta \\ \eta & \tilde u_0    
    \end{pmatrix}
    &\underbrace{=}_{\eqref{rank-R1-parabolic}}
    \Rank \cF(\cG(\tilde t_0,\tilde u_0))\\
    &=\Rank \cF(\cG(\tilde t_0,\tilde u_0))_{\cB\setminus \{Y^k\}}\\
    &\leq \Rank \cF(A_{\min})_{\cB\setminus \{Y^k\}}+
    \Rank 
    \begin{pmatrix}
    \tilde t_0 & \eta \\ \eta & \tilde u_0    
    \end{pmatrix}\\
    &\leq \Rank \cF(A_{\min})
    +
        \Rank 
    \begin{pmatrix}
    \tilde t_0 & \eta \\ \eta & \tilde u_0    
    \end{pmatrix},
\end{split}
\end{align}
where in the 
second equality we used 
the assumption that $A_{\gamma_{2}(\tilde t_0,\tilde u_0)}$ satisfies (Hyp).
It follows that equalities hold for all inequalities in
\eqref{rank-inequality-chain-7}.
In particular, 
\begin{equation}
    \label{141025-1430}
        \Rank \cF(A_{\min})=\Rank \cF(A_{\min})_{\cB\setminus \{Y^k\}}.
\end{equation}
Similarly,
\begin{equation}
    \label{rank-equality-v8}
        \Rank \cF(A_{\min})=\Rank \cF(A_{\min})_{\cB\setminus \{X^k\}}.
\end{equation}

By Lemma \ref{lemma:claim-3}, there is a point $(\tilde t,\tilde u)\in \mc R_2\cap (\RR_+)^2$,
             such that 
             \begin{equation}  
             \label{141025-1431}
             \tilde t\tilde u=\eta^2
             \quad\text{and}\quad 
             (t_{\max}-\tilde t)(u_{\max}-\tilde u)=k_{12}^2.
             \end{equation}
             Moreover, there are exactly two such points 
             $(\tilde t,\tilde u)$ satisfying \eqref{141025-1431}:
             \begin{align*}
             \Big(t_{\max}-\frac{\eta^2}{u}\Big)\Big(u_{\max}-u\Big)=k_{12}^2
             &\quad\Leftrightarrow\quad
             (t_{\max}u-\eta^2)(u_{\max}-u)=k_{12}^2u\\
             &\quad\Leftrightarrow\quad
             t_{\max}u^2+(k_{12}^2-t_{\max}u_{\max}-\eta^2)u+\eta^2 u_{\max}=0
             \\
            &\quad\Leftrightarrow\quad
             u_{\pm,\eta^2}
             =
             \frac{-B\pm \sqrt{B^2-4t_{\max}u_{\max}\eta^2}}{2t_{\max}},
             \end{align*}
             where 
                $B=k_{12}^2-t_{\max}u_{\max}-\eta^2$.
            Clearly, $t_{\max}u_{\max}\geq k_{12}^2$, and
            since $\eta^2> 0$, it follows that $B<0$. 
            A short computation shows that
            \begin{equation}
            \label{141025-1513}
            0=B^2-4t_{\max}u_{\max}\eta^2
            \quad \Leftrightarrow \quad
            \eta^2\in 
            \{
            (\sqrt{t_{\max}}\sqrt{u_{\max}}+k_{12})^2,
            (\sqrt{t_{\max}}\sqrt{u_{\max}}-k_{12})^2
            \}.
            \end{equation}
        We have 
        $$p_{\max}= 
        (\sqrt{t_{\max}}\sqrt{u_{\max}}-|k_{12}|)^2<
        (\sqrt{t_{\max}}\sqrt{u_{\max}}+|k_{12}|)^2.$$
        Since $\eta^2<p_{\max}$, \eqref{141025-1513} implies that $B^2-4t_{\max}u_{\max}\eta^2\neq 0$.
        Therefore 
            $$0<u_{-,\eta^2}<u_{+,\eta^2}.$$
    Let $t_{\pm,\eta^2}:=\frac{\eta^2}{u_{\pm,\eta^2}}$.
Note that 
\begin{align*}
\Rank \cF(\cG(t_{\pm,\eta^2},u_{\pm,\eta^2}))_{\cB\setminus \{X^k\}}
\underbrace{=}_{t_{\pm,\eta^2}>0}
\Rank \cF(A_{\min})_{\cB\setminus \{X^k\}}+1
&\underbrace{=}_{\eqref{rank-equality-v8}}
\Rank \cF(\cG(A_{\min}))+1\\
&\underbrace{=}_{\eqref{rank-R1-parabolic}}
\Rank \cF(\cG(t_{\pm,\eta^2},u_{\pm,\eta^2})).
\end{align*}
So $A_{\gamma_2(t_{\pm,\eta^2},u_{\pm,\eta^2})}$ satisfies (Hyp)$_1$ if and only if 
\begin{equation}
    \label{rank-equality-v9}
        \Rank \cF(\cG(t_{\pm,\eta^2},u_{\pm,\eta^2}))_{\cB\setminus \{Y^k\}}
        =
        \Rank \cF(\cG(t_{\pm,\eta^2},u_{\pm,\eta^2})).
\end{equation}
Since $t_{\pm,\eta^2}>0$ and $u_{\pm,\eta^2}>0$, it follows that  
\begin{align*}
    \Rank \cF(\cG(t_{\pm,\eta^2},u_{\pm,\eta^2}))_{\cB\setminus \{X^k,Y^k\}}
    &>\Rank\cF(\cG(t_{\pm,\eta^2},u_{\pm,\eta^2}))_{\cB\setminus \{\textit{1},X^k,Y^k\}},\\
    \Rank \cF(\cG(t_{\pm,\eta^2},u_{\pm,\eta^2}))_{\cB\setminus \{\textit{1},Y^k\}}
    &>\Rank\cF(\cG(t_{\pm,\eta^2},u_{\pm,\eta^2}))_{\cB\setminus \{\textit{1},X^k,Y^k\}}.
\end{align*}
If
\begin{align}
    \label{rank-equality-case-1}
    \Rank \cF(\cG(t_{\pm,\eta^2},u_{\pm,\eta^2}))_{\cB\setminus \{Y^k\}}
    &=\Rank\cF(\cG(t_{\pm,\eta^2},u_{\pm,\eta^2}))_{\cB\setminus \{\textit{1},X^k,Y^k\}}+2,
\end{align}
then \eqref{rank-equality-v9} holds. 
Indeed, in this case 
$$
(A_{\min})_{\{1,X^k\}}=
\left(\cF(A_{\min})_{\{\textit{1},X^k\},\cB\setminus \{\textit{1},X^k,Y^k\}}\right)
\left(\cF(A_{\min})_{\cB\setminus \{\textit{1},X^k,Y^k\}}\right)^\dagger
\left(\cF(A_{\min})_{\cB\setminus \{\textit{1},X^k,Y^k\},\{\textit{1},X^k\}}\right),
$$
whence
\begin{align*}
\Rank \cF(\cG(t_{\pm,\eta^2},u_{\pm,\eta^2}))_{\cB\setminus \{Y^k\}}
&=
\Rank \cF(A_{\min})_{\cB\setminus \{Y^k\}}+1\\
&\underbrace{=}_{\eqref{141025-1430}}
\Rank \cF(A_{\min})+1
\underbrace{=}_{\eqref{rank-R1-parabolic}}
\Rank \cF(\cG(t_{\pm,\eta^2},u_{\pm,\eta^2})).
\end{align*}
Write
\begin{equation}
    \label{def:T2}
    \cT_3:=\{Y^{k-1},Y^{k-2},\ldots,Y,X,X^2,\ldots,X^{k-1}\}.
\end{equation}
Assume now that \eqref{rank-equality-case-1} does not hold for both 
$(t_{-,\eta^2},u_{-,\eta^2})$ and $(t_{+,\eta^2},u_{+,\eta^2})$. 
Therefore there are relations 
            \begin{align}
            \label{relations-in-F-v2}
            \begin{split}
            \cF(\cG(t_{+,\eta^2},u_{+,\eta^2}))_{\cB,\{\textit{1}\}}
            +\alpha_{+} \cF(\cG(t_{+,\eta^2},u_{+,\eta^2}))_{\cB,\{X^k\}}+
            \cF(\cG(t_{+,\eta^2},u_{+,\eta^2}))_{\cB,\cT_3}v_+&=0,\\
            \cF(\cG(t_{-,\eta^2},u_{-,\eta^2}))_{\cB,\{\textit{1}\}}
            +\alpha_{-} \cF(\cG(t_{-,\eta^2},u_{-,\eta^2}))_{\cB,\{X^k\}}+
            \cF(\cG(t_{-,\eta^2},u_{-,\eta^2}))_{\cB,\cT_3}v_-&=0,
            \end{split}
            \end{align}
            for some 
            $\alpha_{+},\alpha_-\in \RR$, $v_{+},v_-\in \RR^{2k-1}$
            in 
            $\cF(\cG(t_{+,\eta^2},u_{+,\eta^2}))$
            and
            $\cF(\cG(t_{-,\eta^2},u_{-,\eta^2}))$,
            respectively.
            Since
            $
            \cF(\cG(t_{\pm,\eta^2},u_{\pm,\eta^2}))
            \succeq
            \cF(A_{\min})\succeq 0,
            $
            the relations \eqref{relations-in-F-v2} must hold also in $\cF(A_{\min})$.
            Subtracting these relations we get
            \begin{equation}
                \label{linear-dependence-Xk-column}
                    (\alpha_{+}-\alpha_-) 
                    \cF(A_{\min})_{\cB,\{X^k\}}
                    +
                    \cF(A_{\min})_{\cB,\cT_3}(v_+-v_-)=0.
            \end{equation}
            If $\alpha_+=\alpha_-$, then $v_+-v_-\in \ker \cF(A_{\min})_{\cB,{\cT}_3}$
            and hence $\cF(A_{\min})_{\cB,{\cT}_3}v_+=\cF(A_{\min})_{\cB,{\cT}_3}v_-$.
            But observing the first entries of the left hand side vectors in \eqref{relations-in-F-v2},
            this cannot hold since 
            $$
            \cF(\cG(t_{+,\eta^2},u_{+,\eta^2}))_{\{\textit 1\}}=
            t_{+,\eta^2}
            \neq
            t_{-,\eta^2}
            =
            \cF(\cG(t_{-,\eta^2},u_{-,\eta^2}))_{\{\textit 1\}}.
            $$ 
            So $\alpha_+\neq \alpha_-$
            and from \eqref{linear-dependence-Xk-column} it follows that in $\cF(A_{\min})$, the column $X^k$ is linearly dependent from the columns in ${\cT}_3$. Using one of the 
            relations \eqref{relations-in-F-v2} for $\cF(A_{\min})$, the same holds for the column \textit{1}. But then 
    \begin{align}
    \label{rank-conditions-v4.2}
    \begin{split}
    \Rank \cF(\cG(t_{\pm,\eta^2},u_{\pm,\eta^2}))_{\cB\setminus \{Y^{k}\}}
    &=
    \Rank \cF(A_{\min})_{\cB\setminus \{Y^{k}\}}
    +
    1   
    =
    \Rank \cF(\cG(t_{\pm,\eta^2},u_{\pm,\eta^2})),
    \end{split}
    \end{align}
    and \eqref{rank-equality-v9} holds for both points $(t_{\pm,\eta^2},u_{\pm,\eta^2})$.
             Note that $\gamma_1(t_{\pm,\eta^2}, u_{\pm,\eta^2})$ is $\RR$--representable, since
            $A_{\gamma_1(t_{\pm,\eta^2}, u_{\pm,\eta^2})}(k)\succ 0$ (by $t_{\pm,\eta^2}<t_{\max}$).
           So at least one of $(t_{\pm,\eta^2},u_{\pm,\eta^2})$ 
    is a good choice for 
            $(\tilde t_0,\tilde u_0)$ in \eqref{measure-cond-hyp-2}.\\

    This concludes the proof of the implication $\eqref{sol:x-y-axy-pt1}\Rightarrow\eqref{sol:x-y-axy-pt2}$.            
\end{proof}

\subsection{Cardinality of a minimal representing measure}

It remains to characterize the cardinality of a minimal $\cZ(p)$--rm 
for $\beta$ in Theorem \ref{sol:x-y-axy}.\\

Let $\cT_4:=\{Y^k,\ldots,Y,X,\ldots,X^{k-1}\}$,
$\vec\cT_4:=(\vec{Y}^{(k,1)},\vec{X}^{(1,k-1)})$ and
$\cT_5=\{1,X^k\}\cup \cT_4$. 
Write
\begin{align}
\label{widehat-F}
\begin{split}
&\cF(\mc G(\mathbf{t},\mathbf{u}))_{(1,\vec{\cT}_4,X^k,\overrightarrow{\cC\setminus \cT_5})}\\
&=
\kbordermatrix
{
& \mathit{1} & \vec{\cT}_4 & X^k & \overrightarrow{\cC\setminus \cT_5}\\[0.2em]
\mathit{1} & (A_{\min})_{1,1}+\mathbf{t} & (f_{12})^T & (A_{\min})_{2,k} & (f_{14})^T\\[0.2em]
(\vec{\cT}_4)^T
& f_{12} & F_{22} & f_{23} & F_{24}\\[0.2em]
X^k & (A_{\min})_{2,k} & (f_{23})^T & (A_{\min})_{k+1,k+1}+\mathbf{u} & (f_{34})^T\\[0.2em]
\overrightarrow{\cC\setminus\cT_5} & f_{14} & (F_{24})^T & f_{34} & F_{44}
}.
\end{split}
\end{align}
Note that
\begin{align*}
    f_{12}^T
    &=
        \begin{pmatrix}
            (b_{13}^{(0)})^T & (b_{23}^{(0)})^T
        \end{pmatrix},
    \quad
    (A_{\min})_{2,k}
    =\beta_{k,1}-\beta_{k-1,1},
    \quad
    F_{22}
    =
        \begin{pmatrix}
            B_{11} & B_{13}^{(1,k-1)}\\[0.3em]
            (B_{13}^{(1,k-1)})^T & B_{23}^{(1,k-1)}
        \end{pmatrix}.
\end{align*}
\smallskip

The following theorem characterizes the cardinality of a minimal measure in case 
$\beta$ admits a $\cZ(p)$--rm.

\begin{theorem}
\label{thm:hyperbolic-2-minimal-measures}
    Let $p(x,y)=y(x+y-xy)$
	and $\beta=(\beta_{i,j})_{i,j\in \ZZ_+,i+j\leq 2k}$, where $k\geq 3$,
        admits a $\cZ(p)$--representing measure.
        Assume the notation above.
        The following statements hold:
\begin{enumerate}
        \item There exists at most $(\Rank \widetilde{\mc{M}}(k;\beta)+2)$--atomic $\cZ(p)$--representing 
            measure.
        \smallskip
        \item There is no
        $\cZ(p)$--representing measure with less than $\Rank \widetilde{\mc{M}}(k;\beta)+2$ atoms
	if and only if 
                $\eta\neq 0$,
                $k_{12}\neq 0$,
                $\Rank \cH(A_{\min})=k$
                and
                $A_{\min}$
                does not satisfy (Hyp).
        \smallskip
        \item There does not exist a $(\Rank \widetilde{\mc{M}}(k;\beta))$--atomic $\cZ(p)$--representing measure
	   if and only if any of the following holds:
        \smallskip
        \begin{enumerate}
            \item 
                $F_{22}\succ 0$,
                $H_{22}\not \succ 0$,
                $\Rank \cH(A_{\min})=\Rank H_{22}+1$
                and 
                $t_{\max}u_{\max}>\eta^2>0$.
                \smallskip
            \item 
                $F_{22}\succ 0$,
                $H_{22}\succ 0$
                and one of the following holds:
                \smallskip
                \begin{enumerate}
                \item 
                $\eta=0$,
                $k_{12}\neq 0$,
                $\Rank \cH(A_{\min})=k+1$
                and
                $A_{\min}$ does not satisfy (Hyp)$_1$.
                \smallskip
                \item 
                $\eta\neq 0$,
                $k_{12}\neq 0$,
                $A_{\min}$
                satisfies (Hyp)$_1$
                and $\Rank \cH(A_{\min})=k$.
                \smallskip
                \item 
                $\eta\neq 0$,
                $k_{12}\neq 0$
                and
                $A_{\min}$
                does not satisfy (Hyp)$_1$.
                \end{enumerate}
    \end{enumerate}
\end{enumerate}

        In particular, a $p$--pure sequence $\beta$ with a 
            measure
admits at most $(3k+1)$--atomic $\cZ(p)$--representing 
            measure.
\end{theorem}


\begin{proof}
    By Lemma \ref{071123-2008}.\eqref{071123-2008-pt3},
    \begin{equation}
        \label{181123-0849}
        \Rank \widetilde{\mc M}(k;\beta)
        =\Rank \cF(A_{\min})+
        \Rank \cH(A_{\min}).
    \end{equation}
    By \eqref{measure-cond-hyp-2}, 
    there exists a pair 
    $(\tilde t_0,\tilde u_0)\in \RR^2$ such that 
    $\cF(\mc G(\tilde t_0,\tilde u_0))$
    and 
    $\cH(\mc G(\tilde t_0,\tilde u_0))$
    admit a $\cZ(x+y-xy)$--rm and a 
    $\RR$--rm, respectively. 
    In the proof we will separate the following cases:
\begin{itemize}
\item \textbf{Case 1: $F_{22}$ is not pd.}
\item \textbf{Case 2: $F_{22}$ is pd, $H_{22}$ is not pd and $u_{\max}=0$.}
\item \textbf{Case 3: $F_{22}$ is pd, $H_{22}$ is not pd and $u_{\max}>0$.}
\item \textbf{Case 4: $F_{22}$ and $H_{22}$ are pd, $\eta=0$.}
\item \textbf{Case 5: $F_{22}$ and $H_{22}$ are pd, $\eta\neq 0$.}
\end{itemize}

    \noindent \textbf{Case 1: $F_{22}$ is not pd.}
    Note that the matrix $A_{\gamma_2(t,u)}$ can satisfy 
    (Hyp) only if it satisfies (Hyp)$_1$:
    \begin{equation}
    \label{hyp-real1}
        \Rank \cF(A_{\gamma_2(t,u)})=
        \Rank \cF(A_{\gamma_2(t,u)})_{\cB\setminus\{X^{k}\}}=
        \Rank \cF(A_{\gamma_2(t,u)})_{\cB\setminus\{Y^{k}\}}.
    \end{equation}
        Since $F_{22}$ is not pd,
        $\cF(\cG(\tilde t_0,\tilde u_0))$ satisfies a nontrivial column relation of the form 
    \begin{equation}
    \label{relation-hyp-2}
            \sum_{i=1}^{k}\delta_i Y^{i}
            +
            \sum_{j=1}^{k-1}\xi_j X^j=0
            \qquad
            \text{for some }\delta_i,\xi_j\in \RR,\text{ not all zero}.
    \end{equation}
    Since 
    $\cF(\cG(\tilde t_0,\tilde u_0))$ satisfies the column relation 
    $XY=X+Y$, it follows by recursive generation, that its extensions, generated by a representing measure, must satisfy column relations
    \begin{align}
    \label{column-relation-hyp2}
    \begin{split}
        Y^2X&=Y(YX)=Y(X+Y)=Y^2+YX=Y^2+Y+X,\\
        Y^3X&=Y(Y^2X)=Y(Y^2+Y+X)=Y^3+Y^2+YX=Y^3+Y^2+Y+X,\\
        &\vdots\\
        Y^iX&=Y^{i}+Y^{i-1}+\ldots+Y+X\quad \text{for }i\geq 1.
    \end{split}
    \end{align}
    Multiplying \eqref{relation-hyp-2} with $X$ and using \eqref{column-relation-hyp2}, we get a column relation
    \begin{equation}
    \label{relation-hyp-2-v2}
            \sum_{i=1}^{k}\Big(\sum_{j=i}^{k} \delta_j\Big) Y^{i}
            +
            \Big(\sum_{j=1}^k \delta_j\Big) X
            +
            \sum_{j=2}^{k}\xi_{j-1} X^j=0,
    \end{equation}
    We separate two subcases according to the values of
    $\xi_j$ and $\sum_{j=1}^k \delta_j$.
    \medskip
    
    \noindent \textbf{Case 1.1:
    $\xi_j\neq 0$ for some $j$ or $\sum_{j=1}^k \delta_j\neq 0$.} 
    Multiplying \eqref{relation-hyp-2} with $X^\ell$ for $\ell$ large enough,
    we will eventually get a column relation \eqref{relation-hyp-2-v2} with a nonzero coefficient at $X^k$. But this means $X^k$ must be in the span of the columns $Y^{k},\ldots,Y,X,\ldots,X^{k-1}$. In particular,
    $\tilde u_0=0$.
    By \eqref{form-of-R1-parabolic}, this implies that 
    $\eta=0$ and thus $\widehat A_{\min}=A_{\min}$.
    Moreover, since $\gamma_2(\tilde t_0,0)$ satisfies (Hyp),
    it follows that $\gamma_2(0,0)$ also satisfies (Hyp).
    The $\RR$--representability of $\gamma_{1}(\tilde t_0,0)$ implies that
    $\gamma_{1}(0,0)$ is also $\RR$--representable. 
    So
    $\cF(A_{\min})$ and $\cH(A_{\min})$
    admit a $\cZ(x+y-xy)$--rm and a $\RR$--rm, respectively.
    By Theorem \ref{Hamburger} and Corollary \ref{cor:x+y+axy}, there also exist
    a $(\Rank\cF(A_{\min}))$--atomic and a $(\Rank\cH(A_{\min}))$--atomic rm\textit{s}. 
    By \eqref{181123-0849}, $\beta$ has a $(\Rank \widetilde{\mc M}(k;\beta))$--atomic $\cZ(p)$--rm.
    \medskip
    
    \noindent \textbf{Case 1.2: $\xi_j=0$ for all $j$ and $\sum_{j=1}^k \delta_j=0$.}
    \eqref{relation-hyp-2-v2} implies that there is a column relation
    of the form $\sum_{i=2}^k \delta_i^{(2)} Y^i=0$ in 
    $\cF(\cG(\tilde t_0,\tilde u_0))$ for some $\delta_i^{(2)}\in \RR$, not all zero (observe that the coefficients at $Y$ and $X$ in \eqref{relation-hyp-2-v2} are both $\sum_{j=1}^k \delta_j$).
    Mutliplying $\sum_{i=2}^k \delta_i^{(2)} Y^i=0$ with $X$ we get a relation of the form \eqref{relation-hyp-2-v2} with $\xi_j=0$ for all $j$
    and $\delta_j^{(2)}$ instead of $\delta_j$. If 
    $\sum_{j=1}^k\delta_j^{(2)}=\sum_{j=2}^k\delta_j^{(2)}\neq 0$, then the coefficient at $X$ is nonzero and we can proceed as in Case 1.1 above. 
    Otherwise the coefficients at $X$, $Y$ and $Y^2$ are all zero. Hence, we got a new column relation
    of the form $\sum_{i=3}^k \delta_i^{(3)} Y^i=0$ in 
    $\cF(\cG(\tilde t_0,\tilde u_0))$ for some $\delta_i^{(3)}\in \RR$, not all zero. Proceeding with this procedure inductively we either eventually come into Case 1.1 or end with a relation of the form $\alpha Y^i=0$, $\alpha\neq 0$, $i>0$, which holds in
    $\cF(\cG(\tilde t_0,\tilde u_0))$. But this means all atoms in the conic part of $\cZ(p)$ also lie on the line $y=0$. So a $\cZ(p)$-rm for $\beta$ is a $\cZ(y)$--rm for $\beta$ and, by Theorem \ref{Hamburger},
    $\beta$ has a $(\Rank \widetilde{\mc M}(k;\beta))$--atomic $\cZ(p)$--rm.
    \\

    \noindent 
    \textbf{Case 2: $F_{22}$ is pd, $H_{22}$ is not pd and $u_{\max}=0$.}
    Since $u_{\max}=0$, it follows that $\tilde u_0=0$.
    By \eqref{form-of-R1-parabolic}, this implies that $\eta=0$.
    Analogously as in Case 1.1 above we conclude that $\gamma_2(0,0)$
    satisfies (Hyp) and $\gamma_1(0,0)$ is $\RR$--representable, which implies the existence of a $(\Rank \widetilde{\mc M}(k;\beta))$--atomic $\cZ(p)$--rm for $\beta$.\\

    \noindent{\textbf{Case 3:
        $F_{22}$ is pd,
        $H_{22}$ is not pd,
        $u_{\max}>0$.
    }}
    Since $H_{22}$ is not pd, Theorem \ref{Hamburger} implies that $\tilde u_0=u_{\max}$. But then \eqref{form-of-R2-parabolic} implies that $k_{12}=0$.
    We separate two subcases according to the value of $\eta$.
    \medskip

    \noindent{\textbf{Case 3.1: $\eta=0$.}
    Since $\gamma_{2}(\tilde t_0,u_{\max})$ satisfies (Hyp) and $u_{\max}>0$,
    it follows that $\tilde t_0>0$. But then also 
    $\gamma_{2}(t_{\max},u_{\max})$ satisfies (Hyp). Since 
    $0<\tilde t_0\leq t_{\max}$, \eqref{rank-R1-parabolic} implies that
    \begin{equation}
    \label{rank-equality-v3}
        \Rank \cF(\cG(t_{\max},u_{\max}))=\Rank \cF(A_{\min})+2.
    \end{equation}
    Moreover, $\gamma_{1}(t_{\max},u_{\max})$ is also $\RR$--representable and 
    \eqref{rank-R2-parabolic} implies that
    \begin{equation}
    \label{rank-equality-v2}
        \Rank \cH(\cG(t_{\max},u_{\max}))=\Rank H_{22}=
        \Rank \cH(\cG(0,0))-2=\Rank \cH(A_{\min})-2.
    \end{equation}
    By \eqref{181123-0849}, \eqref{rank-equality-v3} and 
    \eqref{rank-equality-v2},
    there exists a $(\Rank \widetilde{\mc M}(k;\beta))$--atomic $\cZ(p)$--rm for $\beta$.
    \medskip

    \noindent{\textbf{Case 3.2: $\eta\neq 0$.}}
    Since $\eta\neq 0$, it follows, by \eqref{form-of-R1-parabolic},
    that $\tilde t_0>0$ and hence $t_{\max}>0$.
                 Since $k_{12}=0$, \eqref{form-of-R2-parabolic} implies that
                 $\cH(\cG(t_{\max},u_{\max}))$ is psd 
              and
             $\gamma_1(t_{\max},u_{\max})$ is $\RR$--representable.
             Moreover, since $A_{\gamma_2(\tilde t_0,u_{\max})}$
             satisfies (Hyp), also
             $A_{\gamma_2(t_{\max},u_{\max})}$ satisfies (Hyp). Indeed, either $\tilde t_0=t_{\max}$
             or $\tilde t_0<t_{\max}$.
             In the latter case,
             $t_{\max}u_{\max}>\tilde t_0u_{\max}\geq \eta^2$ and $(t_{\max},u_{\max})$ satisfies 
             (Hyp)$_2$.
             Hence, $(t_{\max},u_{\max})$ is a good choice for 
             $(\tilde t_0,\tilde u_0)$ in \eqref{measure-cond-hyp-2}.
     By \eqref{rank-R2-parabolic}, $t_{\max}>0$, $u_{\max}>0$ and $k_{12}=0$,
    imply that 
        $$\Rank\cH(\cG(0,0))=\Rank\cH(\widehat A_{\min})=\Rank H_{22}+2.$$
     We have            
            \begin{align}
            \label{rank-equalities-1609}
            \begin{split}
            &\Rank \cH(\mc G(t_{\max},u_{\max}))
            +
            \Rank \cF(\mc G(t_{\max},u_{\max}))=\\
            =&
            \left\{
            \begin{array}{rl}
                \Rank H_{22}+\Rank \cF(A_{\min})+1,&
                \text{if }t_{\max}u_{\max}=\eta^2,\\[0.3em]
                \Rank H_{22}+\Rank \cF(A_{\min})+2,&
                \text{if }t_{\max}u_{\max}>\eta^2,
            \end{array}
            \right.
            \end{split}
            \end{align}
    where we used 
    \eqref{rank-R1-parabolic} and \eqref{rank-R2-parabolic}
    in the equality.
    By \eqref{K-eta-no-eta}, 
        $$\cH(A_{\min})\big/H_{22}
        =
        \begin{pmatrix}
            t_{\max} & -\eta \\
            -\eta & u_{\max}
        \end{pmatrix}\neq 
        \begin{pmatrix}
            0 & 0 \\ 0 & 0
        \end{pmatrix}  
        .$$
    Hence, 
    $\Rank\cH(A_{\min})=
    \Rank H_{22}+i$ for some $i\in \{1,2\}$.
    We separate these two cases.
    \medskip

    \noindent{\textbf{Case 3.2.1: $\Rank \cH(A_{\min})=\Rank H_{22}+2$.}}
            We have             
            \begin{align*}
            &\Rank \cH(\mc G(t_{\max},u_{\max}))
            +
            \Rank \cF(\mc G(t_{\max},u_{\max}))
             \\
             =&
            \left\{
            \begin{array}{rl}
                \Rank \widetilde {\mc M}(k;\beta)-1,&
                \text{if }t_{\max}u_{\max}=\eta^2,\\[0.3em]
                \Rank \widetilde {\mc M}(k;\beta),&
                \text{if }t_{\max}u_{\max}>\eta^2,
            \end{array}
            \right.
            \end{align*}
    where we used 
    \eqref{rank-equalities-1609}
    and
    \eqref{181123-0849}
    in the equality.
    The case $t_{\max}u_{\max}=\eta^2$ cannot happen,
    since this would imply $\beta$ has a $(\Rank \widetilde{\mc M}(k;\beta)-1)$--atomic $\cZ(p)$--rm, which is not possible.
    Hence, there
             is a $(\Rank \widetilde{\mc M}(k;\beta))$--atomic $\cZ(p)$--rm for $\beta$.
    \medskip

    \noindent{\textbf{Case 3.2.2: $\Rank \cH(A_{\min})=\Rank H_{22}+1$.}}
            In this case we have
            \begin{align}
            \label{rank-computation-v2}
            \begin{split}
            &\Rank \cH(\mc G(t_{\max},u_{\max}))
            +
            \Rank \cF(\mc G(t_{\max},u_{\max}))
            \\
             &\underbrace{=}_{\eqref{181123-0849}}
            \left\{
            \begin{array}{rl}
                \Rank \widetilde {\mc M}(k;\beta),&
                \text{if }t_{\max}u_{\max}=\eta^2,\\[0.3em]
                \Rank \widetilde {\mc M}(k;\beta)+1,&
                \text{if }t_{\max}u_{\max}>\eta^2,
            \end{array}
            \right..
            \end{split}
            \end{align}
    where we used 
    \eqref{rank-R1-parabolic} and \eqref{rank-R2-parabolic}
    in the equality.
            Hence, $\beta$ has a $(\Rank \widetilde {\mc M}(k;\beta))$--atomic rm if $t_{\max}u_{\max}=\eta^2$
            and $(\Rank \widetilde {\mc M}(k;\beta)+1)$--atomic rm if $t_{\max}u_{\max}>\eta^2$.
            It remains to show that in the case $t_{\max}u_{\max}>\eta^2$,
            there does not exist a 
            $(\Rank \widetilde {\mc M}(k;\beta))$--atomic rm.
            Since $H_{22}$ is not pd and $u_{\max}>0$, if 
            $\cH(\mc G(t',u'))$ has a $\RR$--rm, then $u'=u_{\max}$.
            Since 
            $\eta\neq 0$, we know 
            $\cF(\mc G(t',u_{\max}))$ with a $\cZ(x+y-xy)$--rm
            is at least $(\Rank \cF(A_{\min})+1)$--atomic (see \eqref{rank-R1-parabolic}).
            If $t'\neq t_{\max}$, then, by \eqref{rank-R2-parabolic}, 
            $\Rank \cH(\mc G(t',u_{\max}))=\Rank H_{22}+1$. Hence, 
            \begin{align*}
            \Rank \cH(\mc G(t',u_{\max}))+
            \Rank \cF(\mc G(t',u_{\max}))
            &\geq 
            (\Rank H_{22}+1)
            +
            (\Rank \cF(A_{\min})+1)\\
            &\underbrace{=}_{\eqref{181123-0849}}\Rank \widetilde{\mc M}(k;\beta)+1.
            \end{align*}
            
    \noindent \textbf{Case 4:
    $F_{22}$ and $H_{22}$ are pd,
    $\eta=0$.}
    We separate two cases according to the value of $u_{\max}$.
    \medskip
    
    \noindent \textbf{Case 4.1: $u_{\max}=0$.} 
        Since $u_{\max}=0$, it follows that $\tilde u_0=0$ and by \eqref{form-of-R2-parabolic}, $k_{12}=0$. But then $\gamma_{1}(0,0)$ is 
        $\RR$--representable. Similarly, since 
        $\gamma_2(\tilde t_0,0)$ satisfies (Hyp),
        it follows that $\gamma_2(0,0)$ also satisfies (Hyp). 
        So
        $\cF(A_{\min})$ and $\cH(A_{\min})$
        admit a $\cZ(x+y-xy)$--rm and a $\RR$--rm, respectively.
        Then \eqref{181123-0849}
        implies the existence of a $(\Rank \widetilde{\mc M}(k;\beta))$--atomic $\cZ(p)$--rm for $\beta$.
    \medskip
        
    \noindent\textbf{Case 4.2: $u_{\max}>0$.} 
        If $t_{\max}=0$, then $\tilde t_0=0$. Since $\eta=0$,
        $\gamma_{2}(0,\tilde u_0)$ can satisfy (Hyp) only if 
        $\tilde u_0=0$ (see \eqref{rank-R1-parabolic}).
        But $\gamma_{1}(0,0)$ cannot be $\RR$--representable if $t_{\max}=0$,
        since then also $u_{\max}$ should be 0 (see Theorem \ref{Hamburger}).
        It follows that $t_{\max}>0$.
        We separate two subcases according to the value of $k_{12}$.
        \smallskip
    
    \noindent \textbf{Case 4.2.1: $k_{12}=0$}.
        In this case $\gamma_{1}(t_{\max},u_{\max})$ is $\RR$--representable by Theorem \ref{Hamburger}
        and 
    $\gamma_{2}(t_{\max},u_{\max})$ satisfies (Hyp).
    We have
            \begin{align}
            \label{rank-computation}
            \begin{split}
            \Rank \cH(\mc G(t_{\max},u_{\max}))
            +
            \Rank \cF(\mc G(t_{\max},u_{\max}))
            &=
            \Rank H_{22}+(\Rank \cF(A_{\min})+2)\\
            &=
            \Rank \cH(A_{\min})+\Rank \cF(A_{\min})\\
            &=
            \Rank \widetilde {\mc M}(k;\beta),
            \end{split}
            \end{align}
        where we used \eqref{rank-R1-parabolic}, \eqref{rank-R2-parabolic} in the first, \eqref{rank-R1-parabolic} in the second and \eqref{181123-0849} in the third equality.
    So there exists a $(\Rank \widetilde{\mc M}(k;\beta))$--atomic $\cZ(p)$--rm for $\beta$.
        \medskip
    
    \noindent \textbf{Case 4.2.2: $k_{12}\neq 0$.}
     We separate two cases according to $\Rank \cH(A_{\min})$, which can be either $k$ or $k+1$ (using \eqref{rank-R2-parabolic} and 
     $H_{22}$ is pd, $k_{12}\neq 0$).
     \medskip

    \noindent \textbf{Case 4.2.2.1: $\Rank \cH(A_{\min})=k$.}
    By \eqref{rank-R2-parabolic}, it follows that $t_{\max}u_{\max}=k_{12}^2$
    (plug $t=0$ and $u=0$ in \eqref{rank-R2-parabolic}) and hence 
    $\mc R_2=\{(0,0)\}$. This implies that under the assumptions of this case $(0,0)$ is the only candidate for $(\tilde t_0,\tilde u_0)$, which means that there exists a $(\Rank \widetilde{\mc M}(k;\beta))$--atomic $\cZ(p)$--rm for $\beta$.
    \medskip

    \noindent \textbf{Case 4.2.2.2: $\Rank \cH(A_{\min})=k+1$.}
     By \eqref{rank-R2-parabolic}, it follows that $t_{\max}u_{\max}>k_{12}^2$
    (plug $t=0$ and $u=0$ in \eqref{rank-R2-parabolic}).
    We separate two cases according to whether 
        $(0,0)$ is a good choice for $(\tilde t_0,\tilde u_0)$
        in \eqref{measure-cond-hyp-2} or not.
    \medskip

    \noindent \textbf{Case 4.2.2.2.1: $(0,0)$ is a good choice for $(\tilde t_0,\tilde u_0)$ in \eqref{measure-cond-hyp-2}.} In this case there exists a $(\Rank \widetilde{\mc M}(k;\beta))$--atomic $\cZ(p)$--rm for $\beta$. 
    \medskip

    \noindent \textbf{Case 4.2.2.2.2: $(0,0)$ is not a good choice for $(\tilde t_0,\tilde u_0)$ in \eqref{measure-cond-hyp-2}.}
    Note that $A_{\gamma_2(t,u)}$ satisfies (Hyp) precisely for $t>0$ and $u>0$. 
     By Lemma \ref{lemma:claim-3}, there is a point $(\tilde t,\tilde u)\in \mc R_2\cap (\RR_+)^2$,
             such that $0<\tilde t\tilde u$ (since $p_{\max}>0$) and 
             $(t_{\max}-\tilde t)(u_{\max}-\tilde u)=k_{12}^2$.
            By Theorem \ref{Hamburger}, 
            $\gamma_{1}(\tilde t,\tilde u)$ is $\RR$--representable,
            since $A_{\gamma_{1}(\tilde t,\tilde u)}$ is psd and
            $A_{\gamma_{1}(\tilde t,\tilde u)}(k)$ is pd.
    Note also that
    $\gamma_{2}(\tilde t,\tilde u)$ satisfies (Hyp).
    Hence, $(\tilde t,\tilde u)$ 
    is a good choice for $(\tilde t_0,\tilde u_0)$ in \eqref{measure-cond-hyp-2}.
    We have
            \begin{align*}
            \Rank \cH(\mc G(\tilde t,\tilde u))
            +
            \Rank \cF(\mc G(\tilde t,\tilde u))
            &\underbrace{=}_{\substack{\eqref{rank-R1-parabolic},\\\eqref{rank-R2-parabolic}}}
            (\Rank H_{22}+1)+(\Rank \cF(A_{\min})+2)\\
            &=
            \Rank \cH(A_{\min})+\Rank \cF(A_{\min})+1\\
            &\underbrace{=}_{\eqref{181123-0849}}
            \Rank \widetilde {\mc M}(k;\beta)+1,
            \end{align*}
        where in the second equality we used 
        the assumption $\Rank \cH(A_{\min})=k+1$. 
    So there exists a $(\Rank \widetilde{\mc M}(k;\beta)+1)$--atomic $\cZ(p)$--rm for $\beta$.
            It remains to show that
            there does not exist a 
            $(\Rank \widetilde {\mc M}(k;\beta))$--atomic rm.
            By the first sentence of this case above, note that
            if $(t',u')$ is a good choice for $(\tilde t_0,\tilde u_0)$
            in \eqref{measure-cond-hyp-2},
            then
            $\Rank \cF(\mc G(t',u'))=\Rank \cF(A_{\min})+2$. 
            Since $k_{12}\neq 0$, it follows by \eqref{rank-R2-parabolic}
            that $\Rank \cH(\mc G(t',u'))\geq \Rank H_{22}+1=\Rank\cH(A_{\min})-1$. 
            Hence, 
            \begin{align*}
            \Rank \cH(\mc G(t',u'))+
            \Rank \cF(\mc G(t',u'))
            &\geq 
            (\Rank \cH(A_{\min})-1)
            +
            (\Rank \cF(A_{\min})+2)\\
            &\underbrace{=}_{\eqref{181123-0849}}\Rank \widetilde{\mc M}(k;\beta)+1, 
            \end{align*}
            
\noindent\textbf{Case 5: $F_{22}$ and $H_{22}$ are pd, 
        $\eta\neq 0$.}
           We separate two cases according to the value of 
    $k_{12}$.
\medskip

\noindent{\textbf{Case 5.1: $k_{12}=0$.}}
            As in Case 4.2.1 above, $(t_{\max},u_{\max})$
            is a good choice for $(\tilde t_0,\tilde u_0)$ in \eqref{measure-cond-hyp-2}.
            Since $\eta\neq 0$, \eqref{form-of-R1-parabolic} implies that 
            $t_{\max}>0$ and $u_{\max}>0$. Then \eqref{rank-R2-parabolic} implies that 
            $\Rank \cH(\widehat A_{\min})=k+1$.
            By \eqref{K-eta-no-eta}, 
            $\cH(A_{\min})\big/{H_{22}}=
            \begin{pmatrix}
            t_{\max} & -\eta \\ -\eta & u_{\max}     
            \end{pmatrix}$ and hence
            $\Rank \cH(A_{\min})=
            \{k,k+1\}$. 
            We separate two cases according to $\Rank \cH(A_{\min})$.
        \medskip

\noindent{\textbf{Case 5.1.1: $\Rank \cH(A_{\min})=k+1$.}}
        By the same computation as in \eqref{rank-computation}, there is a 
             $(\Rank \widetilde{\mc M}(k;\beta))$--atomic $\cZ(p)$--rm for $\beta$ in this case.
         \medskip
   
\noindent{\textbf{Case 5.1.2: $\Rank \cH(A_{\min})=k$.}}
    Since $H_{22}$ is pd and $\Rank \cH(A_{\min})=k$, it follows that 
    $t_{\max}u_{\max}=\eta^2$. Hence, $(t_{\max},u_{\max})$ is the only candidate for $(\tilde t_0,\tilde u_0)$ in \eqref{measure-cond-hyp-2}.
    By the same computation as in \eqref{rank-computation-v2},
    $\beta$ has a $(\Rank \widetilde {\mc M}(k;\beta))$--atomic $\cZ(p)$--rm.\\

{
\noindent{\textbf{Case 5.2: $k_{12}\neq 0$.}}
            We separate two cases according to whether $A_{\min}$ satisfies (Hyp)$_1$ or not.\\

\noindent{\textbf{Case 5.2.1: 
    $A_{\min}$ satisfies (Hyp)$_1$.}}
            By analogous reasoning as in 
            Case 2.3.2.2 of the proof of Theorem \ref{sol:x-y-axy}, one of $(t_{\pm,\eta^2},u_{\pm,\eta^2})$
            is a good choice for $(\tilde t_0,\tilde u_0)$ in \eqref{measure-cond-hyp-2}.
            By \eqref{form-of-R2-parabolic}, $t_{\max}>0$ and $u_{\max}>0$ 
            (since $k_{12}\neq 0$).
            By \eqref{K-eta-no-eta}, 
            $
            \cH(A_{\min})\big/{H_{22}}=
            \begin{pmatrix}
            t_{\max} & k_{12}-\eta \\ k_{12}-\eta & u_{\max}     
            \end{pmatrix}.
            $
            Since $\cH(A_{\min})\big/{H_{22}}$ is not a zero matrix,
            we have $\Rank \cH(A_{\min})\in \{k,k+1\}$.
            Let us define the number
            \begin{equation}
            \label{def:R}
            R:=
            \left\{
            \begin{array}{rl}
                \Rank \widetilde {\mc M}(k;\beta),&
                \text{if }\Rank \cH(A_{\min})=k+1,\\[0.5em]
                \Rank \widetilde {\mc M}(k;\beta)+1,&
                \text{if }\Rank \cH(A_{\min})=k.
            \end{array}
            \right.
            \end{equation}
            We have
            \begin{align}
            \label{rank-computation-v3}
            \begin{split}
            &\Rank \cH(\mc G(t_{\pm,\eta^2},u_{\pm,\eta^2}))
            +
            \Rank \cF(\mc G(t_{\pm,\eta^2},u_{\pm,\eta^2}))\\
            \underbrace{=}_{\substack{\eqref{rank-R1-parabolic},\\\eqref{rank-R2-parabolic}}}&
            (\Rank H_{22}+1)
            +
            (\Rank \cF(A_{\min})+1)
            \underbrace{=}_{\eqref{181123-0849}}R.
            \end{split}
            \end{align}
            So there is an 
            $R$--atomic 
            $\cZ(p)$--rm for $\beta$.

\medskip

\noindent \textbf{Case 5.2.2: $A_{\min}$ does not satisfy (Hyp)$_1$.}
            Let $R$ be as in \eqref{def:R}.
            We will show that in this case
            there is a $(R+1)$--atomic $\cZ(p)$--rm
            and 
            there does not exist an 
            $R$--atomic $\cZ(p)$--rm.
            Since $\eta\neq 0$, if $\cF(\mc G(t',u'))$ is psd, 
            it follows that $t'u'\geq \eta^2$ by \eqref{form-of-R1-parabolic}.
            By the same argument as in the first paragraph
            of Case 2.3.2.2 of the proof of Theorem \ref{sol:x-y-axy},
            if one of $(t_{\pm,\eta^2},u_{\pm,\eta^2})$
            is a good choice for $(\tilde t_0,\tilde u_0)$ in \eqref{measure-cond-hyp-2},
            then $A_{\min}$ satisfies (Hyp)$_1$, which is a contradiction.
            If $\eta^2=p_{\max}$ from Lemma \ref{lemma:claim-3}, then $(t_{\pm,\eta^2},u_{\pm,\eta^2})$
            are the only candidates for a good choice for $(\tilde t_0,\tilde u_0)$ in \eqref{measure-cond-hyp-2} and hence $\beta$ does not have a $\cZ(p)$--rm. It follows that $\eta^2<p_{\max}$. By Lemma \ref{lemma:claim-3}, there exists
            $(\breve t,\breve u)\in \mc R_2\cap (\RR_+)^2$,
             such that $\breve t\breve u=\frac{\eta^2+p_{\max}}{2}$ and 
             $(t_{\max}-\breve t)(u_{\max}-\breve u)=k_{12}^2$.
            But then $\Rank \cF(\mc G(\breve t,\breve u))=2k+1$,
            and $\gamma_2(\breve t,\breve u)$
            satisfies (Hyp). Also, 
            $\gamma_{1}(\breve t,\breve u)$ is $\RR$--representable,
            since $A_{\gamma_{1}(\breve t,\breve u)}\succeq 0$ and
            $A_{\gamma_{1}(\breve t,\breve u)}(k)$ is pd (since $\breve t<t_{\max}$). 
            Repeating the calculation \eqref{rank-computation-v3}
            and using that 
            $\Rank \cF(\cG(\breve t,\breve u))=\Rank \cF(A_{\min})+2$
            instead of $\Rank \cF(\cG(t_{\pm,\eta^2}, u_{\pm,\eta^2}))=\Rank \cF(A_{\min})+1$,
            we get 
            $\Rank \cH(\cG(\breve t,\breve u))+\Rank\cF(\cG(\breve t,\breve u))=R+1$
            and $\beta$ admits an
            $(R+1)$--atomic $\cZ(p)$--rm.
            It remains to show that
            there does not exist an
            $R$--atomic rm.
            As above, if $\cF(\mc G(t',u'))$ is psd and has a $\cZ(x+y-xy)$--rm,  
            it follows that $t'u'> \eta^2$, which means that
            $\Rank \cF(\mc G(t',u'))=\Rank \cF(A_{\min})+2$.
            Since $k_{12}\neq 0$, $\Rank \cH(\mc G(t',u'))\geq \Rank H_{22}+1$
            by \eqref{rank-R2-parabolic}. Hence, 
            \begin{align*}
            \Rank \cH(\mc G(t',u'))+
            \Rank \cF(\mc G(t',u'))
            &\geq 
            (\Rank H_{22}+1)
            +
            (\Rank \cF(A_{\min})+2)\\
            &\underbrace{=}_{\eqref{181123-0849}} R+1.
            \end{align*}

    It remains to establish the moreover part.
    Note that in the case where $\Rank \widetilde{\mc M}(k;\beta)+2$ atoms might be needed, $\cH(A_{\min})$ is not pd.
    Since for a $p$--pure sequence $\beta$
    with $\widetilde{\mc M}(k;\beta)\succeq 0$, \eqref{181123-0849} implies that $F_{22}$ and $\cH(A_{\min})$ are pd, 
    the existence of a $\cZ(p)$--rm
    implies the existence of at most
    $(\Rank \widetilde{\mc M}(k;\beta)+1)$--atomic $\cZ(p)$--rm.
    }
}
\end{proof}

\subsection{Example}{\footnote{The \textit{Mathematica} file with numerical computations can be found on the link \url{https://github.com/ZalarA/TMP_cubic_reducible}.}} 
\label{ex:hyperbolic-type-2} In this subsection we demonstrate 
the use of Theorems \ref{sol:x-y-axy} and 
    \ref{thm:hyperbolic-2-minimal-measures}
on a numerical example.

Let $\beta$ be a bivariate degree 6 sequence given by
\begin{spacing}{1.3}
$\beta_{00} = 1$,

$\beta_{10} =\frac{11}{50}$,
$\beta_{01}  = -\frac{13}{100}$

$\beta_{20}  = \frac{12397}{18000},$
$\beta_{11} = -\frac{11}{100},$
$\beta_{02}  =\frac{2947}{18000},$

$\beta_{30}  =\frac{1001}{1250},$
$\beta_{21}  =-\frac{383}{18000}$
$\beta_{12}  =\frac{967}{18000},$
$\beta_{03}  =-\frac{1117}{10000},$

$\beta_{40}  =\frac{117670993}{64800000},$
$\beta_{31}  =-\frac{1843}{90000},$
$\beta_{22}  =\frac{73}{2250},$
$\beta_{13}  =-\frac{2609}{45000},$
$\beta_{04}  =\frac{7105993}{64800000},$

$\beta_{50} =\frac{100001}{31250},$
$\beta_{41} =-\frac{295967}{64800000},$
$\beta_{32} =\frac{359}{30000},$
$\beta_{23} =-\frac{383}{15000},$
$\beta_{14} =\frac{3349033}{64800000},$
$\beta_{05} =-\frac{103093}{1000000},$

$\beta_{60}=\frac{1540453883617}{233280000000},$
$\beta_{51}=-\frac{1469467}{324000000},$
$\beta_{42}=\frac{479473}{64800000},$
$\beta_{33}=-\frac{407}{30000},$
$\beta_{24}=\frac{1694473}{64800000},$

\hfill 
$\beta_{15}=-\frac{16656967}{324000000},$
$\beta_{06}=\frac{23769383617}{233280000000}.$
\end{spacing}
	We will prove below that $\beta$ 
    admits a $9$--atomic $\cZ(p)$--rm by applying Theorems \ref{sol:x-y-axy} and 
    \ref{thm:hyperbolic-2-minimal-measures}.
	It is easy to check that $\widehat{\mc M}(3)$ is psd
	and satisfies only one column relation 
    $YX+Y^2-XY^2=\mathbf{0}$.
        It turns out that $\eta=-\frac{51255911}{6577059124404}$,
        $t_{\max}=\frac{1827880655851}{20096569546790}$,
        $u_{\max}=\frac{272763812083768883}{833444932244474880}$
        and $k_{12}=-\frac{9}{55}$.
        Computing $t_{-,\eta^2},u_{-,\eta^2}$ we get
        \begin{tiny}
        \begin{align*}
            u_{-,\eta^2}
            &=-\frac{49 (-18583967869070689172740711 + 
    1644264781101 \sqrt{127741799953693985969528905})}{55397740704244472768199800832},\\
            t_{-,\eta^2}
            &=\frac{49 (18583967869070689172740711 + 
   1644264781101 \sqrt{
    127741799953693985969528905)}}{199331524341418907147142346748}.
        \end{align*}
        \end{tiny}
    It is easy to check that $\cH(\cG(t_{-,\eta^2},u_{-,\eta^2}))$
    is psd of rank 3 and $\cH(\cG(t_{-,\eta^2},u_{-,\eta^2}))(2)\succ 0$.
    Hence, $\cH(\cG(t_{-,\eta^2},u_{-,\eta^2}))$ admits a 3--atomic $\RR$--rm.
    Moreover, $\cF(\cG(t_{-,\eta^2},u_{-,\eta^2}))$
    satisfies (Hyp)$_1$ and has rank 6. So it admits a $6$--atomic $\cZ(x+y-xy)$--rm, whence $\beta$ has a 9--atomic $\cZ(p)$--rm.
    This also follows from Theorem \ref{thm:hyperbolic-2-minimal-measures}, since $F_{22}\succ 0$, $H_{22}\succ 0$, $A_{\min}$ satisfies (Hyp)$_{1}$ 
    with
    $\Rank\cF(A_{\min})=\Rank\cF(A_{\min})_{\cB\setminus \{X^3\}}=\Rank\cF(A_{\min})_{\cB\setminus \{Y^3\}}=5$
    and $\Rank \cH(A_{\min})=4$.

%% file: Hyperbolic-type-3-new.tex
\section{Hyperbolic type 3 relation: 
            $p(x,y)=y(ay+x^2-y^2)$, 
            $a\notin \RR\setminus\{0\}$.
        }
\label{sec:hyperbolic-type-3}

In this section we solve constructively the $\cZ(p)$--TMP for 
the sequence $\beta=\{\beta_{i,j}\}_{i,j\in \ZZ_+,i+j\leq 2k}$ 
of degree $2k$, $k\geq 3$,
where $p(x,y)$ is as in the title of the section.
The main results are Theorem \ref{221023-1854}, which characterizes concrete numerical conditions for the existence of a $\cZ(p)$--rm
for $\beta$, and Theorem \ref{thm:hyperbolic-3-minimal-measures},
which characterizes the number of atoms needed in a minimal $\cZ(p)$--rm. 
A numerical example demonstrating the main results is presented in Subsection \ref{ex:hyperbolic-type-3}.

\subsection{Existence of a representing measure}
Assume the notation from Section \ref{Section-common-approach}.
If $\beta$ admits a $\cZ(p)$--TMP, then $\mc M(k;\beta)$ 
must satisfy the relations
\begin{equation}
    \label{221023-1844}
	aY^{2+j}X^{i}
        +
        Y^{1+j}X^{2+i}
        =
        Y^{3+j}X^{i}\quad
        \text{for }i,j\in \ZZ_+\text{ such that }i+j\leq k-3.
\end{equation}
	In the presence of all column relations \eqref{221023-1844}, the column space $\cC(\mc M(k;\beta))$ is spanned by the columns in the tuple
    \begin{equation}
    \label{221023-1848}
    \vec{\cT}:=
        (
        \vec{X}^{(0,k)}, 
        Y\vec{X}^{(0,k-1)},
        Y^2\vec{X}^{(0,k-2)}),
    \end{equation}
    where 
    $$
        Y^i\vec{X}^{(j,\ell)}:=(Y^iX^j,Y^iX^{j+1},\ldots,Y^iX^{\ell})
        \quad\text{with }i,j,\ell\in \ZZ_+,\; j\leq \ell,\; i+\ell\leq k.
    $$
     Let $\widetilde \cM(k;\beta)$ 
    be as in 
    \eqref{071123-1939}
    and define
    \begin{equation}
        \label{091123-0719}
            A_{\min}:=A_{12}(A_{22})^{\dagger} (A_{12})^T
        \quad\text{and}\quad
            \widehat A_{\min}
            :=A_{\min}+\eta E_{2,2}^{(k+1)},
    \end{equation}
    where 
    $$
    \eta:=(A_{\min})_{1,3}-(A_{\min})_{2,2}.
    $$
    See Remark \ref{general-procedure} for the explanation of these definitions.
    Let $\cF(\mathbf{A})$ and $\cH(\mathbf{A})$ 
    be as in 
    \eqref{071123-1940}.
Write
\begin{align}
\label{081123-1936}
\begin{split}
        \cH(\widehat A_{\min})
        &:=
            \kbordermatrix{ 
                & \mathit{1} & X& \vec{X}^{(2,k)}\\
            \mathit{1}  & \beta_{0,0}-(A_{\min})_{1,1} & \beta_{1,0}-(A_{\min})_{1,2} & (h_{12}^{(1)})^T\\[0.2em]
            X
                & \beta_{1,0}-(A_{\min})_{1,2}& \beta_{2,0}-(A_{\min})_{1,3} & (h_{12}^{(2)})^T\\[0.2em]
            (\vec{X}^{(2,k)})^T & 
            h_{12}^{(1)} &
            h_{12}^{(2)} &
            H_{22}},\\[0.5em]
H_1&:=\cH(\widehat A_{\min})_{\{\mathit 1\}\cup\vec{X}^{(2,k)}}=
\kbordermatrix{
    & \mathit{1} & \vec{X}^{(2,k)}\\[0.2em]
    \mathit{1} & \beta_{0,0}-(A_{\min})_{1,1} & (h_{12}^{(1)})^T\\[0.2em]
    (\vec{X}^{(2,k)})^T & h_{12}^{(1)} & H_{22}},\\[0.5em]
H_2&:=\cH(\widehat A_{\min})_{\vec{X}^{(1,k)}}=
\kbordermatrix{
    & X & \vec{X}^{(2,k)}\\[0.2em]
    X & \beta_{2,0}-(A_{\min})_{1,3} & (h_{12}^{(2)})^T\\[0.2em]
    (\vec{X}^{(2,k)})^T & h_{12}^{(2)} & H_{22}}.
\end{split}
\end{align}
We define also the matrix function 
    \begin{equation}
        \label{301023-1930}
        \mc G:\RR^2\to S_{k+1},\qquad 
	\mc G(\mathbf{t},\mathbf{u})=
        \widehat A_{\min}
        +\mathbf{t}E_{1,1}^{(k+1)}
        +\mathbf{u}\big(E_{1,2}^{(k+1)}+E_{2,1}^{(k+1)}\big).
    \end{equation}
Let 
\begin{align}
\label{def:R1-and-R2}
    \mc R_1
    &=\big\{(t,u)\in \RR^2\colon \cF(\mc G(t,u))\succeq 0\big\}
    \quad\text{and}\quad
    \mc R_2
    =\big\{(t,u)\in \RR^2\colon \cH(\mc G(t,u))\succeq 0\big\}.
\end{align}
Further, we introduce real numbers
                    \begin{align}
                        \label{101123-1548}
                         \begin{split}
                        t_0
                        &:=\beta_{0,0}-(A_{\min})_{1,1}-
                            (h_{12}^{(1)})^T (H_{22})^{\dagger} h_{12}^{(1)},\\
                        u_0
                        &:=
                        \beta_{1,0}-(A_{\min})_{1,2}
                        -(h_{12}^{(1)})^T (H_{22})^{\dagger} h_{12}^{(2)},
                        \end{split}
                    \end{align}
                    and a function
                    \begin{equation}
                        \label{101123-1550}
                        h(\mathbf{t})=
                        \sqrt{
                        (H_1/H_{22}-\mathbf{t})
                        (H_2/H_{22})
                        }.
                    \end{equation}
It turns out \cite[Theorem 5.1, Claims 1--2]{YZ24} that
\begin{align}
    \label{forms-of-R1-R2}   
    \begin{split}
    \mc R_1
    &=
    \left\{
    \begin{array}{rl}
    \big\{
    (t,u)\in \RR^2\colon 
    t\geq 0, 
    u\in \left[-\sqrt{\eta t},\sqrt{\eta t}\right]
    \big\},&
    \text{if }\eta\geq 0,\\[0.3em]
    \varnothing,&
    \text{if }\eta<0,
    \end{array}
    \right.\\
    \mc R_2
    &=
    \left\{
    \begin{array}{rl}
    \big\{
    (t,u)\in \RR^2\colon 
    t\leq t_0, 
    u\in [u_0-h(t),u_0+h(t)]
    \big\},&\text{if }H_2\succeq 0,\\[0.3em]
    \varnothing,&\text{if }H_2\not\succeq 0.
    \end{array}
    \right.
    \end{split}
\end{align}
Therefore $\mc R_1$ has one of the following forms:
\begin{center}
    \begin{tabular}{lr}
    \includegraphics[width=5cm]{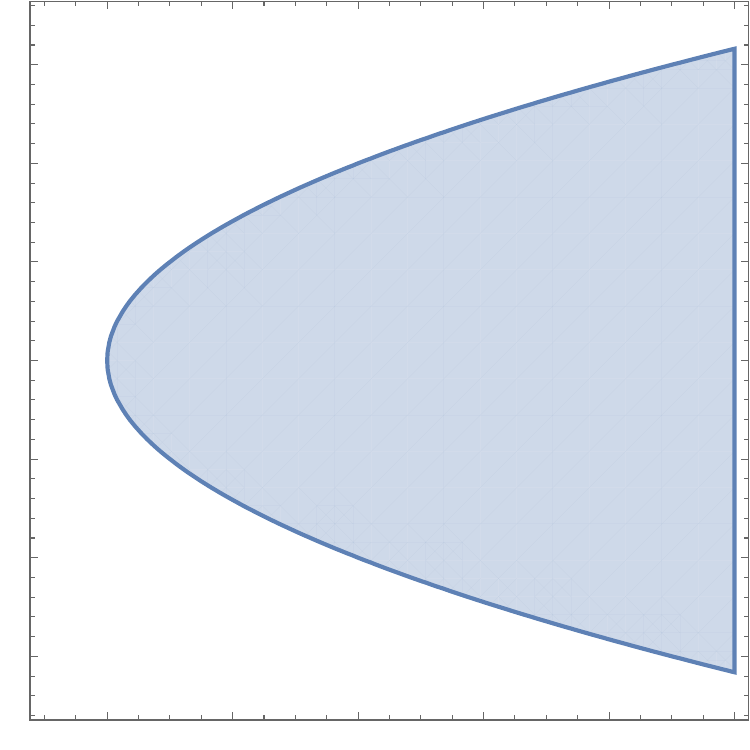}
    &
    \includegraphics[height=4cm]{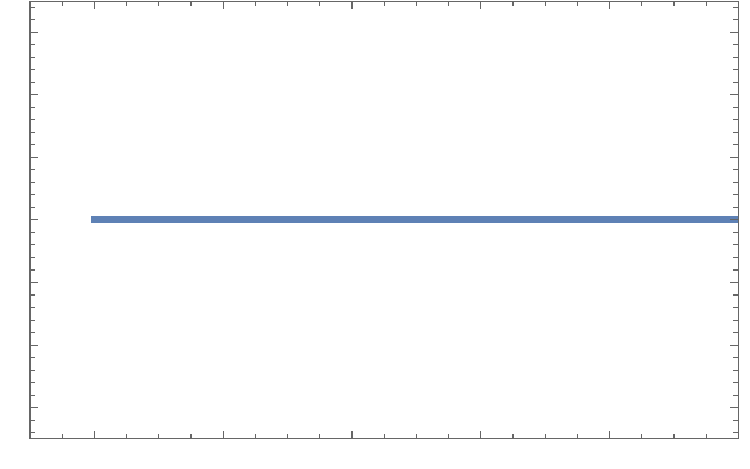}
    \end{tabular}
    \end{center}
where the left case occurs if $\eta>0$, the right if $\eta=0$, while the case $\eta<0$ gives an empty set;
and $\mc R_2$ can be one of the following:
    \begin{center}
    \begin{tabular}{lr}
    \includegraphics[width=5cm]{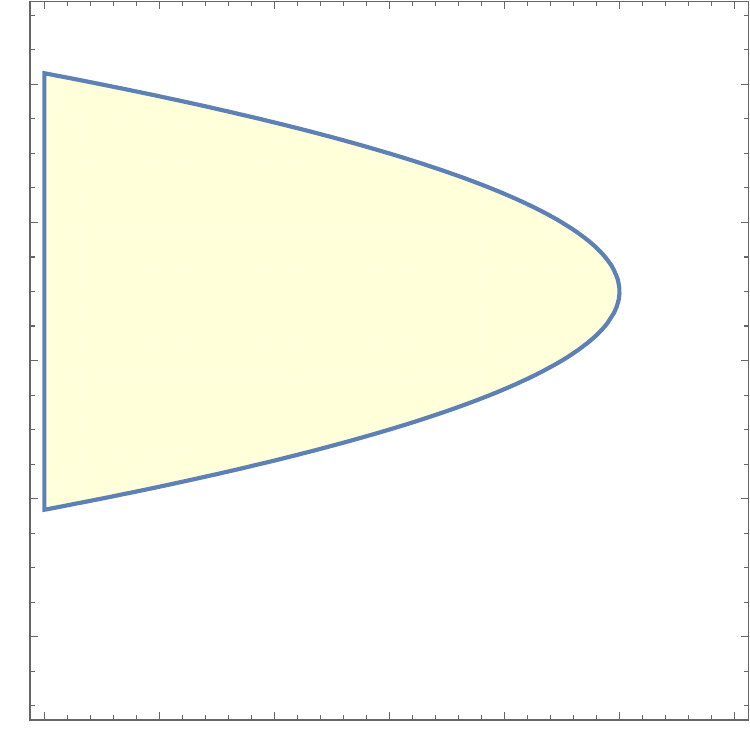}
    &
    \includegraphics[height=4cm]{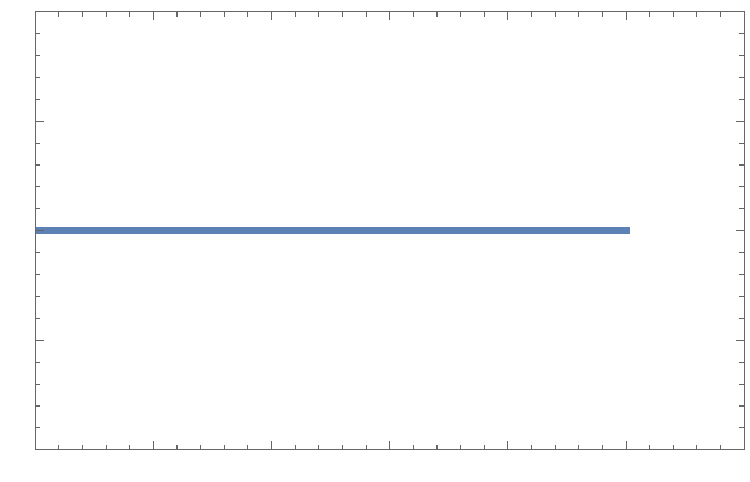}
    \end{tabular}
    \end{center}
where the left case occurs if $H_2/H_{22}>0$, the right if $H_2/H_{22}=0$, while the case 
$H_2/H_{22}<0$ gives an empty set.

By Lemmas \ref{071123-0646}--\ref{071123-2008} and Remark \ref{general-procedure}, the existence of a $\mc Z(p)$--rm for $\beta$ is equivalent to: 
\begin{align}
\label{251023-1603-v2}
\begin{split}
&\widetilde{\cM}(k;\beta)\succeq 0, 
\text{ the relations }
\eqref{221023-1844}\text{ hold 
and }\\
&\text{there exists } (\tilde t_0,\tilde u_0)\in \mc R_1\cap\mc R_2 \text{ such that }
\cF(\mc G(\tilde t_0,\tilde u_0)) 
\text{ and } \cH(\mc G(\tilde t_0,\tilde u_0))\\
&
\text{ admit a }
\cZ(ay+x^2-y^2)\text{--rm and a }\RR\text{--rm, respectively.}
\end{split}
\end{align}

Define the sequence
					\begin{align}
                        \label{def:gamma-hyp3}
                        \begin{split}
						\gamma(\mathbf{t},\mathbf u)
							:=\;&\big(
                                \beta_{0,0}-(A_{\min})_{1,1}-\mathbf t,
                                \beta_{1,0}-(A_{\min})_{1,2}-\mathbf u,
                                \beta_{2,0}-\beta_{0,2}+a\beta_{0,1},\\
                                &
                                \;\;\beta_{3,0}-\beta_{1,2}+a\beta_{1,1},
                                \ldots,
                                \beta_{2k,0}-\beta_{2k-2,2}+a\beta_{2k-2,1}\big).
                        \end{split}
					\end{align}
Note that 
$\cH(\cG(\mathbf t,\mathbf u))=A_{\gamma(t,u)}$
(see \eqref{vector-v} and Remark \ref{general-procedure}.\eqref{general-procedure-pt1}).\\

Write
$$
    \cB:=\{\textit{1},X,X^2,\ldots,X^k,Y,YX,\ldots,YX^{k-1}\}.
$$
We say the matrix $A\in S_{k+1}$ satisfies \textbf{the property 
$(\widetilde{\text{Hyp}})$} if
\begin{equation}
    \underbrace{\Rank \cF(A)=
    \Rank \cF(A)_{\cB\setminus\{X^{k}\}}=
    \Rank \cF(A)_{\cB\setminus\{YX^{k-1}\}}}_{(\widetilde{\text{Hyp}})_1}
    \quad
    \text{or}
    \quad
    \underbrace{\Rank \cF(A)=2k+1}_{(\widetilde{\text{Hyp}})_2}.
\end{equation}


We denote by $\partial \mc R_i$ and $\interior{\mc R}_i$ the topological boundary and the interior of the set $\mc R_i$, respectively.\\

The solution to the $\cZ(p)$--TMP is the following.

    \begin{theorem}
		\label{221023-1854}
	Let $p(x,y)=y(ay+x^2-y^2)$,
        $a\in \RR\setminus\{0\}$,
	and $\beta=(\beta_{i,j})_{i,j\in \ZZ_+,i+j\leq 2k}$, where $k\geq 3$.
	Assume also the notation above.
	Then the following statements are equivalent:
	\begin{enumerate}	
		\item\label{221023-1854-pt1} 
                $\beta$ has a $\cZ(p)$--representing measure.
                \smallskip
		\item\label{221023-1854-pt3}  
  $\widetilde{\mc{M}}(k;\beta)$ is positive semidefinite,
  the relations \eqref{221023-1844} hold,
        $A_{\min}$ either satisfies ($\widetilde{\text{Hyp}}$)$_1$
        or the rank equality $\Rank \cF(A_{\min})=2k-1$,
	and 
	one of the following statements is true:
		\begin{enumerate}
			\item
				\label{221023-1857-pt3.1}
                    $\eta=0$, $A_{\min}$ satisfies ($\widetilde{\text{Hyp}}$)$_1$ and 
                    $\gamma(0,0)$
                    is $\RR$--representable.
                    \smallskip
                \item
			\label{221023-1857-pt3.2}
                    $\eta>0$
                    and
                    one of the following holds:
                    \smallskip
                    \begin{enumerate}
                    \item 
                    \label{221023-1857-pt3.2.1}
                    The set $\partial\mc R_1\cap\partial\mc R_2$ has two elements and $H_2$ is positive definite.
                    \smallskip
                    \item
                    \label{221023-1857-pt3.2.2}
                    $\partial\mc R_1\cap\partial\mc R_2=\{(\tilde t,\tilde u)\}$
                    and 
                    there exist 
                    $(\hat t,\hat u)\in \{((\tilde t,\tilde u),(t_0,\tilde u))\}$
                    such that $\gamma(\hat t,\hat u)$ is $\RR$--representable
                    and $\cF(\cG(\hat t, \hat u))$ satisfies $(\widetilde{\text{Hyp}})$.
                    \end{enumerate}
	\end{enumerate}
 \end{enumerate}
        \smallskip
\end{theorem}

Before we prove Theorem \ref{221023-1854} we need few 
lemmas.
Their statements and the proofs coincide verbatim with \cite[Theorem 5.1, Claims 1--5]{YZ24}, 
but we state them for easier readability.\\

Let $\mc R_1, \mc R_2$ be as in \eqref{def:R1-and-R2}.
Claims 1 and 2 below describe ranks of 
    $\cF(\cG(t,u))$ and $\cH(\cG(t,u))$
for various choices of $(t,u)$ in $\mc R_1$ and $\mc R_2$,
respectively. 

\begin{lemma}[{\cite[Theorem 5.1, Claim 1]{YZ24}}]
\label{Hyp3-claim1}
    Assume that $\widetilde{\mc{M}}(k;\beta)\succeq 0$.
    Then $\mc R_1$ is as in \eqref{forms-of-R1-R2} above.
If $\eta\geq 0$, then we have
\begin{align}
    \label{rank-R1}
    \Rank \cF(\mc G(t,u))=
    \left\{
    \begin{array}{rl}
            \Rank \cF(A_{\min}),&    
            \text{if }
                t=0, 
                \eta=0,\\[0.3em]
            \Rank \cF(A_{\min})+1,&    
            \text{if }
                (t>0 \text{ or }\eta>0)
                \text{ and }
                u\in \{-\sqrt{\eta t},\sqrt{\eta t}\},\\[0.3em]
            \Rank \cF(A_{\min})+2,&    
            \text{if }
                t>0,\eta>0,
                u\in \left(-\sqrt{\eta t},\sqrt{\eta t}\right),
    \end{array}
    \right.
\end{align}
where $A_{\min}$ is as in \eqref{091123-0719}.
\end{lemma}

\begin{lemma}[{\cite[Theorem 5.1, Claim 2]{YZ24}}]
\label{Hyp3-claim2}
    Assume that $\widetilde{\mc{M}}(k;\beta)\succeq 0$.
    Let $t_0, u_{0}$, $h(\mathbf{t})$ be as in
    \eqref{101123-1548}, \eqref{101123-1550}.
If $H_2\succeq 0$, then we have 
\begin{align}
    \label{rank-R2}
    \Rank \cH(\mc G(t,u))=
    \left\{
    \begin{array}{rl}
            \Rank H_{2},&    
            \text{for }
                t=t_0, 
                u=u_0,\\[0.3em]
            \Rank H_{22}+1,&    
            \text{for }
                t<t_0, 
                u\in \{u_0-h(t),u_0+h(t)\},\\[0.3em]
            \Rank H_{22}+2,&    
            \text{for }
                t<t_0, 
                u\in (u_0-h(t),u_0+h(t)).
    \end{array}
    \right.
\end{align}
\end{lemma}

\begin{lemma}[{\cite[Theorem 5.1, Claim 3]{YZ24}}]
\label{Hyp3-claim3}
    Assume that $\widetilde{\mc{M}}(k;\beta)\succeq 0$
    and $\eta=0$. Then $(0,0)\in \partial\mc R_1\cap\mc R_2$.
\end{lemma}

\begin{lemma}[{\cite[Theorem 5.1, Claim 4]{YZ24}}]
\label{Hyp3-claim4}
    Assume that $\widetilde{\mc{M}}(k;\beta)\succeq 0$
    and $\eta>0$. 
    Then:
\begin{itemize}
\item
The set $\partial\mc R_1\cap\partial \mc R_2$ has at most 2 elements. 
\smallskip
\item 
$\mc R_1\cap \mc R_2\neq \varnothing$ 
if and only if 
$\partial\mc R_1\cap\partial \mc R_2\neq \varnothing.$ 
\smallskip
\item 
If $\partial\mc R_1\cap\partial \mc R_2$ has two elements, then $H_2/H_{22}>0$. 
\smallskip
\item 
If $\partial\mc R_1\cap\partial \mc R_2$ has one element, which we denote by $(\tilde t,\tilde u)$, then one of the following holds:\smallskip
\begin{itemize}
    \item 
        $
        \mc R_1 \cap \mc R_2=
        \partial\mc R_1\cap\partial \mc R_2$.
        \smallskip
    \item
    $\partial \mc R_2=\mc R_2=\{(t,u_0)\colon t\leq t_0\}$
    and 
    $\partial\mc R_1\cap \partial\mc R_2\subsetneq \mc R_1\cap \mc R_2= 
    \{(t,u_0)\colon \tilde t\leq t\leq t_0\}$.
\end{itemize}
\end{itemize}
\end{lemma}

\begin{lemma}[{\cite[Theorem 5.1, Claim 5]{YZ24}}]
\label{Hyp3-claim5}
    Assume that $\widetilde{\mc{M}}(k;\beta)\succeq 0$.
Let $H_2$ (see \eqref{081123-1936}) be positive definite, $(t_1,u_1)\in \partial \mc R_2, (t_2,u_2)\in \partial \mc R_2$
and $u_1\neq u_2$.
Then at least one of
$\cH(\mc G(t_1,u_1))$ or
$\cH(\mc G(t_2,u_2))$ 
admits a $\RR$--representing measure.
\end{lemma}

\begin{proof}[Proof of Theorem \ref{221023-1854}]
First we prove the implication
$\eqref{221023-1854-pt1}\Rightarrow\eqref{221023-1854-pt3}$.
There exists $(\tilde t_0,\tilde u_0)\in \mc R_1\cap \mc R_2$
satisfying \eqref{251023-1603-v2}.
In particular,
    $\mc R_1\neq \varnothing$
    and 
    by \eqref{forms-of-R1-R2}, $\eta\geq 0$.
    Since $\cF(\mc G(\tilde t_0,\tilde u_0))$ has a $\cZ(ay+x^2-y^2)$--rm,
    it follows, by Theorem \ref{090622-1611}, that
    $\cG(\tilde t_0,\tilde u_0)$ satisfies ($\widetilde{\text{Hyp}}$).
Note that 
    $$
    \cF(\cG(t,u))
    =
    \cF(A_{\min})
    +
    \begin{pmatrix}
        t & u  \\ u & \eta
    \end{pmatrix}
    \oplus 
    \mathbf{0},
    $$
    where $\mathbf{0}$ is a zero matrix of the appropriate size.
    Moreover, by definition of $A_{\min}$, we have
    \begin{align}
    \label{rank-conditions}
    \begin{split}
    \Rank \cF(\cG(t,u))
    &=
    \Rank \cF(A_{\min})
    +
    \Rank 
        \begin{pmatrix}
            t & u \\ 
            u & \eta
        \end{pmatrix},\\
    \Rank \cF(\cG(t,u))_{\cB\setminus \{X^k\}}
    &\leq
    \Rank \cF(A_{\min})_{\cB\setminus \{X^k\}}
    +
    \Rank 
        \begin{pmatrix}
            t & u \\ 
            u & \eta
        \end{pmatrix}\\
    &\leq
    \Rank \cF(A_{\min})
    +
    \Rank 
        \begin{pmatrix}
            t & u \\ 
            u & \eta
        \end{pmatrix},\\
    \Rank \cF(\cG(t,u))_{\cB\setminus \{YX^{k-1}\}}
    &\leq
    \Rank \cF(A_{\min})_{\cB\setminus \{YX^{k-1}\}}
    +
    \Rank 
        \begin{pmatrix}
            t & u \\ 
            u & \eta
        \end{pmatrix}\\
    &\leq
    \Rank \cF(A_{\min})
    +
    \Rank 
        \begin{pmatrix}
            t & u \\ 
            u & \eta
        \end{pmatrix}.
    \end{split}
    \end{align}
    If $\cG(\tilde t_0,\tilde u_0)$
    satisfies ($\widetilde{\text{Hyp}}$)$_1$,
    then all inequalities in \eqref{rank-conditions}
    must be equalities and in particular,
    $A_{\min}$ satisfies ($\widetilde{\text{Hyp}}$)$_1$.
    If $\cG(\tilde t_0,\tilde u_0)$
    satisfies ($\widetilde{\text{Hyp}}$)$_2$,
    then clearly 
        $
            \Rank \cF(A_{\min})=2k-1.
        $
    From now on we separate two cases according to the value of $\eta$.\\

    \noindent
    \textbf{Case 1: $\eta=0$.}
    For $\eta=0$ we have $\widehat A_{\min}=A_{\min}$.
    By \eqref{forms-of-R1-R2}, we have $\tilde u_0=0$ and $\tilde t_0\geq 0$.
    By Lemma \ref{Hyp3-claim1}, 
    $\cG(\tilde t_0,0)$ cannot
    satisfy ($\widetilde{\text{Hyp}}$)$_2$
    and hence it satisfies ($\widetilde{\text{Hyp}}$)$_1$.
    But then by the explanation above, $A_{\min}$
    satisfies ($\widetilde{\text{Hyp}}$)$_1$ and by
    Corollary \ref{090622-1611-hyp-v3}, it has a 
    $\cZ(ay+x^2-y^2)$--rm. By Theorem \ref{Hamburger},
    $\RR$--representability of $\gamma(\tilde t_0,0)$
    implies $\RR$--representability of $\gamma(0,0)$.
    This is Theorem \ref{221023-1854}.\eqref{221023-1857-pt3.1}.\\

    \noindent
    \textbf{Case 2: $\eta>0$.}
        By Lemma \ref{Hyp3-claim4}, 
        $\partial\mc R_1\cap\partial \mc R_2$ has one or two elements.
    We separate two cases according to the number of elements in $\partial\mc R_1\cap\partial \mc R_2$.\\

    \noindent
    \textbf{Case 2.1: 
    $\partial\mc R_1\cap\partial \mc R_2$ has two elements.}
    By Lemma \ref{Hyp3-claim4},
    $H_2/H_{22}>0$.  
    Assume that $H_2\not\succ 0$.
    Then
    there is a nontrivial column relation among columns $X^2,\ldots,X^k$ in $H_2$. By Proposition \ref{extension-principle},
    the same holds for $\cH(\mc G(\tilde t_0,\tilde u_0))$. 
    Let $\sum_{i=0}^{k-2} c_i X^{i+2}=\mathbf{0}$ be the nontrivial column relation in $\cH(\mc G(\tilde t_0,\tilde u_0))$.
    But then
        $\cZ(x^2\sum_{i=0}^{k-2} c_i x^i)=\cZ(x\sum_{i=0}^{k-2} c_i x^i)$
    and it follows by \cite{CF96} that 
    $\sum_{i=0}^{k-2} c_i X^{i+1}=\mathbf{0}$
    is also a nontrivial column relation in $\cH(\mc G(\tilde t_0,\tilde u_0))$.
    Inductively, this implies $H_2/H_{22}=0$,
    which is a contradiction. 
    Hence, $H_2\succ 0$.
    This is the case of
    Theorem \ref{221023-1854}.\eqref{221023-1857-pt3.2.1}.
    \\

    \noindent
    \textbf{Case 2.2: 
    $\partial\mc R_1\cap\partial \mc R_2$ has one element.}
     Let us denote this element by $(\tilde t,\tilde u)$.
    By Lemma \ref{Hyp3-claim4}, 
    $\partial\mc R_1\cap\partial \mc R_2=\mc R_1\cap \mc R_2$
    or
    $\partial \mc R_2=\mc R_2=\{(t,u_0)\colon t\leq t_0\}$
    and 
    $\partial\mc R_1\cap\partial \mc R_2\subsetneq \mc R_1\cap \mc R_2= 
    \{(t,u_0)\colon \tilde t\leq t\leq t_0\}$
    .
    We separate two cases according to these two possibilities.\\

    \noindent
    \textbf{Case 2.2.1: $\partial\mc R_1\cap\partial \mc R_2=\mc R_1\cap \mc R_2$.} In this case 
    $(\tilde t_0,\tilde u_0)=(\tilde t,\tilde u)$ and hence
    $\gamma(\tilde t,\tilde u)$ is $\RR$--representable,
    while by Corollary \ref{090622-1611-hyp-v3},
    $\cF(\mc G(\tilde t,\tilde u))$ 
    satisfies $(\widetilde{\text{Hyp}})$.
    This is the case of Theorem \ref{221023-1854}.\eqref{221023-1857-pt3.2.2}.\\

    \noindent
    \textbf{Case 2.2.2: 
    $\partial \mc R_2=\mc R_2=\{(t,u_0)\colon t\leq t_0\}$
    and 
    $\partial\mc R_1\cap\partial \mc R_2
    \subsetneq \mc R_1\cap \mc R_2= 
    \{(t,u_0)\colon \tilde t\leq t\leq t_0\}$.}
    By \eqref{forms-of-R1-R2}, it follows that $H_2/H_{22}=0$ (see definition
    \eqref{101123-1550} of $h(\mathbf{t})$).
    Since $H_2$ is not pd, Theorem \ref{Hamburger} used for  
    $\cH(\mc G(\tilde t_0,\tilde u_0))$, implies that the last column of $H_2$ is in the
    span of the others. Hence, by Proposition \ref{extension-principle}, the same holds for $\cH(\mc G(\tilde t,\tilde u))$ and $\cH(\mc G(t_0,\tilde u))$, whence 
    $\gamma(\tilde t,\tilde u)$ and $\gamma(t_0,\tilde u)$ are both $\RR$--representable.
    By Lemma \ref{Hyp3-claim1},
    $
    \Rank \cF(\cG(\tilde t,\tilde u))
    =\Rank \cF(A_{\min})+1
    $
    and
    $
    \Rank \cF(\cG(t,\tilde u))
    =\Rank \cF(A_{\min})+2
    $
    for $t\in (\tilde t,t_0]$.
    If $\cG(t,\tilde u)$ satisfies $(\widetilde{\text{Hyp}})_2$
    for some $t\in (\tilde t,t_0]$, then it satisfies 
    $\cG(t,\tilde u)$ satisfies $(\widetilde{\text{Hyp}})_2$
    for every $t\in (\tilde t,t_0]$.
    Similarly, by \eqref{rank-conditions},  
    if $\cG(t,\tilde u)$ satisfies $(\widetilde{\text{Hyp}})_1$
    for some $t\in (\tilde t,t_0]$, then it satisfies 
    $\cG(t,\tilde u)$ satisfies $(\widetilde{\text{Hyp}})_1$
    for every $t\in (\tilde t,t_0]$.
    This holds because validity of $(\widetilde{\text{Hyp}})_1$
    for one $t\in (\tilde t,t_0]$, implies that all inequalities in \eqref{rank-conditions} must be equalities,
    which is true only if
    \begin{align}
    \label{rank-conditions-v5}
    \begin{split}
    &(A_{\min})_{\{1,X\}}=\\
    =&
    \cF(\cG(t,\tilde u))_{\{1,X\},\cB\setminus \{1,X,X^k\}}
    \big(
    \cF(\cG(t,\tilde u))_{\cB\setminus \{1,X,X^k\}}
    \big)^\dagger
    \cF(\cG(t,\tilde u))_{\cB\setminus \{1,X,X^k\},\{1,X\}}\\
    =&
    \cF(\cG(t,\tilde u))_{\{1,X\},\cB\setminus \{1,X,YX^{k-1}\}}
    \big(
    \cF(\cG(t,\tilde u))_{\cB\setminus \{1,X,YX^{k-1}\}}
    \big)^\dagger
    \cF(\cG(t,\tilde u))_{\cB\setminus \{1,X,YX^{k-1}\},\{1,X\}}.
    \end{split}
    \end{align}
    But then \eqref{rank-conditions-v5}
    holds for every $t\in (\tilde t,t_0]$ and consequently
    $(\widetilde{\text{Hyp}})_1$
    holds for $\cG(t,\tilde u)$
    for every $t\in (\tilde t,t_0]$.
    If $\cG(t_0,\tilde u)$ does not satisfy $(\widetilde{\text{Hyp}})$, it does not admit a 
    $\cZ(ay+x^2-y^2)$--rm by Corollary \ref{090622-1611-hyp-v3},
    which further implies that $\cG(\tilde t,\tilde u)$ satisfies $(\widetilde{\text{Hyp}})$.
     This is the case of Theorem \ref{221023-1854}.\eqref{221023-1857-pt3.2.2}.\\
    
    \noindent This concludes the proof of the implication
    $\eqref{221023-1854-pt1}\Rightarrow\eqref{221023-1854-pt3}$
    of Theorem \ref{221023-1854}.\\


It remains to prove the implication 
$\eqref{221023-1854-pt3}\Rightarrow\eqref{221023-1854-pt1}$
of Theorem \ref{221023-1854}.
    We separate four cases according to the assumptions in $\text{Theorem }\ref{221023-1854}.\eqref{221023-1854-pt3}$.\\

    \noindent\textbf{Case 1: Theorem \ref{221023-1854}.\eqref{221023-1857-pt3.1} holds.}
    By Lemma \ref{Hyp3-claim3}, $(0,0)\in \mc R_1\cap \mc R_2$.
    Further, $\eta=0$ implies that 
    $\widehat A_{\min}=A_{\min}=\cG(0,0)$ and hence 
    $\cG(0,0)$ satisfies $(\widetilde{\text{Hyp}})_1$ by assumption.
    By Corollary \ref{090622-1611-hyp-v3}, $\cF(\cG(0,0))$ admits a $\Rank (\cF(A_{\min}))$--atomic
    $\cZ(ay+x^2-y^2)$--rm.
    By assumption, $\gamma(0,0)$ is $\RR$--representable.
    Hence, $(0,0)$ is a good choice for $(\tilde t_0,\tilde u_0)$ in \eqref{251023-1603-v2}. 
    This proves \eqref{221023-1854-pt1} in this case.
    \\

    \noindent\textbf{Case 2: Theorem \ref{221023-1854}.\eqref{221023-1857-pt3.2.1}
    holds.}
    By assumption, $\partial \mc R_1\cap \partial \mc R_2=\{(t_1,u_1),(t_2,u_2)\}$ has two distinct elements. Hence, $\partial \mc R_{2}$ is not a half-line 
    and $\mc R_1\cap\mc R_2$ has a non--empty interior, which is equal to 
    $\interior{\mc R}_1\cap \interior{\mc R}_2$.
    Since $H_2\succ 0$, it follows, by Lemma \ref{Hyp3-claim2},
    that $\cH(\cG(t,u))\succ 0$ for every 
    $(t,u)\in \interior{\mc R}_2$. Hence, 
    $\gamma(t,u)$ is $\RR$--representable for every 
    $(t,u)\in \interior{\mc R}_2$.
    We separate two cases according to which of the assumptions:
    \begin{itemize}
    \item $A_{\min}$ satisfies ($\widetilde{\text{Hyp}}$)$_1$.
    \item $\Rank \cF(A_{\min})=2k-1$.
    \end{itemize}
    holds.\\
    
    \noindent \textbf{Case 2.1: $ A_{\min}$ satisfies ($\widetilde{\text{Hyp}}$)$_1$.}
        By Lemma \ref{Hyp3-claim1}, we have
        $$
        \Rank \cF(\cG(t_1, u_1))
        =
        \Rank \cF(\cG(t_2,u_2))
        =\Rank \cF(A_{\min})+1.
        $$
        We will prove that
        \begin{equation}
        \label{assum:1609}
        \Rank \cF(\cG(t_i,u_i))_{\cB\setminus \{X^k\}}
        =\Rank \cF(A_{\min})_{\cB\setminus \{X^k\}}+1
        \quad \text{for }i=1,2.
        \end{equation}
        If \eqref{assum:1609} is not true, then
        \begin{align*}
        (A_{\min})_{\{1,X\}}
        &\neq  
        \underbrace{\cF(\cG(t_i,u_i)_{\{1,X\},\cB\setminus \{1,X,X^k\}}
    \big(
    \cF(\cG(t_i,u_i))_{\cB\setminus \{1,X,X^k\}}
    \big)^\dagger
    \cF(\cG(t_i,u_i))_{\cB\setminus \{1,X,X^k\},\{1,X\}}}_{\breve A_{\min}}.
        \end{align*}
        Note that the definition of $\breve A_{\min}$ does not depend on $i$, because $t_i$ and $u_i$ do not appear in the corresponding restrictions of $\cF(\cG(t_i,u_i))$.
        Clearly, 
        $$
        \begin{pmatrix}
        t_i & u_i \\ u_i & \eta
        \end{pmatrix}\succeq (A_{\min})_{\{1,X\}}-\breve A_{\min}\succeq 0
        \quad \text{for }i=1,2,
        $$
        whence 
        \begin{equation}
        \label{kernels-inlcusion}
        \ker\big((A_{\min})_{\{1,X\}}-\breve A_{\min}\big)\subseteq 
        \ker\begin{pmatrix}
        t_i &  u_i \\  u_i & \eta
        \end{pmatrix}\quad \text{for }i=1,2.
        \end{equation}
        Since
        \eqref{kernels-inlcusion} holds for $i=1,2$, it follows
        that 
        $\ker\big((A_{\min})_{\{1,X\}}-\breve A_{\min}\big)=\RR^2$,
        which contradicts to 
        $(A_{\min})_{\{1,X\}}\neq  \breve A_{\min}$.
        Hence, \eqref{assum:1609} is true.
        Similarly,
        $$
        \Rank \cF(\cG(t_i,u_i))_{\cB\setminus \{YX^{k-1}\}}
        =\Rank \cF(A_{\min})_{\cB\setminus \{YX^{k-1}\}}+1
        \quad \text{for }i=1,2.
        $$
        So $\cG(t_i,u_i)$ satisfies
        $(\widetilde{\text{Hyp}})_1$ for $i=1,2$.
        By Corollary \ref{090622-1611-hyp-v3},
        $\cF(\cG(t_i,u_i))$ admits a
        $\cZ(ay+x^2-y^2)$--rm for $i=1,2$.
        By Lemma \ref{Hyp3-claim5}, there is $j\in \{1,2\}$
    such that $\cH(\cG(t_j,u_j))$ admits a $\RR$--rm,
    whence $( t_j,u_j)$ is a good choice for $(\tilde t_0,\tilde u_0)$ in \eqref{251023-1603-v2}. 
    This proves \eqref{221023-1854-pt1} in this case.
        \\

    \noindent \textbf{Case 2.2: $\Rank\cF(A_{\min})=2k-1$.}
    By Lemma \ref{Hyp3-claim1}, 
    $\Rank\cF(\cG(t,u))=2k+1$ for every $(t,u)\in \interior{\mc R}_1$.
    By Corollary \ref{090622-1611-hyp-v3},
        $\cF(\cG(t,u))$ admits a
        $\cZ(ay+x^2-y^2)$--rm for every $(t,u)\in \interior{\mc R}_1$.
        Hence, $(t,u)\in \interior{\mc R}_1\cap \interior{\mc R}_2$ is a good choice for $(\tilde t_0,\tilde u_0)$ in \eqref{251023-1603-v2}. 
    This proves \eqref{221023-1854-pt1} in this case.\\

    \noindent\textbf{Case 3: Theorem \ref{221023-1854}.\eqref{221023-1857-pt3.2.2}
    holds.}
    Clearly, one of the points $(\tilde t, \tilde u)$
    or $(t_0,\tilde u)$  
    is a good choice for $(\tilde t_0,\tilde u_0)$ in \eqref{251023-1603-v2}. 
    This proves \eqref{221023-1854-pt1} in this case.\\

        \noindent This concludes the proof of the implication 
$\eqref{221023-1854-pt3}\Rightarrow\eqref{221023-1854-pt1}$
of Theorem \ref{221023-1854}.
    \end{proof}

\subsection{Cardinality of a minimal representing measure}

The following theorem characterizes the cardinality of a minimal measure in case 
$\beta$ admits a $\cZ(p)$--rm.

\begin{theorem}
\label{thm:hyperbolic-3-minimal-measures}
    Let $p(x,y)=y(ay+x^2-y^2)$,
        $a\in \RR\setminus\{0\}$,
	and $\beta=(\beta_{i,j})_{i,j\in \ZZ_+,i+j\leq 2k}$, where $k\geq 3$,
        admits a $\cZ(p)$--representing measure.
        Assume the notation above.
        The following statements hold:
        \begin{enumerate}
        \item There exists at most $(\Rank \widetilde{\mc{M}}(k;\beta)+2)$--atomic $\cZ(p)$--representing 
            measure for $\beta$.
            \smallskip
        \item There is no 
        $\cZ(p)$--representing measure with less than $\Rank \widetilde{\mc{M}}(k;\beta)+2$ atoms
	if and only if 
                $A_{\min}$ does not satisfy $(\widetilde{\text{Hyp}})_1$,
                $\Rank\cF(A_{\min})=2k-1$,
                $\eta>0$,
                $\partial\mc R_1\cap\partial\mc R_2$ has two elements, $H_2$ is positive definite
                and $\Rank \cH(A_{\min})=k$.
                \smallskip
        \item There exists a $\Rank \widetilde{\mc{M}}(k;\beta)$--atomic $\cZ(p)$--representing measure for $\beta$
	   if and only if any of the following holds:
        \smallskip
        \begin{enumerate}
            \item 
                $\eta=0$.
                \smallskip
            \item 
                $\eta>0$,
                $\Card(\partial \mc R_1\cap \partial \mc R_2)=2$,
                $\widehat A_{\min}$ satisfies ($\widetilde{ \text{Hyp}}$)$_1$,
                $\cH(A_{\min})$ is positive definite.
                \smallskip
            \item 
                $\eta>0$,
                $\Card(\partial \mc R_1\cap \partial \mc R_2)=\Card(\mc R_1\cap \mc R_2)=1$
                and the equality
                $\Rank \cH(A_{\min})
                =\Rank H_{22}+2
                $ holds.
                \smallskip
            \item 
                $\eta>0$,
                $\Card(\partial \mc R_1\cap \partial \mc R_2)=1$,
                $\{(\tilde t,\tilde u)\}=\partial \mc R_1\cap \partial \mc R_2\subsetneq \mc R_1\cap \mc R_2$,
                $\cF(\cG(\tilde t,\tilde u))$ admits a $\cZ(ay+x^2-y^2)$--representing measure
                and
                $\cH(A_{\min})=\Rank H_{22}+2$.
        \end{enumerate}
        \end{enumerate}
        
        In particular, a $p$--pure sequence $\beta$ with a 
            measure
admits at most $(3k+1)$--atomic $\cZ(p)$--representing 
            measure.
\end{theorem}

\begin{proof}[Proof of Theorem \ref{thm:hyperbolic-3-minimal-measures}]    
    By Lemma \ref{071123-2008}.\eqref{071123-2008-pt3},
    \begin{equation}
        \label{261023-1536}
        \Rank \widetilde{\mc M}(k;\beta)
        =\Rank \cF(A_{\min})+
        \Rank \cH(A_{\min}).
    \end{equation}
        We observe again the proof of the implication
    $
    \eqref{221023-1854-pt3}
    \Rightarrow
    \eqref{221023-1854-pt1}
    $
    of Theorem \ref{221023-1854}.
    
    In the proof of the implication
    Theorem $\ref{221023-1854}.\eqref{221023-1857-pt3.1}\Rightarrow
    \eqref{251023-1603-v2}$
    we established that
    $\cF(A_{\min})$ and $\cH(A_{\min})$
    admit
    a $\Rank \cF(A_{\min})$--atomic
    and
    a $\Rank \cH(A_{\min})$--atomic rm\textit{s}.
    Using \eqref{261023-1536} it follows that $\beta$ has a 
    $\Rank \widetilde{\mc M}(k;\beta)$--atomic $\cZ(p)$--rm.\\
    
    In the proof of the implication
    Theorem $\ref{221023-1854}.\eqref{221023-1857-pt3.2.1}\Rightarrow
    \eqref{251023-1603-v2}$ we separated two cases:\\
    
    \noindent \textbf{Case 1: $A_{\min}$ satisfies ($\widetilde{\text{Hyp}}$)$_1$.}
    In this case we established that $\gamma(t',u')$ is $\RR$--representable
    for some $(t',u')\in \partial\mc R_1\cap\partial\mc R_2$,
    where $\Rank\cF(\cG(t',u'))=\Rank \cF(A_{\min})+1$
    and $\Rank\cH(\cG(t',u'))=k$.
    Since 
    $\cH(A_{\min})\succeq \Rank\cH(\cG(t',u'))$,
    it follows that 
    \begin{align*}
    \label{rank-equalities-v7}
    \begin{split}
    &\Rank\cF(\cG(t',u'))+
    \Rank\cH(\cG(t',u'))=\\
    =&\left\{
    \begin{array}{rl}
        \Rank \widetilde{\mc M}(k;\beta),&
        \text{if }
        \Rank\cH(\cG(t',u'))=
        \Rank\cH(A_{\min})-1,\\
        \Rank \widetilde{\mc M}(k;\beta)+1,&
        \text{if }
        \Rank\cH(\cG(t',u'))=
        \Rank\cH(A_{\min}).
    \end{array}
    \right.
    \end{split}
    \end{align*}
    
    It remains to show that if there does not exist $(t,u)$,
    which is a good choice for $(\tilde t_0,\tilde u_0)$ in \eqref{251023-1603-v2}, such that 
    $\Rank\cH(\cG(t,u))=
        \Rank\cH(A_{\min})-1$, then
    there is no 
    $\big(\Rank \widetilde{\mc M}(k;\beta)\big)$--atomic $\cZ(p)$--rm. Since $\eta>0$,
    it follows, by Lemma \ref{Hyp3-claim1},
    that $\Rank\cF(\cG(t,u))\geq \Rank \cF(A_{\min}+1$
    for any good choice $(t,u)$.
    Since also  $\Rank\cH(\cG(t,u))\geq 
        \Rank\cH(A_{\min})$,
    it follows that 
     $\Rank\cF(\cG(t,u))+\Rank\cH(\cG(t,u))\geq  
        \Rank  \widetilde{\mc M}(k;\beta)+1$.\\

    \noindent \textbf{Case 2: $\Rank \cF(A_{\min})=2k-1$.}
    If $A_{\min}$ does not satisfy ($\widetilde{\text{Hyp}}$)$_1$, 
    $\cG(t,u)$ does not satisfy ($\widetilde{\text{Hyp}}$)$_1$
    for any $(t,u)\in \RR^2$. So every $(t,u)$ which is a good choice for $(\tilde t_0,\tilde u_0)$, must satisfy 
    $\Rank \cF(\cG(t,u))=2k+1=\Rank \cF(A_{\min})+2$.
    By Lemma \ref{Hyp3-claim5}, there exists
    $(t',u')\in \interior{\mc R}_1\cap \partial{\mc R}_2$,
    such that $\cH(\cG(t',u'))$ admits a $\RR$--rm
    and satisfies 
    $$\Rank\cH(\cG(t',u'))
    =
    \Rank H_2=
    \left\{
    \begin{array}{rl}
    \Rank \cH(A_{\min}),& \text{if }\Rank \cH(A_{\min})=k,\\
    \Rank \cH(A_{\min})-1,& \text{if }\Rank \cH(A_{\min})=k+1.
    \end{array}
    \right.
    $$
    Hence,
    \begin{align*}
    \begin{split}
    &\Rank\cF(\cG(t',u'))+
    \Rank\cH(\cG(t',u'))=\\
    =&\left\{
    \begin{array}{rl}
        \Rank \widetilde{\mc M}(k;\beta)+2,&
        \text{if }
        \Rank\cH(A_{\min})=k,\\
        \Rank \widetilde{\mc M}(k;\beta)+1,&
        \text{if }
        \Rank\cH(A_{\min})=k+1.
    \end{array}
    \right.
    \end{split}
    \end{align*}
    If $(t,u)$ is a good choice for $(\tilde t_0,\tilde u_0)$, then $\Rank\cH(\cG(t,u))\geq k$ (since $H_2\succ 0$) and also $\Rank\cF(\cG(t,u))=\Rank \cF(A_{\min})+2$.
    So 
    \begin{align*}
    &\Rank\cF(\cG(t,u))+
    \Rank\cH(\cG(t,u))\geq\\
    \geq&
    \Rank \cF(A_{\min})+2
    +
    \left\{
    \begin{array}{rl}
    \Rank\cH(A_{\min}),& \text{if }\Rank\cH(A_{\min})=k,\\
    \Rank\cH(A_{\min})-1,& \text{if }\Rank\cH(A_{\min}=k+1.
    \end{array}
    \right.
    \end{align*}
    So the measure cannot contain less atoms than the one in $(t',u')$ above.\\
    
    Under the assumption
    Theorem $\ref{221023-1854}.\eqref{221023-1857-pt3.2.2}$
    we separate two cases:\\
    
    \noindent \textbf{Case 1: $\partial\mc R_1\cap \partial\mc R_2=\mc R_1\cap \mc R_2=\{(\tilde t,\tilde u)\}$.}
    Under the assumptions of this case,
    $\gamma(\tilde t,\tilde u)$ is $\RR$--representable.
    Hence, $(\tilde t,\tilde u)$ is a good choice for $(\tilde t_0,\tilde u_0)$ in \eqref{251023-1603-v2}.
    If $\Rank \cH(A_{\min})=\Rank H_{22}+2$, then a $\big(\Rank \widetilde{\mc M}(k;\beta)\big)$--atomic $\cZ(p)$--rm exists. This is due to the rank equality
    $$r:=\Rank \cF(\cG(\tilde t,\tilde u))+\Rank \cH(\cG(\tilde t,\tilde u))
    \underbrace{=}_{\substack{\eqref{rank-R1-parabolic}\\\eqref{rank-R2-parabolic}}}\Rank \cF(A_{\min})+1+\Rank H_{22}+1.$$
    Hence,
    $r=\Rank \widetilde{\mc M}(k;\beta)$ if and only if $\Rank \cH(A_{\min})=\Rank H_{22}+2$.
    Otherwise we have $\Rank \cH(A_{\min})=\Rank H_2=\Rank H_{22}+1$ (since $\eta>0$ and $H_2/H_{22}=0$) and 
    $r=\Rank \widetilde{\mc M}(k;\beta)+1$.\\

    \noindent \textbf{Case 2: $(\tilde t,\tilde u)=:\partial\mc R_1\cap \partial\mc R_2\subsetneq\mc R_1\cap \mc R_2$.}
    In this case it follows by Theorem \ref{221023-1854}
    that one of the points
    $(\tilde t,\tilde u)$ or $(t_0,\tilde u)$
    is a good choice for $(\tilde t_0,\tilde u_0)$ in \eqref{251023-1603-v2}.
    
    Assume that $(t_0,\tilde u)$ is a good choice.
    Since 
        $\Rank \cH(\cG(t_0,\tilde u))=\Rank \cH(A_{\min})-1$, which is due to $\eta>0$ and $H_2/H_{22}=0$ (if $H_2/H_{22}>0$, then  \eqref{forms-of-R1-R2}
        would imply that 
        $\Card \partial\mc R_1\cap \partial\mc R_2>1$),
    and 
        $\Rank \cF(\cG(t_0,\tilde u))=\Rank \cF(A_{\min})+2$
    (by Lemma \ref{Hyp3-claim1}),
    it follows that 
    a $\big(\Rank \widetilde{\mc M}(k;\beta)+1\big)$--atomic $\cZ(p)$--rm exists.

    As in the proof of Case 1 above,
    if $(\tilde t,\tilde u)$ is a good choice for $(\tilde t_0,\tilde u_0)$ in \eqref{251023-1603-v2}, then 
    a $\big(\Rank \widetilde{\mc M}(k;\beta)\big)$--atomic $\cZ(p)$--rm exists if and only if 
    $\Rank \cH(A_{\min})=\Rank H_{22}+2$.
    Otherwise the measure is 
    $\big(\Rank \widetilde{\mc M}(k;\beta)+1\big)$--atomic.

    It remains to show that if $(\tilde t,\tilde u)$ is not a
     good choice for $(\tilde t_0,\tilde u_0)$ in \eqref{251023-1603-v2}, there does not exist a 
    $\big(\Rank \widetilde{\mc M}(k;\beta)\big)$--atomic $\cZ(p)$--rm. By Lemma \ref{Hyp3-claim4}, the candidates for a good choice are points $(t,u_0)$ for $t\in (\tilde t,t_0]$. But as in the second paragraph above,
    we have 
    $\Rank \cH(\cG(t,\tilde u))\geq \Rank \cH(A_{\min})-1$
    and $\Rank \cF(\cG(t,\tilde u))=\Rank \cF(A_{\min})+2$
    for every such $t$. So 
    $$
    \Rank \cH(\cG(t,\tilde u))+\Rank \cF(\cG(t,\tilde u))\geq 
    \Rank \widetilde{\mc M}(k;\beta)+1.
    $$
    \smallskip
    
        It remains to establish the moreover part.
    Note that in the case where $\Rank \widetilde{\mc M}(k;\beta)+2$ atoms might be needed, $\cH(A_{\min})$ is not pd.
    Since for a $p$--pure sequence $\beta$
    with $\widetilde{\mc M}(k;\beta)\succeq 0$, \eqref{261023-1536} implies that $\cH(A_{\min})$ is pd, 
    the existence of a $\cZ(p)$--rm
    implies the existence of at most
    $(\Rank \widetilde{\mc M}(k;\beta)+1)$--atomic $\cZ(p)$--rm.
    This concludes the proof of Theorem 
    \ref{thm:hyperbolic-3-minimal-measures}.
\end{proof}

\subsection{Example}{\footnote{The \textit{Mathematica} file with numerical computations can be found on the link \url{https://github.com/ZalarA/TMP_cubic_reducible}.}} 
\label{ex:hyperbolic-type-3}
In this subsection we demonstrate 
the use of Theorems \ref{221023-1854} and 
    \ref{thm:hyperbolic-3-minimal-measures}
on a numerical example.

Let $\beta$ be a bivariate degree 6 sequence given by
\begin{spacing}{1.3}
$\beta_{00} = 1$,

$\beta_{10} =\frac{37}{54}$,
$\beta_{01}  = \frac{2}{3}$

$\beta_{20}  = \frac{769}{648},$
$\beta_{11} = \frac{25}{54},$
$\beta_{02}  =\frac{1201}{648},$

$\beta_{30}  =\frac{11737}{7776},$
$\beta_{21}  =\frac{337}{648}$
$\beta_{12}  =\frac{12025}{7776},$
$\beta_{03}  =\frac{913}{216},$

$\beta_{40}  =\frac{258721}{93312},$
$\beta_{31}  =\frac{4825}{7776},$
$\beta_{22}  =\frac{169153}{93312},$
$\beta_{13}  =\frac{9625}{2592},$
$\beta_{04}  =\frac{957985}{93312},$

$\beta_{50} =\frac{5088937}{1119744},$
$\beta_{41} =\frac{72097}{93312},$
$\beta_{32} =\frac{2497225}{1119744},$
$\beta_{23} =\frac{136801}{31104},$
$\beta_{14} =\frac{10813225}{1119744},$
$\beta_{05} =\frac{2326373}{93312},$

$\beta_{60}=\frac{115846129}{13436928},$
$\beta_{51}=\frac{1107625}{1119744},$
$\beta_{42}=\frac{38072593}{13436928},$
$\beta_{33}=\frac{2034025}{373248},$
$\beta_{24}=\frac{156268657}{13436928},$
$\beta_{15}=\frac{27728525}{1119744},$

\hfill $\beta_{06}=\frac{826264081}{13436928}.$
\end{spacing}
	We will prove below that $\beta$ 
    admits a $9$--atomic $\cZ(p)$--rm by applying Theorems \ref{221023-1854} and 
    \ref{thm:hyperbolic-3-minimal-measures}.
	It is easy to check that $\widetilde{\mc M}(3)$ is psd
	and satisfies only one column relation $
    2Y^2+X^2Y-Y^3=\mathbf{0}$.
        It turns out that $\eta=0$,
        $\Rank \cF(A_{\min})=\Rank \cF(A_{\min})_{\cB\setminus\{X^{k}\}}=
    \Rank \cF(A_{\min})_{\cB\setminus\{YX^{k-1}\}}=5$, whence 
    $A_{\min}$ satisfies ($\widetilde{\text{Hyp}}$)$_1$.
    By Theorem \ref{221023-1854}, $\beta$ has a $\cZ(p)$--rm.
    By Theorem 
    \ref{thm:hyperbolic-3-minimal-measures},
    there is a $\Rank \widetilde{M}(3)$--atomic $\cZ(p)$--rm (i.e., $9$--atomic).